\documentclass[12pt]{article}
\usepackage{amssymb}
\usepackage{amsfonts}
\usepackage{amsmath,latexsym,amssymb,amsfonts,amsbsy, amsthm}
\usepackage[usenames]{color}
\usepackage{xspace,colortbl}
\usepackage{mathrsfs}
\usepackage[colorlinks,linkcolor=blue,citecolor=blue]{hyperref}
\usepackage[titletoc]{appendix}
\allowdisplaybreaks

\newcommand{\beq}{\begin{equation}}
\newcommand{\R}{{\mathbb R}}
\newcommand{\eeq}{\end{equation}}
\newcommand{\ben}{\begin{eqnarray}}
\newcommand{\een}{\end{eqnarray}}
\newcommand{\beno}{\begin{eqnarray*}}
\newcommand{\eeno}{\end{eqnarray*}}

\newtheorem{thm}{Theorem}[section]

\newtheorem{lem}[thm]{Lemma}
\newtheorem{prop}[thm]{Proposition}
\newtheorem{coro}[thm]{Corollary}
\newtheorem{rmk}[thm]{Remark}

\renewcommand{\theequation}{\thesection.\arabic{equation}}

\title{\textbf{The structure of finite Morse index solutions to two free boundary
problems in $\R^2$}}
\author{Kelei Wang
\\
{\small   School of Mathematics and Statistics \& Computational
Science Hubei Key Laboratory}
\\
{\small  Wuhan University, Wuhan 430072, China}\\
{\small wangkelei@whu.edu.cn } \ }
\date{}

\begin{document}
\maketitle
\begin{abstract}
We give a description of the structure of finite Morse index
solutions to two free boundary problems in $\R^2$. These free
boundary problems are models of phase transition and they are
closely related to minimal hypersurfaces. We show that finite Morse
index solutions in $\R^2$ have finitely many ends and they converge
exponentially to these ends at infinity. As an important tool for
the proof, a quadratic decay estimate for the curvature of free
boundaries is established.
\end{abstract}


\vskip 0.2cm


\renewcommand{\theequation}{\thesection.\arabic{equation}}
\setcounter{equation}{0}

\tableofcontents

\section{Introduction}\label{Sec 1}
\numberwithin{equation}{section}
 \setcounter{equation}{0}

In nonlinear elliptic problems, the structure of stable or finite
Morse index solutions is always of great interest. For example, the
classification of stable minimal hypercones is related to Bernstein
problem and it plays an important role in the partial regularity
theory for minimal hypersurfaces. For sets with minimal perimeter
(i.e. minimal hypersurfaces which are {\em minimizers}) in $\R^n$,
the celebrated Bernstein theorem states that if $n\leq 7$, it must
be flat, i.e. an half space. For {\em stable} minimal hypersurfaces,
it has been long conjectured that the same conclusion should be true, although
only the dimension $3$ case has been proven ( M. do Carmo and C.-K. Peng
\cite{Carmo-Peng}, Fischer-Colbrie and R. Schoen \cite{F-Schoen}).
It turns out that this characterization of stable minimal surfaces
in $\R^3$ and various tools developed in its proof (e.g. interior
curvature estimates for stable minimal surfaces \cite{Schoen}) is
very helpful in the study of the structure of minimal surfaces in
$\R^3$, e.g. in the Colding-Minicozzi theory.

Although by now only a few results about stable minimal
hypersurfaces are known in high dimensions, it was proved by Cao-Shen-Zhu \cite{Cao-S-Z} that, for any $n$, a stable minimal hypersurface
in $\R^n$ has one end (in other words, it is connected at infinity).
For minimal hypersurfaces with finite Morse index, Li and Wang
\cite{Li-Wang} also show the finiteness of ends. In $\R^3$ the later
fact was known for a long time, because a minimal surface with
finite Morse index in $\R^3$ has finite total curvature
(Fishcer-Colbrie \cite{Fischer}), and a classical result of Osserman
says such a surface is conformal to a Riemannian surface with
finitely many points removed (corresponding to the ends), see for
example \cite[Section 2.3]{Meeks}.

In the realm of the Allen-Cahn equation,
\begin{equation}\label{Allen-Cahn}
\Delta u=u^3-u,
\end{equation}
 we face a similar situation. Due to
its close connection with minimal hypersurfaces, there is the De
Giorgi conjecture concerning
the one dimensional symmetry of entire solutions, which corresponds to the Bernstein theorem. Similar to the
minimal surface theory, for {\em minimizers} of \eqref{Allen-Cahn},
the De Giorgi conjecture has been proved by Savin \cite{Savin} (see
also the author \cite{Wang} for a new proof). For {\em stable}
solutions, at present the one dimensional symmetry is only known to
be true in dimension $2$ (due to an observation of Dancer) and the case
$3\leq n\leq 7$ remains open. Note that a proof of {\em stable De Giorgi conjecture}
in $\R^{n-1}$ will imply the original De Giorgi conjecture in
$\R^n$, see \cite{A-C}.

In view of the above mentioned results about minimal surfaces with
finite Morse index, and because {\em stable De Giorgi conjecture} is
known to be true in dimension $2$, it is conjectured that (see for
example \cite{DKW 3} and \cite{Gui 1})

{\it{\bf Conjecture } A finite Morse index solution of
\eqref{Allen-Cahn} in $\R^n$ has finitely many ends.}

Here an end of $u$ is defined to be an unbounded connected component
of $\{u=0\}$. In fact, this conjecture will imply that near
infinity, a finite Morse index solution is composed by a finite
number of stable solution (which is one dimensional by the {\em
stable De Giorgi conjecture}), patched together suitably.

Solutions with finite ends have been studied by many authors, see
\cite{Gui 1, DKP, KLP 1, KLP 2, KLP 3}. This conjecture has been proved by the author and Wei in \cite{Wang-Wei 2}.

In this paper, we study two free boundary problems related to the
Allen-Cahn equation and prove the above conjecture for these two
models.

\subsection{The first problem}\label{subsection 1.1}
The first problem is
\begin{equation}\label{equation 0.1}
 \left\{\begin{aligned}
&\Delta u=0,\ \ \ \mbox{in}\ \Omega:=\{-1<u<1\},\\
&u=\pm 1,\ \ \ \mbox{outside}\ \Omega,\\
&|\nabla u|=1,~~\mbox{on}~\partial\Omega.
                          \end{aligned} \right.
\end{equation}
Throughout this paper we only consider classical solutions, i.e.
$\Omega$ is assumed to be an open domain of $\R^2$ with smooth
boundary and $u\in C(\R^2)\cap C^2(\overline{\Omega})$, with the
equation and the boundary conditions in \eqref{equation 0.1}
satisfied pointwisely. Of course, in view of the regularity theory
for free boundaries in \cite{AC} and \cite{Weiss}, this hypothesis
can be relaxed a lot, but we will not pursue it here.

Equation \eqref{equation 0.1} arises as the Euler-Lagrange equation
of the functional
\begin{equation}\label{functional 0.1}
\int\left(|\nabla u|^2+\chi_{\{-1<u<1\}}\right).
\end{equation}
Here the potential energy $\chi_{\{-1<u<1\}}$ can be viewed as a double well
potential, although in a rather degenerate manner. 

The problem \eqref{equation 0.1} is
similar to the Allen-Cahn equation in many aspects. This model has been studied in \cite{CC} and \cite{Kam}. In
\cite{CC}, Caffarelli and C\'{o}rdoba put this problem in a
continuous family of phase transition models with double well
potentials and proved the uniform $C^{1,\alpha}$ regularity of
Lipschitz transition layers in the corresponding singular
perturbation problems. Compared to the Allen-Cahn equation, there are only minor technical differences. In \cite{Kam},
Kamburov developed the techniques introduced in \cite{DKW} which
deals with the Allen-Cahn equation, and gave a counterexample of the
De Giorgi conjecture for \eqref{equation 0.1} in $\R^n$, $n\geq9$.

Define the quadratic form
\[Q(\eta):=\int_\Omega|\nabla\eta|^2-\int_{\partial\Omega}H\eta^2,\quad \eta\in C_0^\infty(\R^2),\]
where $H$ is the mean curvature of $\partial\Omega$ with respect to
$\nu$, the unit normal vector of $\partial\Omega$ pointing to
$\Omega^c$. This is the second variation form of the functional
\eqref{functional 0.1} at $u$, see \cite{CJK} and
\cite{JS} for the derivation.

A solution $u$ is said to be of finite Morse index, if
\[\sup\mbox{dim}\{X: X \mbox{ subspace of } C_0^\infty(\R^2), Q\lfloor_{X}\leq 0\}<+\infty.\]

A standard argument shows that a finite Morse index solution is
stable outside a compact set, that is, there exists an $R_0>0$ such
that for any $\eta\in C_0^\infty(\R^2\setminus B_{R_0}(0))$,
\begin{equation}\label{stability inequality}
\int_\Omega|\nabla\eta|^2\geq\int_{\partial\Omega}H\eta^2.
\end{equation}

Our main result for this problem can be stated as follows:
\begin{thm}\label{main result 1}
Let $u$ be a solution of \eqref{equation 0.1} in $\R^2$. Assume
$\Omega$ to be connected. If $u$ is stable outside a compact set,
then
\begin{enumerate}
\item $u$ has the natural energy growth bound
\[\int_{B_R(0)}\left(\frac{1}{2}|\nabla u|^2+\chi_{\{-1<u<1\}}\right)\leq CR,\quad\forall R>0,\]
for some constant $C$ (depending on $u$);
\item the total curvature is finite,
\[\int_\Omega\left(|\nabla^2u|^2-|\nabla|\nabla u||^2\right)<+\infty.\]
\item for some $R>0$ large, there are only finitely many connected components of $\Omega\setminus
B_R$, which we denote by $D_i$, $1\leq i\leq N$ for some $N>0$;
\item $\Omega^c\setminus B_R$ consists only of finitely many unbounded connected
components;
\item each $D_i$ has the form
\[D_i=\{x: f_i^-(e_i\cdot x+b_i)\leq e_i^\bot\cdot x+a_i\leq f_i^+(e_i\cdot x+b_i)\},\]
where $e_i$ is a unit vector, $a_i$ and $b_i$ are constants, and
$f_i^{\pm}$ two smooth functions defined on $[R,+\infty)$, $f_i^+$
convex and $f_i^-$ concave, satisfying
\[f_i^-< f_i^+<f_i^-+4;\]
\item we have the balancing formula
\[\sum_{i=1}^N e_i=0;\]
\item for each $i$,
\[\lim_{t\to+\infty}f_i^\pm(t)=\pm1,\]
where the convergence rate is exponential.
\end{enumerate}
\end{thm}

An important result used in the proof of this theorem is the
following characterization of stable solutions.
\begin{thm}\label{stable solutions 1}
Let $u$ be a stable solution of \eqref{equation 0.1}. Then there
exists a unit vector $e$ such that $u(x)\equiv u(x\cdot e)$.
\end{thm}

In Theorem \ref{main result 1}, each end $D_i$ corresponds to a
stable solution. As in the Allen-Cahn equation \cite{GG} and
\cite{A-C}, the proof of Theorem \ref{stable solutions 1} uses the
Liouville theorem for the degenerate elliptic equation
\begin{equation}\label{degenerate equation}
\mbox{div}(\sigma^2\nabla \psi)=0.
\end{equation}

Another method to prove Theorem \ref{stable solutions 1} involves an equivalent formulation of the stability condition of
 Sternberg-Zumbrun  (see for
example \cite{S-Z} and \cite{FSV}),
\begin{equation}\label{S-Z inequality}
  \int |\nabla\eta|^2|\nabla u|^2\geq\int \eta^2|B|^2|\nabla u|^2.
\end{equation}
Here
\begin{equation}\label{Sternberg-Z}
|B|^2:=\frac{|\nabla^2u|^2-|\nabla|\nabla u||^2}{|\nabla
u|^2}=|A|^2+|\nabla_T\log|\nabla u||^2,
\end{equation}
where $A$ is the second
fundamental form of the level set (the curvature of the level set
because we are in $\R^2$), and $\nabla_T$ is the tangential derivative
along the level set.

 This
inequality is also used in this paper to establish a local integral
curvature bound. In turn, this curvature bound implies that at
infinity the solution is close to a one dimensional solution at
$O(1)$ scale.

In Theorem \ref{main result 1}, if the hypothesis on the
connectedness of $\Omega$ is removed, we have the following strong
half space theorem, which holds for any solution of \eqref{equation
0.1}, without any stability condition.
\begin{thm}\label{strong half space theorem}
Let $u$ be a solution of \eqref{equation 0.1} in $\R^2$
. If $\Omega$ is not connected, then $u$ is one dimensional.
\end{thm}
Note that we can single out the restriction of $u$ to a connected
component of $\Omega$ as a solution to \eqref{equation 0.1}. Hence
two components of $\Omega$ give two such solutions, $u_1$ and $u_2$,
satisfying $u_1\geq u_2$. This is similar to the situation met in
the strong half space theorem for minimal surfaces \cite{H-Meeks}.
Of course, compared to their proofs, the proof of our strong half
space theorem is rather direct. This is because we are in $\R^2$ and
thus the free boundary $\partial\Omega$ is convex (mean convex if we
are in higher dimension spaces).

\subsection{The second problem}\label{subsection 1.2}
 The second problem is a one phase free boundary problem,
\begin{equation}\label{equation 0.2}
 \left\{\begin{aligned}
&\Delta u=W^\prime(u),\ \ \ \mbox{in}\ \Omega:=\{u>0\},\\
&u=0,\ \ \ \mbox{outside}\ \Omega,\\
&|\nabla u|=\sqrt{2W(0)},~~\mbox{on}~\partial\Omega.
                          \end{aligned} \right .
\end{equation}
The solution is a critical point of the following functional
\begin{equation}\label{functional 0.2}
\int \left(\frac{1}{2}|\nabla u|^2+W(u)\chi_{\{u>0\}}\right).
\end{equation}
Here $W$ is a standard double well potential, that is, $W\in
C^2[0,+\infty)$ satisfying
\begin{itemize}
\item [{\bf W1)}]{\em $W\geq0$, $W(1)=0$ and $W>0$ in $[0,1)$;}
\item [{\bf W2)}] {\em $W^\prime\leq0$ on
$[0,1]$;}
\item [{\bf W3)}] {\em there exist two constants $\kappa>0$ and $\gamma\in(0,1)$ such that $W^{\prime\prime}\geq\kappa>0$
on $[\gamma,+\infty)$;}
\item [{\bf W4)}] {\em there exists a constant $p>1$, $W^\prime(u)\geq c(u-1)^p$ for
$u>1$.}
\end{itemize}
A typical example is $W(u)=(1-u^2)^2/4$ which gives the standard
Allen-Cahn nonlinearity.

The potential $W(u)\chi_{\{u>0\}}$ can still be viewed as a double
well potential, degenerate on the negative side. Hence this free
boundary problem still shares many similarities with the Allen-Cahn
equation.

The one phase free boundary problem, especially the partial
regularity theory for its free boundaries, has been studied for a
long time, see for example \cite{AC, Weiss}. This problem also
arises in the study of Serrin's overdetermined problem, see
\cite{Wang-Wei 1}, where a De Giorgi type conjecture was proved for
minimizers of \eqref{functional 0.2} in $\R^n$, $n\leq 7$.

The finite Morse index condition can be defined similarly as in
problem \eqref{equation 0.1}. This condition still implies that $u$
is stable outside a compact set, that is, there exists an $R_0>0$
such that, for any $\eta\in C_0^\infty(\R^2\setminus B_{R_0}(0))$,
\begin{equation}\label{stability inequality 2}
\int_\Omega\biggl(|\nabla\eta|^2+W^{\prime\prime}(u)\eta^2\biggr)\geq\int_{\partial\Omega}\left(-\frac{W^\prime(0)}{\sqrt{2W(0)}}+H\right)\eta^2.
\end{equation}
Here $H$ is the mean curvature of $\partial\Omega$ with respect to
$\nu$, the unit normal vector of $\partial\Omega$ pointing to
$\Omega^c$.

Our main result for this problem is similar to the first one.
\begin{thm}\label{main result 2}
Let $u$ be a solution of \eqref{equation 0.2} in $\R^2$.   Assume
$\Omega$ to be connected.  If $u$ is stable outside a compact set,
then
\begin{enumerate}
\item $u$ has the natural energy growth bound
\[\int_{B_R(0)}\left(\frac{1}{2}|\nabla u|^2+W(u)\chi_{\{u>0\}}\right)\leq CR,\quad\forall R>0,\]
for some constant $C$ (depending on $u$);
\item the total curvature is finite,
\[\int_\Omega\bigl(|\nabla^2u|^2-|\nabla|\nabla u||^2\bigr)<+\infty.\]
\item for any $R>0$ large, there are only finitely many connected components of $\Omega\setminus
B_R$;
\item for some $R_\ast>0$ large, $\Omega^c\setminus B_{R_\ast}$ consists only of finitely many unbounded connected
components, which we denote by $D_i$, $1\leq i\leq N$ for some
$N>0$;
\item each $D_i$ has the form
\[D_i=\{x: f_i^-(e_i\cdot x+b_i)\leq e_i^\bot\cdot x+a_i\leq f_i^+(e\cdot x+b_i),\]
where $e_i$ is a unit vector, $a_i$ and $b_i$ are constants, and
$f_i^{\pm}$ two smooth functions defined on $[R_\ast,+\infty)$, $f_i^+$
 concave and $f_i^-$ convex;
\item for each $i$, both the limits
\[\kappa_i^\pm:=\lim_{t\to+\infty}\frac{d f_i^\pm}{dt}(t)\]
exist. Moreover, by denoting $e_i^\pm$ the asymptotic direction of
the curve $\{e_i^\bot\cdot x+a_i=f_i^\pm(e\cdot x+b_i)\}$ at
infinity, we have the balancing formula
\[\sum_{i=1}^N \left(e_i^++e_i^-\right)=0;\]
\item if $i\neq j$, $\{e_i^+,e_i^-\}\cap\{e_j^+,e_j^-\}=\emptyset$;
\item the limits
\[\lim_{t\to+\infty}\left(f_i^\pm(t)-\kappa_i^\pm t\right)\]
exist, where the convergence rate is exponential.
\end{enumerate}
\end{thm}
Compared to Theorem \ref{main result 1}, the new feature is (7). This
is because in the positive part $\{u>0\}$, different ends are pushed
away, while in the first problem, a pair of $\partial\{u=1\}$ and
$\partial\{u=-1\}$ stay at finite distance. For these two problems,
outside $\Omega$, different ends could be at finite distance (for
example, we could have $\kappa_i^+=\kappa_i^-$ in Theorem \ref{main
result 2} (6)), because in this part there is no interaction between
different ends.

With this description, we can prove
\begin{coro}\label{coro 1}
Let $u$ be a solution of \eqref{equation 0.2} satisfying all of the
hypothesis in Theorem \ref{main result 2}. Assume furthermore that
$u$ has only two ends, then it is one dimensional.
\end{coro}
Since the number of ends is even, this implies that $u$ has at least
four ends, unless it is one dimensional.

As in the first problem, an important result used in the proof of
this theorem is the following characterization of stable solutions.
The proof is similar to the one for Theorem \ref{stable solutions
1}.
\begin{thm}\label{stable solutions 2}
Let $u$ be a stable solution of \eqref{equation 0.2}. Then there
exists a unit vector $e$ such that $u(x)\equiv u(x\cdot e)$.
\end{thm}

Similar to Theorem \ref{strong half space theorem}, we also have a
strong half space theorem for \eqref{equation 0.2}.
\begin{thm}\label{strong half space theorem 2}
Let $u$ be a solution of \eqref{equation 0.2} in $\R^2$
. If $\Omega$ is not connected, then $u$ is one dimensional.
\end{thm}
This can be proved by the same method as in Theorem \ref{strong half
space theorem}.

\subsection{Idea of the proof}\label{subsection 1.3}

Although these two free boundary problems look more complicated than
the Allen-Cahn equation, it turns out that the classification of
finite Morse index solutions to these two problems is a little
simpler. This is mainly due to the
convexity (mean convexity if the dimension larger than $2$) of free
boundaries in these two problems. This convexity is a consequence of
the Modica type inequality. Although it is also believed that the
Modica inequality in the Allen-Cahn equation gives a kind of (mean)
convexity, it seems not so easy to realize this. For a connection of
the Modica inequality with a kind of convexity in the Allen-Cahn
equation, see \cite{Smrynelis}.

This convexity provides us with the crucial Lipschitz regularity of
free boundaries at infinity. This Lipschitz regularity in turn, combined
with the stability of the solution near infinity and some
topological considerations, allows us to deduce the finiteness
of ends. For the first problem, this is fairly direct, because by
using the stability we can show that at infinity a pair of
$\partial\{u=1\}$ and $\partial\{u=-1\}$ stay at finite distance.
(This fact also implies that bounded connected components of
$\Omega^c$ stay in a fixed compact set.) For the second problem,
because different components of $\partial\Omega$ are expected to be
pushed away on the positive side, we have to make use of an idea of
Dancer in \cite{Dancer}, employing the stability condition to prove
that the number of nodal domains of solutions to the linearized
equation of \eqref{equation 0.2} is finite. This uses the Liouville
property for nonnegative subsolutions to the degenerate elliptic
equation \eqref{degenerate equation}. Of course, the fact that we are in dimension $2$
is crucial here. Using the convexity of free boundaries, this
finiteness information is transferred to $u$, which implies the
finiteness of unbounded components of $\Omega$.

In the second problem, the above argument does not give any
information on bounded components of $\Omega^c$. To show that
bounded components of $\Omega^c$ stay in a fixed compact set, we
turn to a quadratic curvature decay, which can be stated as
\begin{thm}\label{thm curvature decay}
Let $u$ be a solution of \eqref{equation 0.2} which is stable
outside a compact set. There exists a constant $C$  such that
\[H(x)\leq \frac{C}{1+|x|}, \ \ \ \ \forall\ x\in \partial\Omega.\]
\end{thm}

The formulation of Sternberg-Zumbrun on the stability condition for semilinear
equations is quite similar to the one for minimal hypersurfaces.
Hence it is natural to search for an inequality for the curvature type term $|B|^2$ (see \eqref{Sternberg-Z}), corresponding to the Simons inequality for the second
fundamental form of minimal hypersurfaces. However, as far as the
author knows, this goal has not been achieved yet.

Instead, in this paper we give a proof of Theorem \ref{thm curvature
decay} by a contradiction and blow up method. First, we use the blow
up method (more precisely, the doubling lemma of
Pol\'{a}cik-Quittner-Souplet \cite{Polacik-Q-S}) to reduce the
estimate to some uniform estimates in the corresponding singular
perturbation problem, that is:

{\bf Question (CD):} {\it Consider a solution $u_\varepsilon$ to the
problem
\begin{equation*}
 \left\{\begin{aligned}
&\varepsilon\Delta u_\varepsilon=\frac{1}{\varepsilon}W^\prime(u_\varepsilon),\ \ \ &\mbox{in}\ \{u_\varepsilon>0\}\cap B_1(0),\\
&|\nabla
u_\varepsilon|=\frac{1}{\varepsilon}\sqrt{2W(0)},~~&\mbox{on}~\partial\{u_\varepsilon>0\}\cap
B_1(0).
                          \end{aligned} \right .
\end{equation*}
Assume $\partial\{u_\varepsilon>0\}$ is uniformly bounded in $C^{1,1}$ norm. Can we
 deduce that the curvature of $\partial\{u_\varepsilon>0\}$ converges to $0$ uniformly?}

Due to the presence of free boundaries, under the hypothesis in
Question (CD), we can assume $\{u_\varepsilon>0\}$ have the
following forms:
\begin{itemize}
\item Case 1. (Multiplicity $1$) $\{x_2>f_\varepsilon(x_1)\}$ for a concave function
$f_\varepsilon$;
\item Case 2. (Multiplicity $2$)
$\{f_\varepsilon^-(x_1)<x_2<f_\varepsilon^+(x_1)\}$ for a concave
function $f_\varepsilon^-$ and a convex function $f_\varepsilon^+$.
\end{itemize}

To prove this type of uniform regularity, we use the method in \cite{Wang-Wei 2}. The main idea is to find the elliptic equation satisfied by $f_\varepsilon$ or $f_\varepsilon^\pm$. This needs a
 good approximate solution, which is almost the composition of the one dimensional solution $g$ with the distance to free boundaries. Therefore Fermi coordinates with respect to free boundaries are introduced. We also need to use the nondegeneracy condition on $g$ to get a good estimate on the error between $u_\varepsilon$ and the approximate solution. It turns out the approximate solution can be taken to be $g(d)$ in Case 1, while a perturbation  along the normal direction is needed in Case 2 to fulfill an orthogonal condition.

In Case 1, $f_\varepsilon$ satisfies the minimal surface equation with some remainder terms of higher order. Then second order estimate on $f_\varepsilon$ follows from standard elliptic estimates. The stability condition is not needed in this case. In Case 2, there is an interaction between $f_\varepsilon^\pm$, which is exactly the Liouville equation with some remainder terms of higher order. Here we revise the reduction method in \cite{Wang-Wei 2} to localize the interaction between different components of transition layers, i.e. it suffices to consider interactions between adjacent components. For the Allen-Cahn equation
\[\Delta u=W^\prime(u),\]
this localization technique allows us to consider the case when $W^{\prime\prime}(1)\neq W^{\prime\prime}(-1)$, which is not covered in the method in \cite{Wang-Wei 2}.

We also show that the stability condition of $u_\varepsilon$ induces a stability condition for this Liouville equation. By the fact that there is no entire stable solution of Liouville equation in dimension $1$, we get a second order estimate on $f_\varepsilon^\pm$.

Next, let us discus the exponential convergence of $u$ at infinity.
Once
we know that at infinity $u$ is close to a finite number of one
dimensional solutions patched together, by using the uniform
Lipschitz regularity of $\partial\Omega$, we can further show that
the convergence rate (to the one dimensional profile) is
exponential. This is mainly due to the following two facts: (i)
because we are in dimension $2$, the minimal hypersurfaces are just
straight lines, and hence there is no effect of the curvature; (ii)
the second eigenvalue for the linearized problem at the one
dimensional solution $g$ is positive. Note that the one dimensional
solution is stable. However, there is an eigenfunction with
eigenvalue $0$. Fortunately this eigenfunction is exactly
$g^\prime$, which comes from the translation invariance of the
problem. We would like to mention that the positivity of the second
eigenvalue can be viewed as a {\em nondegeneracy } condition,
because the first eigenvalue $0$ comes from the translation
invariance of the problem. This fact has been used a lot in the
construction of solutions to the Allen-Cahn equation, see for
example \cite{DKW 2}.

To prove the exponential convergence, we view the equation as an
evolution problem in the form
\[\frac{d^2 u}{dt^2}=\nabla\mathcal{J}(u),\]
where $\mathcal{J}$ is the corresponding functional defined on the
real line $\R$. Let $\mathcal{M}$ be the manifold of one dimensional
solutions. (This manifold is the real line $\R$, formed by
translations of a one dimensional solution $g$.) Take the nearest
point $P(u)$ on $\mathcal{M}$ to $u$. Then roughly speaking,
$u-P(u)$ almost lies in the subspace orthogonal to the first
eigenfunction of $P(u)$. By some more computations we get
\begin{equation}\label{main estimate}
\frac{d^2}{dt^2}\|u-P(u)\|^2\geq\mu\|u-P(u)\|^2+\mathcal{R},
\end{equation}
 where
$\mu$ is a positive constant (related to the second eigenvalue of
$g$). The norm $\|\cdot\|$ is usually taken to be a $L^2$ one. The
remainder term $\mathcal{R}$ is of the order $O(e^{-ct})$ for some
constant $c>0$. This then implies the exponential convergence of
$\|u-P(u)\|^2$, and the exponential convergence of $u(t)$ with some
more work.

This approach was used in Gui \cite{Gui}. In this paper we take a
related but different one. For the first free boundary problem, we
use the $L^2$ norm of $1-|\nabla u|$ to control the convergence
rate. This quantity is in fact equivalent to the $L^2$ norm of
$\frac{du}{dt}$. For the second free boundary problem, we take $P(u)$ to be
the one dimensional solution with the same free boundary point. This
turns out to be a rather good approximation of the nearest point, which is sufficient for our use.

Another approach to prove the exponential convergence  is presented
in del Pino-Kowalczyk-Pacard \cite{DKP}, using linear operator
theory in weighted Sobolev spaces. However, due to the presence of
free boundaries, it is not obvious to extend this method to our
setting. A second problem related to this linear theory is that, in
this paper we do not show the equivalence of the condition of
stability outside a compact set and the finite Morse index
condition. (The Schr\"{o}dinger operator case was proved in
\cite{dev}.) We also do not establish any relation between the Morse
index and the number of ends either. For related discussion about
the Morse index of minimal surfaces, see for examples \cite{Choe}
and \cite{Li-Wang 2}.

Finally, in a recent preprint of Jerison and Kamburov \cite{JK},
they also study the structure of solutions to one phase free
boundary problems in $\R^2$. Their study is more on the line of the Colding-Minicozzi
theory, that is, instead of the finite Morse index condition, they
use some assumptions on the topology of the set $\Omega=\{u>0\}$ and
from these assumptions the structure of solutions is obtained. For
Serrin's overdetermined problem, the topology and geometry of
solutions in $\R^2$ have also been studied in
\cite{pacard,ros,ros-r-s,tra}, which could be more complicated due
to the lack of a variational structure. Here we want to emphasize
that the stability condition is very strong and it provides much
better control than these topological conditions. One example is the
curvature estimates derived from the stability condition, which
gives us a clear picture of the solution at $O(1)$ scale and enables
us to obtain the convergence of translations of a fixed solution.

The paper is divided into three parts, the first two dealing with
the two free boundary problems respectively. The organization of the
first two parts are almost the same and can be read independently.
We first prove some uniform estimates for entire solutions of these
two free boundary problems such as the Modica inequality. Then we
prove the one dimensional symmetry of stable solutions and use this
and an integral curvature estimate to prove the finiteness of ends.
In the last step we prove the refined asymptotics at infinity.  In
Part II, more effort is needed to prove the finiteness of bounded
connected components of $\Omega^c$, the proof of which is postponed
to Part III due to its length. The whole Part III is devoted to the
proof of Theorem \ref{thm curvature decay}, which is also relatively
independent of the first two parts. Finally there is an appendix dealing with two nondegeneracy results used in this paper.

\part{The first problem}

In dimension $1$, the problem \eqref{equation 0.1} has a solution
$g$ defined by
 \makeatletter
\let\@@@alph\@alph
\def\@alph#1{\ifcase#1\or \or $'$\or $''$\fi}\makeatother
\begin{equation*}
{g(x)=}
\begin{cases}
1, &x\geq 1, \nonumber\\
x, &-1<x<1,\\
-1,&x\leq-1.\nonumber
\end{cases}
\end{equation*}
\makeatletter\let\@alph\@@@alph\makeatother

Of course, the trivial extension $u(x_1,x_2):=g(x_1)$ is a solution
of \eqref{equation 0.1}. Furthermore, for any $\cdots<t_i<t_i+2\leq
t_{i+1}<t_{i+1}+2\leq \cdots$, with $i\in I$ ($I$ finite or
countably infinite), the function
\begin{equation*}
v^\ast(x_1,x_2):=\left(-1\right)^ig(x_1-t_i) \quad\quad \mbox{if } \frac{t_{i-1}+t_i}{2}\leq x_1 \leq \frac{t_i+t_{i+1}}{2},
\end{equation*}
 is still a solution of
\eqref{equation 0.1}. All of these solutions are stable in $\R^2$. We
call them one dimensional solutions.

 Notation:
$\partial^{\pm}\Omega:=\partial\Omega\cap\{u=\pm1\}$.

\section{Uniform estimates}\label{sec 2}
\setcounter{equation}{0}

In this section we prove a Modica type inequality and establish the
convexity of free boundaries. Note that we do not need any stability
condition here and $u$ only denotes a solution to \eqref{equation
0.1}.

The main result in this section is
\begin{prop}\label{gradient bound}
In $\Omega$, $|\nabla u|\leq 1$.
\end{prop}
We would like to interpret this gradient bound as a Modica
inequality. As the following proof shows, this gradient bound holds
for entire solutions of \eqref{equation 0.1} in $\R^n$, for any
$n\geq 1$.

To prove Proposition \ref{gradient bound}, we need the following two
lemmas, which are basically consequences of the Hopf Lemma.

\begin{lem}\label{lem 2.2}
There exists a constant $d_A$ 
such that,
\[\mbox{dist}(x,\partial\Omega)\geq d_A\]
for any $x\in\{-3/4<u<3/4\}$.
\end{lem}
\begin{proof}
Take an arbitrary point $x\in\{-3/4<u<3/4\}$.
Assume $\mbox{dist}(x,\partial\Omega)=:h$ is attained at
$y\in\partial^-\Omega$. Then $B_h(x)$ is tangent to
$\partial^-\Omega$ at $y$.

The function $\tilde{u}:=1+u$ is non-negative, harmonic in $B_h(x)$,
satisfying $\tilde{u}(x)>1/4$ and $\tilde{u}(y)=0$. Hence by the
Hopf lemma, there exists a universal constant $c$ such that
\begin{equation}\label{using the Hopf lemma}
1=|\nabla \tilde{u}(y)|=\Big|\frac{y-x}{|y-x|}\cdot\nabla
\tilde{u}(y)\Big|\geq\frac{c\tilde{u}(x)}{h}\geq\frac{c}{h},
\end{equation}
which gives $h\geq 2c=:d_A$.
\end{proof}

\begin{lem}\label{lem 2.3}
There exists a universal constant $C$ such that $|\nabla u|\leq C$
in $\Omega$.
\end{lem}
\begin{proof}
For $x\in\Omega\setminus\{x:\mbox{dist}(x,\partial\Omega)>d_A/4\}$,
$B_{d_A/8}(x)\subset\Omega$. Applying the standard gradient estimate
for harmonic functions we deduce that
\[|\nabla u(x)|\leq \frac{C}{d_A}\sup_{B_{d_A/8}(x)}|u|\leq C.\]

Next, given $x\in\{x\in\Omega:\mbox{dist}(x,\partial\Omega)\leq
d_A/4\}$, let $y\in\partial\Omega$ attain
$\mbox{dist}(x,\partial\Omega)=:r$. Without loss of generality,
assume $y\in \partial^-\Omega$. Then $\tilde{u}:=1+u$ is positive,
harmonic in $B_r(x)$. Since $y\in\partial B_r(x)$ and
$\tilde{u}(y)=0$, as in \eqref{using the Hopf lemma}, we deduce that
\[\tilde{u}(x)\leq Cr.\]
Then by the Harnack inequality,
\[cr\leq \inf_{B_{r/2}(x)}\tilde{u}\leq\sup_{B_{r/2}(x)}\tilde{u}\leq Cr.\]
By the interior gradient estimate for harmonic functions,
\[|\nabla u(x)|=|\nabla \tilde{u}(x)|\leq C\frac{\mbox{osc}_{B_{r/2}(x)}\tilde{u}}{r}\leq C.\qedhere\]
\end{proof}
Now we come to the proof of Proposition \ref{gradient bound}.
\begin{proof}[Proof of Proposition \ref{gradient bound}]
Take $x_k\in\Omega$ so that $|\nabla u(x_k)|\to\sup_\Omega|\nabla
u|$, which we assume to be strictly larger than $1$. The proof is
divided into two cases.

{\bf Case 1.} $\mbox{dist}(x_k,\partial\Omega)$ does not converge to $0$.\\
Consider
\[u_k(x):=u(x_k+x).\]
Clearly $u_k$ is still a solution of \eqref{equation 0.1} in $\R^2$.
Since they are uniformly bounded in $\mbox{Lip}(\R^2)$, after passing to a subsequence $u_k$ converges to $u_\infty$ uniformly on any compact
set of $\R^2$.

The set $\Omega_\infty:=\{-1<u_\infty<1\}$ is open. It is clear that
$\Delta u_\infty=0$ in $\Omega_\infty$.

By the above construction, the origin $0\in\Omega_\infty$ and
\[\lambda:=|\nabla u_\infty(0)|=\sup_{\Omega_\infty}|\nabla u_\infty|>1.\]
Since $\Delta|\nabla u_\infty|^2=2|\nabla^2 u_\infty|^2\geq0$ in the
connected component of $\Omega_\infty$ containing $0$, by the strong
maximum principle, $|\nabla u_\infty|$ is constant and
$\nabla^2u_\infty\equiv 0$ in this component. Note that we still
have $|u_\infty|\leq 1$ in $\R^n$. Hence, after a rotation and a
translation, this component is $\{|x_1|<1/\lambda\}$ and in this
domain
\[
u_\infty(x)\equiv\lambda x_1.
\]
However, arguing as in the proof of \cite[Lemma 4.2]{JK} we can also
show that $\lambda\leq 1$. (Roughly speaking, this is because $u_k$
are classical solutions, hence their uniform limit $u_\infty$ is a
viscosity subsolution.) This is a contradiction and we finish the
proof in this case.

{\bf Case 2.} $h_k:=\mbox{dist}(x_k,\partial\Omega)\to0$.\\
Take a point $y_k\in\partial\Omega$ to attain this distance. Without
loss of generality, assume $y_k\in\partial^-\Omega$.

Let
\[v_k(x):=\frac{1}{h_k}\biggl[1+u_k(y_k+h_kx)\biggr].\]
By Lemma \ref{lem 2.2}, $u<-3/4$ in $B_{2h_k}(y_k)$. Thus in
$B_2(0)$, $v_k$ is nonnegative and harmonic in its positivity set.
Hence it is subharmonic in $B_2(0)$.

Denote $(x_k-y_k)/h_k$ by $z_k$. Since $|z_k|=1$, after passing to a subsequence it
converges to a limit point $z_\infty$.

As in Case 1, $v_k$ converges to a Lipschitz function $v_\infty$
uniformly on any compact set, and $\Delta v_\infty=0$ in
$\{v_\infty>0\}$.

Since $\Delta v_k=0$ and $v_k>0$ in $B_1(z_k)$, by noting that
$|\nabla v_k(z_k)|>1$, we obtain
\[\inf_{B_r(z_k)}v_k\geq c(r)>0,\quad \mbox{for }\forall r\in(0,1)  \mbox{ and  some constant } c(r),\]
where $c(r)$ is independent of $k$. Thus $\{v_\infty>0\}$ is
nonempty.

Still as in Case 1, we get
\[|\nabla v_\infty(z_\infty)|=\sup_{\{v_\infty>0\}}|\nabla v_\infty|>1.\]
Using the strong maximum principle, after suitable rotation and
translation, $v_\infty$ has the form $|\nabla v_\infty(z_\infty)|
x_2^+$. This leads to a contradiction with the free boundary
condition as in Case 1.
\end{proof}

\begin{prop}\label{strict convexity}
Every connected component of $\{u=1\}$ or $\{u=-1\}$ is convex.
Moreover, it is strictly convex unless $u$ is one dimensional.
\end{prop}
This follows from the calculation in \cite{CJK} (see also \cite{JK}
and \cite{JS}).

This convexity implies that
\begin{coro}
$\Omega$ is unbounded.
\end{coro}
\begin{proof}
Assume $\Omega\subset B_R(0)$ for some $R>0$. Then the connected
component of $\Omega^c$ containing $\R^2\setminus B_R(0)$ is the
whole $\R^2$, because it is convex. This is a contradiction.
\end{proof}

Next we give a proof of Theorem \ref{strong half space theorem}.
Thus assume $\Omega$ is not connected and take two different
connected components of $\Omega$, $\Omega_i$, $i=1,2$.
 For each $i=1,2$,
define $u_i$ to be the restriction of $u$ to $\Omega_i$, with the
obvious extension to $\Omega^c_i$. Then $u_i$ are still
solutions of \eqref{equation 0.1}.

Since $\Omega_2$ is connected, it is contained in a connected
component of $\Omega^c_1$, say $D$. Because $D$ is convex, it is
contained in an half space $H$, say $\{x_1>0\}$. In this setting, we
have the following {\em weak half space theorem}.
\begin{thm}\label{half space theorem}
Let $v$ be a solution of \eqref{equation 0.1}. Suppose
$\{-1<v<1\}$ is contained in an half space. Then $v$ is one
dimensional.
\end{thm}
\begin{proof}
As before, we can assume $\{-1<v<1\}$ to be connected and it is
contained in $\{x_1>0\}$.

Denote by $D$ the connected component of $\R^2\setminus\{-1<v<1\}$
containing $\{x_1<0\}$. Since $D$ is convex, it can be directly
checked that $\partial D$ is a graph in the form $\{x_1=f(x_2)\}$.
By definition, $f$ is a nonnegative concave function, hence a
constant. Then Proposition \ref{strict convexity} implies that $g$
is one dimensional.
\end{proof}
Theorem \ref{strong half space theorem} follows from this theorem.

For applications below, we present a non-degeneracy result for the
set $\{u=\pm 1\}$.
\begin{prop}\label{nondegenarcy of contact set 1}
If $u\neq v^\ast$ (with a unit vector $e$ and $\cdots<t_i\leq
t_{i+1}\leq \cdots$, where there are two constants $t_i=t_{i+1}$),
then for every $x\in\partial^\pm\Omega$ and $r>0$,
$|B_r(x)\setminus\Omega|>0$.
\end{prop}
\begin{proof}
Assume there is an $x_0\in\partial\Omega$ and $r_0>0$, such that
$|B_{r_0}(x_0)\setminus\Omega|=0$. By the convexity of
$\partial\Omega$, $\partial\Omega\cap B_{r_0}(x_0)$ are straight
lines. By Lemma \ref{strict convexity}, in $B_{r_0}(x_0)$,
$u=v^{\ast}$
for a unit vector $e$ and two constants $t_1=t_2$. 
 Then by Theorem \ref{strong half space theorem} and the unique continuation principle applied to $u$, this holds everywhere in $\R^2$.
\end{proof}
In the following, we will always assume $u$ satisfies this
non-degeneracy condition.

\section{The stable De Giorgi conjecture}\label{sec 3}
 \setcounter{equation}{0}

In this section we prove the stable De Giorgi conjecture. Then we
use the stability to derive an integral curvature bound and use this
to study the convergence of translations at infinity of a solution
$u$ to \eqref{equation 0.1}, if it is stable outside a
compact set.

Let us first prove Theorem \ref{stable solutions 1}.
\begin{proof}[Proof of Theorem \ref{stable solutions 1}]
Similar to \cite[Section 2.3]{JS}, the stability condition implies
the existence of a positive function $\varphi\in
C^\infty(\overline{\Omega})$, satisfying
\begin{equation*}
 \left\{\begin{aligned}
&\Delta \varphi=0,\ \ \ \mbox{in}\ \Omega,\\
&\varphi_\nu=-H\varphi,~~\mbox{on}~\partial\Omega.
                          \end{aligned} \right.
\end{equation*}
By a direct differentiation, for any unit vector $e$, the directional
derivative $u_e:=e\cdot\nabla u$ also satisfies this equation.

Let $\psi:=u_e/\varphi$. It satisfies
\begin{equation}\label{degenerate equation with boundary condition}
 \left\{\begin{aligned}
&\mbox{div}\left(\varphi^2\nabla\psi\right)=0,\ \ \ \mbox{in}\ \Omega,\\
&\psi_\nu=0,~~\mbox{on}~\partial\Omega.
                          \end{aligned} \right.
\end{equation}
For any $\eta\in C_0^\infty(\R^2)$, testing this equation with
$\psi\eta^2$ and integrating by parts on $\Omega$, we obtain
\[\int_\Omega\biggl(\varphi^2|\nabla\psi|^2\eta^2+2\varphi^2\eta\psi\nabla\eta\nabla\psi\biggr)=0.\]
By the Cauchy inequality,
\[\int_\Omega\varphi^2|\nabla\psi|^2\eta^2\leq 8\int_\Omega\varphi^2\psi^2|\nabla\eta|^2.\]
Then we can use standard log cut-off test functions to show that
\[\int_\Omega\varphi^2|\nabla\psi|^2=0.\]
As in \cite{GG} or \cite{A-C}, this implies the one dimensional
symmetry of $u$.
\end{proof}

In the following part of this section, we assume $u$ to be a
solution of \eqref{equation 0.1}, and it is stable outside a compact set of
$\R^2$.
\begin{lem}\label{lem 3.1}
For any $L>1$, there exists an $R(L)$ such that, there is no bounded
component of $\{u=\pm1\}$ contained in $B_{R(L)}(0)^c$ with diameter
smaller than $L$.
\end{lem}
\begin{proof}
{\bf Step 1.} For any $R$ large, take a $\varphi\in
C_0^\infty(\R^2\setminus B_R(0))$ and test the stability condition with
$\varphi|\nabla u|$. After some integration by parts, we obtain
\begin{eqnarray}\label{3.0.1}
\int_{\partial\Omega}H\varphi^2
&\leq&\int_\Omega\biggl(|\nabla\varphi|^2|\nabla u|^2+2\varphi|\nabla
u|\nabla\varphi\cdot\nabla |\nabla
u|+\varphi^2|\nabla|\nabla u||^2\biggr) \nonumber\\
&=&\int_{\partial\Omega}\frac{1}{2}\varphi^2\left(|\nabla
u|^2\right)_\nu+\int_{\Omega}\bigl(|\nabla\varphi|^2|\nabla
u|^2-\varphi^2|B|^2|\nabla u|^2\bigr).
\end{eqnarray}

On $\partial\Omega$,
\[\left(|\nabla u|^2\right)_\nu=-2u_{\nu\nu}=
2H.\] Hence \eqref{3.0.1} is transformed to
\[\int_{\Omega}|\nabla
u|^2|B|^2\varphi^2\leq\int_\Omega|\nabla u|^2|\nabla\varphi|^2.\]
Then we can use standard log cut-off test functions to show that
\[\int_{\Omega\setminus B_R}|\nabla
u|^2|B|^2\leq \frac{C}{\log R},\] which converges to $0$ as
$R\to+\infty$.

{\bf Step 2.} For any $\eta\in C_0^\infty(B_R(0)^c)$,
\begin{eqnarray}\label{3.0.2}
\int_{\partial\Omega}H\eta&=&\int_{\partial\Omega}\left(\frac{|\nabla
u|^2}{2}\right)_\nu\eta\nonumber
\\
&=&\int_{\Omega}\biggl(\nabla \frac{|\nabla u|^2}{2}\nabla \eta+\Delta
\frac{|\nabla
u|^2}{2}\eta\biggr)\\
&=&\int_{\Omega}\biggl(\nabla^2u\left(\nabla u,\nabla\eta\right)-\eta|\nabla^2
u|^2\biggr).\nonumber
\end{eqnarray}

{\bf Claim.} In $\Omega$,
\begin{equation}\label{3.0.3}
|\nabla^2u\cdot\nabla u|^2\leq |\nabla u|^2|B|^2,
\end{equation}
 and
 \begin{equation}\label{3.0.4}
|\nabla^2 u|^2=2|\nabla u|^2|B|^2.
\end{equation}
 This can be proved by writing
these quantities in  the coordinate form and using the equation
$\Delta u=0$.

Assume there is a connected component of $\{u=0\}$, $D$, contained
in $B_R(0)^c$ with its diameter smaller than $L$. Take a point $x$
in this component and $\eta$ to be a standard cut-off function in
$B_{2L}(x)$ with $\eta\equiv 1$ in $B_L(x)$. Substituting this into
\eqref{3.0.2} and using \eqref{3.0.3}, \eqref{3.0.4}, by noting that
$H\geq 0$ on $\partial\Omega$, we obtain
\begin{eqnarray*}
\int_{\partial D}H&\leq& \frac{C}{L}\int_{B_{2L(x)}}|\nabla
u||B|+C\int_{B_{2L(x)}}|\nabla
u|^2|B|^2\\
&\leq& C\int_{B_{2L(x)}}|\nabla
u|^2|B|^2+C\left(\int_{B_{2L(x)}}|\nabla
u|^2|B|^2\right)^{\frac{1}{2}}\\
&\leq& \frac{C}{\sqrt{\log R}}.
\end{eqnarray*}

On the other hand, because $D$ is convex, the Gauss-Bonnet theorem
says
\[\int_{\partial D}H=2\pi.\]
Thus we get a contradiction if
\[\frac{C}{\sqrt{\log R}}<2\pi.\qedhere\]
\end{proof}
As a corollary, we have
\begin{coro}
For any $x\in\partial\Omega\setminus B_{R(L)}(0)$ and $r\in(0,L/2)$,
the connected component of $\partial\Omega\cap B_r(x)$ passing
through $x$, denoted by $\Gamma^{x,r}$, has its boundary in
$\partial B_r(x)$. Hence,
\[\mathcal{H}^1(\Gamma^{x,r})\geq 2r.\]
\end{coro}

Moreover, the proof of Lemma \ref{lem 3.1} also implies that
\begin{coro}\label{integral curvature estimate}
For any $\varepsilon>0$ small and $L>0$ large, there exists an
$R(L,\varepsilon)$ so that the following holds. For any
$x\in\partial\Omega\setminus B_{R(L,\varepsilon)}(0)$, the connected
component of $\partial\Omega$ passing through $x$, $\Gamma^{x,L}$,
satisfies
\[\int_{\Gamma^{x,L}\cap B_L(x)}H\leq \varepsilon,\]
and
\[\mbox{dist}_{H}\left(\Gamma^{x,L}\cap B_L(x), \{e^{x,L}\cdot (y-x)=0\}\cap B_L(x)\right)\leq \varepsilon,\]
where $e^{x,L}$ is a unit vector.
\end{coro}

Next we claim that
\begin{lem}\label{limit at infinity}
 For any $\ell\geq 1$ and
$x_k\in\Omega$, $|x_k|\to\infty$, the translation function
\[u_k(x):=u(x_k+x)\]
converges in the $C^\ell$ sense to $v^\ast(e\cdot x)$ for some unit
vector $e$ and a sequence of
\[\cdots<t_i<t_i+2\leq t_{i+1}<t_{i+1}+2\leq \cdots,\quad i\in I.\]
Moreover, the translations of $\Omega$, $\Omega_k:=\Omega-x_k$
converges to $\{-1<v^\ast(e\cdot x)<1\}$ in the $C^\ell$ sense on
any compact set.
\end{lem}
\begin{rmk}
In the above, we say $\Omega_k$ converges in the $C^\ell$ sense, if there exists a fixed vector $e$ such that
$\partial\Omega_k$ can be represented by the graph of a family of
functions defined on $\{e\cdot x=0\}$, with these functions converge
to $\{e\cdot x=t_i\}$ respectively in $C^\ell$. Note that this
implies, for $k$ large, $\Omega_k$ is $C^\ell$ diffeomorphic to
$\{-1<v^\ast(e\cdot x)<1\}$. We say $u_k$ converges to
$v^\ast(e\cdot x)$ in the $C^\ell$ sense, if the pull back of $u_k$
through the above diffeomorphism converges to $v^\ast(e\cdot x)$ in
$C^\ell_{loc}(\overline{\{-1<v^\ast(e\cdot x)<1\}})$.
\end{rmk}
\begin{proof}[Proof of Lemma \ref{limit at infinity}]
First assume $u_k$ converges to a limit function $u_\infty$
uniformly on any compact set of $\R^2$. Of course $\Delta
u_\infty=0$ in $\{-1<u_\infty<1\}$.

By the previous corollary, $\partial\{-1<u_k<1\}$ converges in the
Hausdorff distance to a family of lines, which we assume to be
parallel to the $x_1$-axis.

For any $k$ large, take a connected component $D_k$ of
$\{-1<u_k<1\}$. 
 By the previous
analysis, $\partial D_k$ consists of two convex curves
$\Gamma_{1,k}$ and $\Gamma_{2,k}$. Moreover, there exist two
constants $t_{1,k}$ and $t_{2,k}$ such that,
\begin{equation}\label{3.3}
\lim_{k\to+\infty}\mbox{dist}_H(\Gamma_{i,k},
\{x_2=t_{i,k}\})=0,\quad i=1,2.
\end{equation}
After a translation in the $x_2$ direction, we can assume
$\overline{D_k}$ converges to $\overline{D_\infty}=\{0\leq x_2\leq
t_\infty\}$ in the Hausdorff distance, where
$$t_\infty=\lim_{k\to+\infty}\left(t_{2,k}-t_{1,,k}\right)\in [0,+\infty].$$

Next we divide the proof into two cases.\\
{\bf Case 1.} Assume $u_k=-1$ on $\Gamma_{1,k}$ and $u_k=1$ on
$\Gamma_{2,k}$.

By Lemma \ref{lem 2.2}, $|t_{2,k}-t_{1,k}|\geq 2d_A$. This implies
\begin{equation}\label{3.4}
t_\infty\geq 2d_A.
\end{equation}

 Take an
arbitrary point $x_k\in\Gamma_{1,k}$. The following function is well
defined in $B_{d_A}(x_k)$:
 \makeatletter
\let\@@@alph\@alph
\def\@alph#1{\ifcase#1\or \or $'$\or $''$\fi}\makeatother
\begin{equation*}
{v_k(x):=}
\begin{cases}
0, &x\in B_{d_A}(0)\setminus \left(D_k-x_k\right), \\
u_k+1, &x\in D_k-x_k.
\end{cases}
\end{equation*}
\makeatletter\let\@alph\@@@alph\makeatother Indeed,
$(\Gamma_{1,k}-x_k)\cap B_{d_A}(0)$ is a convex curve with boundary
points in $\partial B_{d_A}(0)$, hence it divides $B_{d_A}(0)$ into
two connected open sets, one is $D_k-x_k$ where $v_k>0$ and the
other one is $B_{d_A}(0)\setminus (D_k-x_k)$ where $v_k=0$.

In $B_{d_A}(0)$, $v_k$ is a classical solution of the one phase free
boundary problem
\begin{equation}\label{one phase free boundary problem}
 \left\{\begin{aligned}
&\Delta v_k=0,\ \ \ \mbox{in}\ \{v_k>0\},\\
&|\nabla v_k|=1,~~\mbox{on}~\partial\{v_k>0\}.
                          \end{aligned} \right.
\end{equation}
 \eqref{3.3}
implies that $\partial\{v_k>0\}$ is flat in the sense of \cite{AC}.
Hence by the regularity theory for free boundaries in \cite{AC} and
the higher regularity of free boundaries in \cite{KN},
$\partial\{v_k>0\}\cap B_{d_A/2}(0)$ can be represented by the graph
of a function $f_k$ defined on the $x_1$-axis, with its $C^\ell$
norm uniformly bounded for any $\ell\geq 1$.

This implies, for any $\ell\geq 1$, $\partial D_k$ converges to
$\partial D_\infty$ in $C^\ell$ on any compact set. Then by standard
elliptic estimates, $u_k$ are uniformly bounded in
$C_{loc}^\ell(\overline{D_k})$ for any $\ell\geq 1$, and they
converge to $u_\infty$ in the $C^{\ell}$ sense.

Now $u_\infty$ satisfies
\begin{equation*}
 \left\{\begin{aligned}
&\Delta u_\infty=0,\ \ \ \mbox{in}\ \{0<x_2<t_\infty\},\\
&-1<u_\infty<1,\ \ \ \mbox{in}\ \{0<x_2<t_\infty\},\\
&u_\infty=-1,\ \ \ \mbox{on}\ \{x_2=0\},\\
&u_\infty=1,\ \ \ \ \mbox{on}\ \{x_2=t_\infty\},\\
 &|\nabla v_\infty|=1,~~\mbox{on}~\{x_2=0\}\cup\{x_2=t_\infty\}.
                          \end{aligned} \right.
\end{equation*}
In the above, if $t_\infty=+\infty$, it is understood that the boundary
condition on $\{x_2=t_\infty\}$ is void.

We claim that $u_\infty=-1+x_2$ in $D_\infty$, and hence
$t_\infty=2$. This can be proved by many methods. A direct way is by
noting that we have $|\nabla u_\infty|\leq 1$ in $D_\infty$
(obtained by passing to the limit in $|\nabla u_k|\leq 1$), hence we
can use Proposition \ref{strict convexity} to deduce that $\nabla^2
u_\infty\equiv0$ in $D_\infty$ and the claim follows.

{\bf Case 2.} Assume $u_k=-1$ on $\partial D_k$.

As in Case 1, the following function is well defined:
 \makeatletter
\let\@@@alph\@alph
\def\@alph#1{\ifcase#1\or \or $'$\or $''$\fi}\makeatother
\begin{equation*}
{v_k(x):=}
\begin{cases}
0, &\mbox{outside } D_k, \\
u_k+1, &\mbox{in } D_k.
\end{cases}
\end{equation*}
\makeatletter\let\@alph\@@@alph\makeatother
 $v_k$ is a classical
solution of the one phase free boundary problem \eqref{one phase
free boundary problem}.

This  case can be further divided into two subcases.\\
{\bf Subcase 1.} $\lim_{k\to0}|t_{k,2}-t_{k,1}|=0$.

Because $|\nabla v_k|\leq 1$ and $v_k=0$ on $\partial\{v_k>0\}$, we
have
\[\sup_{D_k}v_k\to0.\]
Take a standard cut-off function $\eta\in C_0^\infty(B_2(0))$ with
$\eta=1$ in $B_1(0)$. Then by \eqref{one phase free boundary
problem}, integrating by parts gives
\[
\int_{\partial D_k}\eta =-\int_{D_k}v_k\Delta\eta
\leq\left(\sup_{D_k}v_k\right)\int_{D_k}\big|\Delta\eta\big|\to0.
\]
On the other hand, because $\partial D_k$ are convex curves
satisfying \eqref{3.3},
\[\lim_{k\to+\infty}\mathcal{H}^1(\partial D_k\cap B_1(0))=4.\]
This is a contradiction.

{\bf Subcase 2.} $\lim_{k\to0}|t_{k,2}-t_{k,1}|\in(0,+\infty]$.

As in Case 1, $u_\infty$ satisfies
\begin{equation*}
 \left\{\begin{aligned}
&\Delta u_\infty=0,\ \ \ \mbox{in}\ \{0<x_2<t_\infty\},\\
&-1<u_\infty<1,\ \ \ \mbox{in}\ \{0<x_2<t_\infty\},\\
&u_\infty=-1,\ \ \ \mbox{on}\ \{x_2=0\}\cup\{x_2=t_\infty\},\\
 &|\nabla v_\infty|=1,~~\mbox{on}~\{x_2=0\}\cup\{x_2=t_\infty\}.
                          \end{aligned} \right.
\end{equation*}
In the above, if $t_\infty=+\infty$, it is understood that the boundary
condition on $\{x_2=t_\infty\}$ is void.

Still as in Case 1, by applying Proposition \ref{strict convexity},
$\nabla^2u_\infty=0$ in $D_\infty$, which then leads to a
contradiction. Thus this case is impossible.
\end{proof}

\section{Finiteness of ends}\label{sec 4}
\setcounter{equation}{0}

In this section, $u$ still denotes a solution of \eqref{equation
0.1}, which is stable outside a compact set of $\R^2$. By the results
established in the previous section, we can take a $K\gg 1$ and then
$R_1\gg 1$ so that, for any $x\in \Omega\setminus B_{R_1}$,
\begin{enumerate}
\item[(F1)] $|\nabla u|\geq 1-1/K$ in $\Omega\cap B_K(x)$;
\item[(F2)] $\Omega\cap B_K(x)=\cup_i \Upsilon_i$, where
\[\Upsilon_i=\{y: f_i^-(y\cdot e)< y\cdot e^{\bot}< f^+_i(y\cdot e)\},\]
with $e$ a unit vector and
\[\cdots<f_{i-1}^+<f_i^-<f_i^+<\cdots;\]
\item[(F3)] $f_i^{\pm}$ are defined on $(-K,K)$, with their $C^\ell$ norm
bounded for any $\ell\geq 1$, uniformly in $x$;
\item[(F4)] each $f_i^-$ is concave and each $f^+_i$ is convex;
\item[(F5)] $2-1/K\leq f_i^+-f_i^-\leq 2+1/K$ in $(-K,K)$.
\end{enumerate}

With these preliminaries we prove
\begin{lem}\label{lem 4.3}
Each connected component of $\Omega^c$ intersects $B_{R_1+2K}(0)$.
\end{lem}
\begin{proof}
Let $D$ be an arbitrary connected component of $\Omega^c$. Assume it
does not intersect $B_{R_1+2K}(0)$. Without loss of generality,
assume $u=1$ in $D$.

Since $D$ is a convex set, $\partial D$ is a simple curve or two
disjoint simple curves. (These curves could be unbounded.) If the
latter case happens, $\R^2\setminus D$ has two connected components.
Then the component of $\R^2\setminus D$ which does not intersect
$B_{R_1+2K}(0)$ contains a component of $\Omega$ and it does not
intersect $\Omega$. This is clearly a contradiction with our
assumption that $\Omega$ is connected. Hence $\partial D$ is a
single simple curve.

By our preliminary analysis, for each $x\in
\partial D$, there exists a unit vector $e(x)$ such that,
\[\partial D\cap B_K(x)=\{y: y\cdot e^{\bot}=f^+(y\cdot e)\}.\]
Moreover, there exists a curve
\[\Gamma=\{y: y\cdot e^{\bot}=f^-(y\cdot e)\},\]
where $2-1/K\leq f^+-f^-\leq 2+1/K$, such that
\[\{y\in B_K(x):f^-(y\cdot e)<y\cdot e^{\bot}<f^+(y\cdot e)\}\subset\Omega.\]
In fact, this curves can be represented in the form
\begin{equation}\label{2.5}
y-t(y)\nu(y),
 \end{equation}
 where the function
$t(y)$ is defined on $\partial D\cap B_K(x)$ and $\nu(y)$ is the
unit normal of $\partial D$ at $y$, pointing to $\Omega$.

By abusing notations, we denote the connected component of
$\partial\Omega$ where $\Gamma$ lies on, still by $\Gamma$. Note
that $\Gamma$ is also a simple smooth curve.

By continuation, any point $x\in\Gamma$ satisfies
$\mbox{dist}(x,\partial D)\leq 2+1/K$. Hence $\Gamma$ dose not
intersect $B_{R_1+K}(0)$. $\Gamma$ can still be represented by the
graph of a function defined on $\partial D$ as in \eqref{2.5}.
$\Gamma$ and $\partial D$ bounds a connected component of $\Omega$,
which does not intersect $B_{R_1+K}(0)$ either. However, this is a
contradiction with our hypothesis that $\Omega$ is connected and the
proof is thus completed.
\end{proof}
This lemma can be reformulated as follows: every connected component
of $\partial\Omega$ intersects $B_{R_1+2K}(0)$. Together with the
facts (F1-5), this implies that there are only finitely many
unbounded components of $\partial\Omega$. Because the above proof
also implies that, for any unbounded component of $\Omega^c$, its
boundary is a simple smooth curve, we obtain
\begin{coro}
There are only finitely many unbounded connected components of
$\Omega^c$.
\end{coro}

Checking the proof of Lemma \ref{lem 4.3}, we also obtain
\begin{coro}
There is no bounded connected component of $\Omega^c$ belonging to
$\R^2\setminus B_{R_1+2K}(0)$.
\end{coro}
This implies that $\partial\Omega\setminus B_{R_1+2K}(0)$ is
composed by finitely many unbounded simple curves. Hence we have
\begin{coro}
There are only finitely many connected components of
$\Omega\setminus B_{R_1+2K}(0)$.
\end{coro}

Putting all of the above facts together we get the following
picture.
\begin{lem}\label{lem 3.6}
There exists a constant $R_2>R_1+2K$ so that the following holds.
Let $\Omega\setminus B_{R_2}(0)=\cup_{i=1}^N D_i$, where each $D_i$
is connected. For every $i$, there exist a unit vector $e_i$, two
constants $a_i$ and $b_i$, and two bounded functions $f_i^{\pm}$
defined on $[R_2,+\infty)$, $f_i^+$ convex and $f_i^-$ concave,
\[f_i^-\leq f_i^+\leq f_i^-+2+1/K,\]
such that
\[D_i=\{x: f_i^-(e_i\cdot x+b_i)\leq e_i^\bot\cdot x+a_i\leq f_i^+(e\cdot x+b_i).\]
\end{lem}
Note that we do not claim $e_i$ to be different for different $i$.

Now we establish the natural energy growth bound.
\begin{lem}\label{energy growth}
There exists a constant $C$ depending on $u$ such that, for any
$R>1$,
\[\int_{B_R(0)}\left(\frac{1}{2}|\nabla u|^2+\chi_\Omega\right)\leq CR.\]
\end{lem}
\begin{proof}
In view of Proposition \ref{gradient bound}, we only need to prove
\begin{equation}\label{4.1.1}
|\Omega\cap B_R(0)|\leq CR.
\end{equation}
This follows directly from the previous lemma.
\end{proof}

For each $\varepsilon>0$, let
\[u_\varepsilon(x):=u(\varepsilon^{-1}x).\]
\begin{prop}\label{balancing condition}
As $\varepsilon\to0$,
\[\varepsilon|\nabla u_\varepsilon|^2dx\rightharpoonup2\sum_{i=1}^N\mathcal{H}^1\lfloor_{\{re_i:r\geq0\}},
\quad\mbox{weakly as Radon measures},\]
where $e_i$ are as in Lemma
\ref{lem 3.6}. Moreover,
\[\sum_{i=1}^Ne_i=0.\]
\end{prop}
The proof in \cite{H-T} (see also \cite{Wang-Wei 1}) can be adapted to
prove this proposition. Note that the blowing down limit is unique,
i.e. independent of subsequences of $\varepsilon\to0$. In fact, by
the convexity of $\Omega^c$, the blowing down limit (in the
Hausdorff distance)
\[D_i^\infty:=\lim_{\varepsilon\to0}\varepsilon D_i=\{re_i: r\geq0\}.\]
Moreover, the blowing down limit of $(|\nabla u|^2+1)\chi_{D_i}dx$
is $2\mathcal{H}^1\lfloor_{\{re_i:r\geq0\}}$.

\section{Refined asymptotics}\label{sec 5}
\setcounter{equation}{0}

In this section we prove the exponential convergence of $u$ to its
ends (one dimensional solutions) at infinity.

 Take a large $R$ and a connected component
of $\Omega\setminus B_R$, which we assume to be
\[\mathcal{C}:=\{(x_1,x_2): f_-(x_1)<x_2<f_+(x_1),\ x_1>R\},\]
where $f_\pm$ are convex (concave, respectively) functions defined
on $[R,+\infty)$.

By Lemma \ref{limit at infinity},
\begin{equation}\label{7.1}
\lim_{x_1\to+\infty}\left(f_+(x_1)-f_-(x_1)\right)=2.
\end{equation}
Then because $f_+^\prime(x_1)$ is non-increasing in $x_1$ and
$f_-^\prime(x_1)$ non-decreasing in $x_1$, both the limits
\[\lim_{x_1\to+\infty}f_\pm^\prime(x_1)\]
exist. Moreover, by \eqref{7.1}, these two limits coincide, which
can be assumed to be $0$ after a rotation.

In the following we will ignore other components of $\{-1<u<1\}$,
thus assume $u\equiv\pm1$ outside $\mathcal{C}$. By the regularity
theory in \cite{AC} and \cite{KN}, both $f_+$ and $f_-$ are smooth.
Then by standard elliptic estimates, there exists a constant $C$ such that
\begin{equation}\label{uniform bound on second derivative}
|\nabla^2u(x)|\leq C,\quad{\mbox{in}}\ \mathcal{C}.
\end{equation}

By these facts and Lemma \ref{limit at infinity}, the limit at
infinity of translations of $u$ along $f_-(x_1)$ must be $g(x_2)$.
Hence, by \eqref{uniform bound on second derivative} and the uniform
smoothness of free boundaries, we get the uniform convergence
\begin{equation}\label{limit of gradient at infinity}
\lim_{x\in\mathcal{C},|x|\to+\infty}|\nabla u|=1.
\end{equation}

It should be emphasized that, in the above setting and the following
proof, we do not need any kind of stability condition.

Let
\[v:=1-|\nabla u|,\]
which vanishes on $\partial\mathcal{C}$.

Direct calculation gives
\begin{equation}\label{01}
\Delta v=-\frac{|\nabla^2 u|^2-|\nabla|\nabla u||^2}{|\nabla u|}.
\end{equation}

Differentiating in $x_1$ twice leads to
\begin{equation*}
\frac{1}{2}\frac{d^2}{dx_1^2}\int_{f_-(x_1)}^{f_+(x_1)}v(x_1,x_2)^2dx_2=\int_{f_-(x_1)}^{f_+(x_1)}\left(\Big|\frac{\partial
v}{\partial
x_1}(x_1,x_2)\Big|^2+v(x_1,x_2)\frac{\partial^2v}{\partial
x_1^2}(x_1,x_2)\right) dx_2.
\end{equation*}
Substituting \eqref{01} into this and integrating by parts, we get
\begin{equation}\label{02}
\frac{1}{2}\frac{d^2}{dx_1^2}\int_{f_-(x_1)}^{f_+(x_1)}v^2=\int_{f_-(x_1)}^{f_+(x_1)}\left(|\nabla
v|^2-v\frac{|\nabla^2 u|^2-|\nabla|\nabla u||^2}{|\nabla u|}\right).
\end{equation}
We have
\begin{equation}\label{02.1}
|\nabla v|^2=\frac{|\nabla^2u\cdot\nabla u|^2}{|\nabla
u|^2}\geq|\nabla^2u\cdot\nabla u|^2,
\end{equation}
because $|\nabla u|\leq 1$. On the other hand, by denoting
$\bar{\nu}:=\nabla u/|\nabla u|$ (recall that we can assume $|\nabla
u|\geq 1/2$ in $\mathcal{C}$) and $\bar{\nu}^\bot$ its rotation by angle
$\pi/2$,
\begin{eqnarray}\label{03}
|\nabla^2 u|^2&=&|\nabla^2u\cdot \bar{\nu}|^2+|\nabla^2u\cdot\bar{\nu}^\bot|^2\nonumber\\
&=&2\frac{|\nabla^2u\cdot \nabla u|^2}{|\nabla u|^2}\\
&\leq&8|\nabla^2u\cdot \nabla u|^2,\nonumber
\end{eqnarray}
where we have used the fact that, by the equation $\Delta u=0$,
\[\nabla^2u(\bar{\nu},\bar{\nu})=-\nabla^2u(\bar{\nu}^\bot,\bar{\nu}^\bot),\quad\mbox{in}\ \mathcal{C}.\]

After enlarging $R$, we can assume $v\leq1/64$ in $\mathcal{C}$.
Combining \eqref{02} and \eqref{03} we obtain
\[|\nabla v|^2-v\frac{|\nabla^2
u|^2-|\nabla|\nabla u||^2}{|\nabla u|}\geq\frac{1}{2}|\nabla v|^2.\]
Hence
\begin{equation}\label{04}
\frac{d^2}{dx_1^2}\int_{f_-(x_1)}^{f_+(x_1)}v(x_1,x_2)^2dx_2\geq\int_{f_-(x_1)}^{f_+(x_1)}|\nabla
v(x_1,x_2)|^2dx_2.
\end{equation}
Because $v(x_1,\cdot)=0$ on $f_-(x_1)$ and $f_+(x_1)$ and
$f_+(x_1)-f_-(x_1)\leq 4$, we have the following Poincare
inequality:
\[\int_{f_-(x_1)}^{f_+(x_1)}\Big|\frac{\partial
v}{\partial x_2}(x_1,x_2)\Big|^2dx_2\geq
\frac{\pi^2}{16}\int_{f_-(x_1)}^{f_+(x_1)}v(x_1,x_2)^2dx_2.\] Thus
\begin{equation}\label{5.0.1}
\frac{d^2}{dx_1^2}\int_{f_-(x_1)}^{f_+(x_1)}v(x_1,x_2)^2dx_2\geq\frac{\pi^2}{16}\int_{f_-(x_1)}^{f_+(x_1)}v(x_1,x_2)^2dx_2.
\end{equation}
Because
\[\lim_{x_1\to+\infty}\int_{f_-(x_1)}^{f_+(x_1)}v(x_1,x_2)^2dx_2=0,\]
differential inequality \eqref{5.0.1} implies that
\[\int_{f_-(x_1)}^{f_+(x_1)}v(x_1,x_2)^2dx_2\leq Ce^{-\frac{\pi}{4}x_1},\quad\forall x_1\ \mbox{ large}.\]

Take a nonnegative function $\eta\in C_0^\infty(-2,2)$ with
$\eta\equiv1$ in $(-1,1)$. For any $t$ large, testing \eqref{04}
with $\eta(x_1+t)$ and integrating by parts, we obtain
\[\int_{t-1}^{t+1}\int_{f_-(x_1)}^{f_+(x_1)}|\nabla v(x_1,x_2)|^2dx_2dx_1\leq Ce^{-\frac{\pi}{4}t}.\]
By \eqref{02.1} and \eqref{03} and using the Cauchy inequality, the
above inequality implies that
\[\int_{t-1}^{t+1}\int_{f_-(x_1)}^{f_+(x_1)}|\nabla^2u(x_1,x_2)|dx_2dx_1\leq Ce^{-\frac{\pi}{8}t}.\]
Integrating this from $x_1$ to $+\infty$, we obtain
\[\int_{t-1}^{t+1}\int_{f_-(x_1)}^{f_+(x_1)}\big|\frac{\partial u}{\partial x_1}(x_1,x_2)\big|dx_2dx_1\leq Ce^{-\frac{\pi}{8}t},\]
which can also be strengthened to
\[\int_{t}^{+\infty}\int_{-\infty}^{+\infty}\big|\frac{\partial u}{\partial x_1}(x_1,x_2)\big|dx_2dx_1\leq Ce^{-\frac{\pi}{8}t}.\]
In the above we have used the fact that $\frac{\partial u}{\partial
x_1}=0$ outside $\mathcal{C}$.

Now the existence of the limit
\[u_\infty(x_2):=\lim_{x_1\to+\infty}u(x_1,x_2)\]
follows. Moreover,
\[\int_{-\infty}^{+\infty}\big|u(x_1,x_2)-u_\infty(x_1,x_2)\big|dx_2\leq Ce^{-\frac{\pi}{8}x_1}.\]
By the uniform Lipschitz bound on $u$, this can also be lifted to
the convergence in $L^\infty(\R)$.

Since $u_\infty(x_2)=g(x_1-t)$ for some constant $t$, by noting the
nondegeneracy condition on $g$ (i.e. $g^\prime=1$ in $\{-1<g<1\}$)
and a corresponding one for $u$ (i.e. a positive lower bound on
$\frac{\partial u}{\partial x_2}$), if $x_1$ large,
\[u(x_1,x_2)+1\geq\frac{1}{2}\biggl(x_2-f_-(x_1)\biggr),\quad 1-u(x_1,x_2)\geq\frac{1}{2}\biggl(f_+(x_1)-x_2\biggr),
\quad\mbox{ in }\mathcal{C},\] which follows from the fact
$\frac{\partial u}{\partial x_2}\geq 1/2$ in $\mathcal{C}$ for $x_1$
large. This then implies
\[\lim_{x_1\to+\infty}f_{\pm}(x_1)=t\pm1.\]
Moreover, the convergence rate is exponential. This finishes the
proof of Theorem \ref{main result 1}.

\part{The second problem}

In dimension $1$, the problem \eqref{equation 0.2} has a solution $g$
satisfying
\begin{equation*}
 \left\{\begin{aligned}
&g(t)\equiv0,\ \ \ \mbox{in}\ (-\infty,0),\\
&g^\prime(t)>0,\ \ \ \mbox{in}\ (0,+\infty),\\
&\lim_{t\to+\infty}g(t)=1.
                          \end{aligned} \right.
\end{equation*}
Here the convergence rate is exponential, there is, there exists a constant $A>0$ such that
\[g(t)=1-Ae^{-\sqrt{2}t}+O\left(e^{-2\sqrt{2}t}\right), \quad \mbox{as } t\to+\infty.\]
Hence the following
quantity is well defined
\[\sigma_0:=\int_0^{+\infty}\left[\frac{1}{2}g^\prime(t)^2+W(g(t))\right]dt<+\infty.\]

Given a unit vector $e$ and a constant $t\in\R$, the trivial
extension $u^\ast(x):=g(x\cdot e-t)$, or the function
\[u^{\ast\ast}(x):=g(x\cdot e-t_1)+g(-x\cdot e+t_2), \quad -\infty<t_2\leq t_1<+\infty,\]
 are solutions of
\eqref{equation 0.2} in $\R^2$. Moreover, they are stable in $\R^2$.
We still call these solutions one dimensional solutions.

\section{Uniform estimates}\label{sec 6}
\setcounter{equation}{0}

In this section $u$ denotes a solution of \eqref{equation 0.2} in
$\R^2$. We prove a Modica type inequality and then deduce the
convexity of free boundaries. In fact, most results in this section
hold for solutions in $\R^n$, for any $n\geq 1$.

The following result is \cite[Proposition 2.1]{Wang-Wei 1}.
\begin{prop}
$u<1$ in $\Omega$.
\end{prop}

As in Part I, to prove the Modica inequality, we first establish a
gradient bound.
\begin{lem}
There exists a constant $C$ such that $|\nabla u|\leq C$ in
$\Omega$.
\end{lem}
\begin{proof}
First standard interior gradient estimates give
\[|\nabla u(x)|\leq C\quad\mbox{in }\{x\in\Omega: \mbox{dist}(x,\partial\Omega)>1\}.\]

For $x_0$ near $\partial\Omega$, denote
$h:=\mbox{dist}(x_0,\partial\Omega)$ and assume this distance is
attained at $y_0\in\partial\Omega$. Let
\[\tilde{u}(x):=\frac{1}{h}u(x_0+hx).\]
Then $\tilde{u}$ is positive in $B_1(0)$, where
\[|\Delta\tilde{u}|=h|W^\prime(h\tilde{u})|\leq Ch.\]
Because $B_h(x_0)$ is tangent to $\partial\Omega$ at $y_0$,
\[1=|\nabla\tilde{u}(z_0)|=z_0\cdot\nabla\tilde{u}(z_0),\]
where $z_0:=\left(y_0-x_0\right)/h$.

 Applying the Hopf Lemma to $\tilde{u}-Ch|x-z_0|^2$ gives
\[1\leq C\left(\tilde{u}(0)-Ch^2\right).\]
Then by the Harnack inequality, if $h<1/C$,
\[\sup_{B_{1/2}(0)}\tilde{u}\leq C.\]
By standard interior gradient estimates,
\[|\nabla u(x_0)|=|\nabla\tilde{u}(0)|\leq C.\qedhere\]
\end{proof}

With this gradient bound at hand, we can prove the Modica
inequality.
\begin{prop}\label{Modica inequality 2}
In $\Omega$,
\[\frac{1}{2}|\nabla u|^2\leq W(u).\]
\end{prop}
\begin{proof}
Denote $P:=|\nabla u|^2/2-W(u)$. Assume
\[\delta:=\sup_{\Omega}P>0,\]
and $x_i\in\Omega$ approaches this sup.

In $\Omega$, $P$ satisfies
\begin{equation}\label{P fct}
\Delta P-2\Delta u\frac{\nabla u}{|\nabla u|^2}\cdot\nabla
P=|\nabla^2u|^2-2\Delta u \nabla^2u\left(\frac{\nabla u}{|\nabla
u|},\frac{\nabla u}{|\nabla u|}\right)+\left(\Delta u\right)^2\geq0.
\end{equation}

If $\limsup \mbox{dist}(x_i,\partial\Omega)>0$, we can argue as in
the proof of the usual Modica inequality to get a contradiction, see
\cite{Modica}.

If $\lim \mbox{dist}(x_i,\partial\Omega)=0$, then $u(x_i)\to0$.
Hence for all $i$ large,
\[\frac{1}{2}|\nabla u(x_i)|^2\geq W(0)+\frac{\delta}{2}.\]
Then we can proceed as in the proof of Proposition \ref{gradient bound} to get a contradiction.
\end{proof}

As in \cite{CJK}, the Modica inequality implies the convexity of
free boundaries.
\begin{lem}\label{convexity}
Each connected component of $\Omega^c$ is convex. Moreover, it is
strictly convex unless $u$ is one dimensional.
\end{lem}
\begin{proof}
Because $P=0$ on $\partial\Omega$ and $P\leq 0$ in $\Omega$,
\begin{equation}\label{6.1.1}
P_\nu\geq 0,\quad \mbox{on}\
\partial\Omega.
\end{equation}
On the other hand,
\begin{equation}\label{P boundary derivative}
P_\nu=\nabla^2u(\nabla u,\nu)-\Delta u|\nabla
u|=\left(\nabla^2u(\nu,\nu)-\Delta u\right)|\nabla u|.
\end{equation}
Hence $\Delta^\prime u:=\Delta u-\nabla^2u(\nu,\nu)\leq0$ on
$\partial\Omega$. As in \cite{CJK}, this implies the convexity of
$\partial\Omega$.

Moreover, if $u\neq g^{\ast}$ or $g^{\ast\ast}$, the inequality in
\eqref{6.1.1} is strict. (This follows from an application of the
Hopf lemma. Note that near $\partial\Omega$, $|\nabla u|$ has a
positive lower bound, hence the second term in \eqref{P fct} is
regular.)  Then $\Delta^\prime u>0$ strictly on $\partial\Omega$,
and the strict convexity of $\partial\Omega$ follows.
\end{proof}
As in Part I, a direct consequence of this convexity is:
\begin{coro}
$\Omega$ is unbounded.
\end{coro}

Next, let \makeatletter
\let\@@@alph\@alph
\def\@alph#1{\ifcase#1\or \or $'$\or $''$\fi}\makeatother
\begin{equation*}
{\Psi(x):=}
\begin{cases}
g^{-1}\circ u(x), &x\in\Omega,\\
0,&x\in\Omega.
\end{cases}
\end{equation*}
\makeatletter\let\@alph\@@@alph\makeatother By the Modica
inequality, $|\nabla\Psi|\leq 1$ in $\Omega$. It can be directly
checked that $\Psi$ satisfies
\[\Delta\Psi=f(\Psi)\left(1-|\nabla\Psi|^2\right),\quad \mbox{in}\ \Omega,\]
where $f(\Psi)=W^\prime(g(\Psi))/\sqrt{2W(g(\Psi))}$.

For applications below, we present a non-degeneracy result for the
zero set $\{u=0\}$. The proof is exactly the same as the one for Proposition \ref{nondegenarcy of contact set 1}.
\begin{prop}\label{nondegenarcy of contact set 2}
If $u\neq u^{\ast\ast}$ (for a unit vector $e$ and two constants
$t_1=t_2$), then for every $x\in\partial\Omega$ and $r>0$,
$|B_r(x)\cap\{u=0\}|>0$.
\end{prop}

In the following, we will always assume $u$ satisfies the above
non-degeneracy condition.

\section{The stable De Giorgi conjecture}\label{sec 7}
\setcounter{equation}{0}

In this section we assume $u$ to be stable outside a compact set. We
use the stability condition to derive an integral curvature bound
and use this to study the convergence of translations of a solution
$u$ to \eqref{equation 0.2}.

Let us first prove the stable De Giorgi conjecture, Theorem
\ref{stable solutions 2}.
\begin{proof}[Proof of Theorem \ref{stable solutions 2}]
As in the proof of Theorem \ref{stable solutions 1}, the stability
condition implies the existence of a positive function $\varphi\in
C^\infty(\overline{\Omega})$ satisfying
\begin{equation}\label{linearized equation}
 \left\{\begin{aligned}
&\Delta \varphi=W^{\prime\prime}(u)\varphi,\ \ \ \mbox{in}\ \Omega,\\
&\varphi_\nu=-\left(\frac{W^\prime(0)}{\sqrt{2W(0)}}-H\right)\varphi,~~\mbox{on}~\partial\Omega.
                          \end{aligned} \right.
\end{equation}
By direct differentiation, for any unit vector $e$, the directional
derivative $u_e:=e\cdot\nabla u$ also satisfies this equation.

Let $\psi:=u_e/\varphi$. It satisfies
\begin{equation}\label{degenerate equation with boundary condition}
 \left\{\begin{aligned}
&\mbox{div}\left(\varphi^2\nabla\psi\right)=0,\ \ \ \mbox{in}\ \Omega,\\
&\psi_\nu=0,~~\mbox{on}~\partial\Omega.
                          \end{aligned} \right.
\end{equation}
 The following proof is
exactly the same as in the proof of Theorem \ref{stable solutions 1}.
\end{proof}

The following result is similar to Lemma \ref{lem 3.1}.
However, since the calculation is a little different, we include a
complete proof here.
\begin{lem}
For any $L>1$, there exists an $R(L)$ such that, there is no bounded
component of $\{u=0\}$ contained in $B_{R(L)}(0)^c$ with diameter
smaller than $L$.
\end{lem}
\begin{proof}
{\bf Step 1.} For any $R$ large, take a $\varphi\in
C_0^\infty(\R^2\setminus B_R(0))$ and test the stability condition with
$\varphi|\nabla u|$. After some integration by parts, we obtain
\begin{eqnarray}\label{7.0.1}
&&-2W(0)\int_{\partial\Omega}\varphi^2\left(\frac{W^\prime(0)}{\sqrt{2W(0)}}-H\right) \nonumber\\
&\leq&\int_\Omega\biggl(|\nabla\varphi|^2|\nabla
u|^2+2\varphi|\nabla u|\nabla\varphi\cdot\nabla |\nabla
u|+\varphi^2|\nabla|\nabla u||^2+W^{\prime\prime}(u)|\nabla u|^2\varphi^2\biggr)\\
&=&\int_{\partial\Omega}\frac{1}{2}\varphi^2\left(|\nabla
u|^2\right)_\nu+\int_{\Omega}\biggl(|\nabla\varphi|^2|\nabla
u|^2-\varphi^2|B|^2|\nabla u|^2\biggr).\nonumber
\end{eqnarray}

On $\partial\Omega$,
\[\left(|\nabla
u|^2\right)_\nu=2\sqrt{2W(0)}u_{\nu\nu}=2\sqrt{2W(0)}W^\prime(0)+
4W(0)H.\]
 Hence \eqref{7.0.1} can be transformed into
\[\int_{\Omega}|\nabla
u|^2|B|^2\varphi^2\leq\int_\Omega|\nabla u|^2|\nabla\varphi|^2.\]
Then we can use standard log cut-off test functions to show that
\[\int_{\Omega\setminus B_R(0)}|\nabla
u|^2|B|^2\leq \frac{C}{\log R},\] which converges to $0$ as
$R\to+\infty$.

{\bf Step 2.} For any $\eta\in C_0^\infty(B_R(0)^c)$,
\begin{eqnarray}\label{6.3}
\int_{\partial\Omega}H\eta&=&\int_{\partial\Omega}P_\nu\eta\nonumber
\\
&=&\int_{\Omega}\biggl(\nabla P\nabla \eta+\eta\Delta P\biggr)\\
&=&\int_{\Omega}\biggl[\left(\nabla^2u\nabla u-\Delta u\nabla
u\right)\nabla\eta+\eta\left(|\nabla^2 u|^2-\big|\Delta
u\big|^2\right)\biggr].\nonumber
\end{eqnarray}

{\bf Claim. } At $x$ where $\nabla u(x)\neq0$,
\begin{equation}\label{7.5}
|\nabla^2u\nabla u-\Delta u\nabla u|\leq C|\nabla u||B|,
\end{equation}
 and
 \begin{equation}\label{7.6}
\Big||\nabla^2 u|^2-\big|\Delta u\big|^2\Big|\leq C\left(|\nabla
u||B|+|\nabla u|^2|B|^2\right).
\end{equation}

To prove this claim, take the coordinates centered at $x$ so that
\[\frac{\nabla u(x)}{|\nabla u(x)|}=(0,1).\]
At $x$ we have
\begin{equation}\label{7.2}
|\nabla u|^2|B|^2=\Big|\frac{\partial^2u}{\partial
x_1^2}\Big|^2+\Big|\frac{\partial^2u}{\partial x_1\partial
x_2}\Big|^2,
\end{equation}
\begin{equation}\label{7.3}
\nabla^2u\nabla u-\Delta u\nabla u=|\nabla
u|\left(\frac{\partial^2u}{\partial x_1\partial
x_2},-\frac{\partial^2u}{\partial x_1^2}\right),
\end{equation}
Because $|\nabla u|\leq \sqrt{2W(u)}\leq C$, \eqref{7.5} follows.
Next
\begin{eqnarray}\label{7.4}
|\nabla^2 u|^2-\big|\Delta
u\big|^2&=&2\left(\Big|\frac{\partial^2u}{\partial x_1\partial
x_2}\Big|^2-\frac{\partial^2u}{\partial
x_1^2}\frac{\partial^2u}{\partial
x_2^2}\right)\\
&=&2\left(\Big|\frac{\partial^2u}{\partial x_1\partial
x_2}\Big|^2+\Big|\frac{\partial^2u}{\partial
x_1^2}\Big|^2\right)-2W^\prime(u)\frac{\partial^2u}{\partial x_1^2}.
\end{eqnarray}
Because $|W^\prime(u)|\leq C$, \eqref{7.6} follows. This finishes
the proof of this {\bf Claim}.

Assume there is a connected component of $\{u=0\}$, $D$, contained
in $B_R(0)^c$ with its diameter smaller than $L$. Take a point $x$
in this component and $\eta$ to be a standard cut-off function in
$B_{2L}(x)$ with $\eta\equiv 1$ in $B_L(x)$. Substituting this into
\eqref{6.3}, using \eqref{7.5} and \eqref{7.6}, and noting that
$H\geq 0$ on $\partial\Omega$, we obtain
\begin{eqnarray*}
\int_{\partial D}H&\leq& C\int_{B_{2L(x)}}|\nabla u||B|+|\nabla
u|^2|B|^2\\
&\leq& C\int_{B_{2L(x)}}|\nabla
u|^2|B|^2+CL\left(\int_{B_{2L(x)}}|\nabla
u|^2|B|^2\right)^{\frac{1}{2}}\\
&\leq& \frac{C\left(1+L\right)}{\sqrt{\log R}}.
\end{eqnarray*}
The remaining proof is exactly the same as the one for Lemma \ref{lem 3.1}.
\end{proof}
As in Part I, the above proof imply two corollaries.
\begin{coro}
For any $x\in\partial\Omega\setminus B_{R(L)}(0)$ and $r\in(0,L/2)$,
the connected component of $\partial\Omega\cap B_r(x)$ passing
through $x$, denoted by $\Gamma^{x,r}$, has its boundary in
$\partial B_r(x)$. Hence,
\[\mathcal{H}^1(\Gamma^{x,r})\geq 2r.\]
\end{coro}

\begin{coro}\label{integral curvature estimate 2}
For any $\varepsilon>0$ small and $L>0$ large, there exists an
$R(L,\varepsilon)$ so that the following holds. For any
$x\in\partial\Omega\setminus B_{R(L,\varepsilon)}(0)$, the connected
component of $\partial\Omega$ passing through $x$, denoted by $\Gamma^{x,L}$,
satisfies
\[\int_{\Gamma^{x,L}\cap B_L(x)}H\leq \varepsilon,\]
and
\[\mbox{dist}_{H}\left(\Gamma^{x,L}\cap B_L(x), \{e^{x,L}\cdot (y-x)=0\}\cap B_L(x)\right)\leq \varepsilon,\]
where $e^{x,L}$ is a unit vector.
\end{coro}
Exactly as in Part I, we have the following characterization
of the convergence of translations of $u$ at infinity.
\begin{lem}\label{translation at infinity}
For any $\ell\geq 1$ and $x_k\in\partial\Omega$, $|x_k|\to\infty$,
the translated function
\[u_k(x):=u(x_k+x)\]
converges in the $C^\ell$ sense to $u^\ast(e\cdot x)$ for some unit
vector $e$, or $u^{\ast\ast}(e,0,t)$ for a unit vector $e$ and a
constant $t\leq0$. Moreover, on any compact set of
$\R^2$, the translation of $\Omega$,
$\Omega_k:=\Omega-x_k$ converges to $\Omega(u^\ast)$ or
$\Omega(u^{\ast\ast})$ in the $C^\ell$ sense.
\end{lem}

\section{Finiteness of unbounded components of $\Omega^c$}\label{sec 8}
\setcounter{equation}{0}

In this section we establish the finiteness of unbounded connected
components of $\Omega^c$, if the solution $u$ of \eqref{equation 0.2}
is stable outside a compact set.

By the stability of $u$ outside a compact set (say $B_{R_0}(0)$ for
some $R_0>0$), there exists a positive function $\varphi\in
C^\infty(\overline{\Omega})$ satisfying
\begin{equation}\label{stability condition 02}
 \left\{\begin{aligned}
&\Delta \varphi=W^{\prime\prime}(u)\varphi,\ \ \ &\mbox{in}\ \Omega\setminus B_{R_0}(0),\\
&\varphi_\nu=-\left(\frac{W^\prime(0)}{\sqrt{2W(0)}}-H\right)\varphi,~~&\mbox{on}~\partial\Omega\setminus
B_{R_0}(0).
                          \end{aligned} \right.
\end{equation}

The following lemma is the main tool of this section.
\begin{lem}\label{crucial lem}
Take a unit vector $e$ and let $D$ be a connected component of
$\{u_e\neq 0\}$. Then $D$ intersects $B_{R_0}$.
\end{lem}
\begin{proof}
Assume $D$ dose not intersect $B_{R_0}$. Let $\psi$ be the
restriction of $|u_e|$ to $D$, extended to be $0$ outside $\Omega$.
Then $0\leq\psi\leq C$ is a continuous subsolution to
\eqref{linearized equation}, because $u_e$ is a solution to
\eqref{linearized equation}.

The function $\phi:=\psi/\varphi_0$ is well defined. It is
nonnegative, continuous. Moreover, it satisfies
\begin{equation}\label{linearized problem 3}
 \left\{\begin{aligned}
&\mbox{div}\left(\varphi_0^2\nabla\phi\right)\geq 0, \quad \mbox{in}\ \Omega,\\
&\phi_\nu\geq0,~~\mbox{on}~\partial\Omega.
                          \end{aligned} \right.
\end{equation}
Note that the support of $\phi$ is contained in $B_{R_0}(0)^c$.

The same proof of Theorem \ref{stable solutions 2}, using the
standard log cut off functions, gives $\phi\equiv 0$. Hence
$u_e\equiv0$ in $D$ and we get a contradiction.
\end{proof}

We use this lemma to prove
\begin{lem}
There are only finitely many unbounded connected components of
$\Omega^c$.
\end{lem}
\begin{proof}
Let $D_\alpha$ be all of the unbounded connected components of
$\Omega^c$. By our preliminary analysis, each $D_\alpha$ is a convex
open domain with smooth boundary. Moreover, $\partial D_\alpha$ is a
single simple curve (see the proof of Lemma \ref{lem 4.3}). Thus we
can take a point $x_\alpha\in\partial D_\alpha$ and a unit vector
$e_\alpha$, such that $D_\alpha$ can be represented by
\[D_\alpha=\{x: x\cdot e_\alpha>f(x\cdot e_\alpha^\perp-x_\alpha \cdot e_\alpha^\perp)+x_\alpha\cdot e_\alpha\},\]
where $f$ is a convex function defined on an (connected) interval
$I_\alpha$ containing $0$. The ray
\[L_\alpha:=\{x_\alpha+te_\alpha,\ t\geq0\}\]
is contained in $D_\alpha$. Hence for different $\alpha$, $L_\alpha$
are disjoint from each other.

Take an $N$ large and two $e_{\alpha_1}$, $e_{\alpha_2}$. We will
prove that if
 \[ e_{\alpha_1}\cdot
e_{\alpha_2}\geq 1-\frac{1}{N},\]
 then in the sector enclosed by $L_{\alpha_1}$ and $L_{\alpha_2}$,
 there are only finitely many
unbounded connected components of $\Omega^c$. It is clear that the
conclusion of this lemma follows from this claim.

Now fix such a sector, which is assumed to be $\{0<x_2<\varepsilon
x_1\}$, with $\varepsilon\leq 1/N$. Take an $R_3$ large so that both
$\{u_{x_2}=0\}\cap\Omega$ and $\partial\Omega$ intersect $\partial
B_{R_3}(0)$ transversally in this cone. (Here $u_{x_2}:=\frac{\partial
u}{\partial x_2}$.)

In the following we denote $\mathcal{C}:=\{0<x_2<\varepsilon
x_1\}\setminus B_{R_3}(0)$.

The number of connected components of $\partial
B_{R_3}(0)\cap\{0<x_2<\varepsilon x_1\}\cap\{u_{x_2}\neq 0\}\cap\Omega$
is denoted by $J$. By Lemma \ref{crucial lem}, there are exactly $J$
connected components of $\{u_{x_2}\neq 0\}\cap\mathcal{C}$.

Now assume the number of unbounded connected components of
$\mathcal{C}\setminus\Omega$ to be larger than $J+1$. Take $J+1$
such components $U_0,\cdots, U_{J+1}$, where $U_0$ is the one
containing $\{x_2=0,x_1\geq R_3\}$ and $U_{J+1}$ the one containing
$\{x_2=\varepsilon x_1,x_1\geq R_3\}$.

For each $i\neq 0, J+1$, there exists an $R_i^\ast$ such that
\[\partial U_i=\left\{(x_1,x_2):x_2=f_i^{\pm}(x_1), \quad \mbox{on}\ x_1\geq R_i^\ast\right\}.\]
Here $f_i^+$ is convex and $f_i^-$ concave.

 For each
$i$, perhaps after enlarging $R_i^\ast$, we can assume
\[\Big|\frac{df_i^\pm}{dx_1}(x_1)\Big|\leq 2\varepsilon,.\]
Since $u=0$ and $|\nabla u|=\sqrt{2W(0)}$ on the curve
$\{x_2=f_i^\pm(x_1)\}$, we see
\[u_{x_2}(x_1,f_i^+(x_1))\geq\sqrt{\frac{W(0)}{2}}>0,\quad u_{x_2}(x_1,f_i^-(x_1))\leq-\sqrt{\frac{W(0)}{2}}<0.\]
The sign follows from the fact that in a neighborhood of
$\{x_2=f_i^\pm(x_1)\}$, $u>0$ on one side and $u=0$ on the other
side. Thus for each $i$, there exists an $R_i^{\ast\ast}$ such that
a neighborhood of $\{(x_1,x_2):x_2=f_i^{\pm}(x_1), \ x_1\geq
R_i^{\ast\ast}\}$ in $\Omega$, which we denote by $U_i^\pm$, is
contained in $\{u_{x_2}>0\}$ and $\{u_{x_2}<0\}$ respectively.

Because the number of connected components of $\{u_{x_2}\neq0\}$ is
not larger than $J$, there exist $i>j$ such that, $U_i^+$ and
$U_j^+$ are contained in the same connected component of
$\{u_{x_2}\neq0\}$. By the connectedness and the unboundedness of
$U_i^+$ and $U_j^+$, there exists a smooth embedded curve contained
in the component containing $U_i^+$ and $U_j^+$. This curve
separates $U_i^-$ from $B_{R_3}(0)$. This contradicts Lemma
\ref{crucial lem} and the claim is proved.
\end{proof}

\begin{coro}
There exists a constant $C$ such that, for any $R>0$ large, the
number of connected components of $\Omega\setminus B_R$ is smaller
than $C$.
\end{coro}
\begin{proof}
First, because $\partial\Omega$ is strictly convex with respect to
the outward normal vector, each component of $\Omega\setminus B_R(0)$
is unbounded. Hence it has an unbounded component of
$\partial\Omega\cap B_R(0)^c$ as its boundary.

On the other hand, the previous lemma says that each unbounded
component of $\Omega^c\setminus B_R(0)$ have at most two unbounded
bounded components of $\partial\Omega\cap B_R(0)^c$.

Thus the number of connected components of $\Omega\setminus B_R(0)$ is
at most two times the number of unbounded components of $\Omega^c$.
\end{proof}


In the following, unbounded connected components of $\Omega^c$ are
denoted by $D_i$, $1\leq i\leq K$ for some $K$. Each $D_i$ contains
a ray $L_i=\{x: x=x_i^\ast+re_i,r\geq 0\}$, where
$x_i^\ast\in\partial D_i$ and $e_i$ is a unit vector.

\begin{lem}
If $i\neq j$, $e_i\neq e_j$.
\end{lem}
\begin{proof}
Assume by the contrary, there are two different components of
$\Omega^c$, $D_1$ and $D_2$, such that
\[\{(x_1,x_2): x_2=t_i, x_1\geq R\}\subset D_i,\quad i=1,2,\]
where $t_1<t_2$.

There is a part of $\partial D_i$ having the form
\[x_2=f_i(x_1),\quad x_1\geq R,\]
where $t_2>f_2(x_1)>f_1(x_1)>t_1$. Here $f_1$ is concave and $f_2$
convex. Hence $f_1$ is eventually increasing in $x_1$ and $f_2$
eventually decreasing in $x_1$. Thus their limits as $x_1\to+\infty$
exist. Moreover,
\begin{equation}\label{4.1}
t_2>\lim_{x_1\to+\infty}f_2(x_1)\geq
\lim_{x_1\to+\infty}f_1(x_1)>t_1.
\end{equation}

By the above choice, $u>0$ in a neighborhood below
$\{x_2=f_2(x_1)\}$ and above $\{x_2=f_1(x_1)\}$.

By Proposition \ref{translation at infinity}, as $t\to+\infty$,
\[u^t(x_1,x_2);=u(x_1+t,x_2+f_1(t))\]
converges to a one dimensional solution. In particular, for any
$L>0$, if $x_1$ is large enough, $u>0$ in $\{(x_1,x_2):
f_1(x_1)<x_2<f_1(x_1)+L\}$. However this contradicts \eqref{4.1} and
the proof is complete.
\end{proof}


\begin{lem}\label{lem 4.15}
Let $\mathcal{C}=\{(x_1,x_2): |x_2|<\lambda x_1, x_1>R\}$ for some
$\lambda>0$ and $R>0$. Assume that $u=0$ in
$\mathcal{C}\cap\{x_2>f_+(x_1)\}$ and
$\mathcal{C}\cap\{x_2<f_-(x_1)\}$, where $f_\pm$ are convex
(concave, respectively) functions defined on $[R,+\infty)$,
satisfying
\[-\lambda x_1<f_-(x)<f_+(x)<\lambda x_1.\]
Then both the limits $\lim_{x_1\to+\infty}f^\prime_{\pm}(x_1)$
exist. Moreover,
\[-\lambda\leq\lim_{x_1\to+\infty}f^\prime_-(x_1)<\lim_{x_1\to+\infty}f^\prime_+(x_1)\leq\lambda.\]
\end{lem}
\begin{proof}
Since $f_+$ is convex, $f^\prime_+(x_1)$ is increasing in $x_1$.
Because $f_+(x_1)\leq\lambda x_1$, it is easy to see that
$f^\prime_+(x_1)\leq\lambda$ for all $x_1>R$, thanks to the convexity
of $f_+$. The existence of $\lim_{x_1\to+\infty}f^\prime_+(x_1)$
then follows. For $f_-$ we have similar statements.

Next assume
\begin{equation}\label{6.2}
\lim_{x_1\to+\infty}f^\prime_-(x_1)=\lim_{x_1\to+\infty}f^\prime_+(x_1).
\end{equation}
Then by noting that $f_+-f_-$ is positive and convex, for all
$x_1>R$,
\[f^\prime_+(x_1)-f^\prime_-(x_1)\leq \lim_{x_1\to+\infty}\left(f^\prime_+(x_1)-f^\prime_-(x_1)\right)=0.\]
Hence $\lim_{x_1\to+\infty}\left(f_+(x_1)-f_-(x_1)\right)$ exists
and this limit lies in $[0,f_+(0)-f_-(0)]$.
 However, by Proposition \ref{translation at infinity}, we know that as $x_1\to+\infty$,
$u(x_1+y_1,f_-(x_1)+y_2)$ converges to $u^\ast$ or $u^{\ast\ast}$
uniformly on any compact set of $\R^2$. In particular, we should
have
\[\lim_{x_1\to+\infty}\left(f^\prime_+(x_1)-f^\prime_-(x_1)\right)=+\infty.\]
This is a contradiction. Thus the assumption \eqref{6.2} does not
hold.
\end{proof}

\section{Finiteness of bounded components of $\Omega^c$}\label{sec 9}
\setcounter{equation}{0}

In this section we use Theorem \ref{thm curvature decay} to show the finiteness of bounded components of $\Omega^c$.
The proof of this theorem is quite involved and will be postponed to
Part III. Roughly speaking, we first use the doubling lemma of
Pol\'{a}cick-Quittner-Souplet \cite{Polacik-Q-S} to reduce the proof
to the following setting:
\begin{enumerate}
\item $u_\varepsilon$ is a solution of
\begin{equation*}
 \left\{\begin{aligned}
&\varepsilon\Delta u_\varepsilon=\frac{1}{\varepsilon}W^\prime(u_\varepsilon),\ \ \ \mbox{in}\ \{u_\varepsilon>0\}\cap Q_1(0),\\
&|\nabla
u_\varepsilon|=\frac{1}{\varepsilon}\sqrt{2W(0)},~~\mbox{on}~\partial\{u_\varepsilon>0\}\cap
Q_1(0).
                          \end{aligned} \right .
\end{equation*}
Here $Q_1(0)=\{|x_1|<1,|x_2|<1\}$.

\item There exists a constant $C>0$ such that
\begin{equation}\label{energy bound A}
\int_{Q_1(0)}\left(\frac{\varepsilon}{2}|\nabla
u_\varepsilon|^2+\frac{1}{\varepsilon}W(u_\varepsilon)\chi_{\{u_\varepsilon>0\}}\right)\leq
C.
\end{equation}
\item $u_\varepsilon$ satisfies the Modica inequality,
\[\frac{\varepsilon}{2}|\nabla u_\varepsilon|^2\leq\frac{1}{\varepsilon}W(u_\varepsilon), \quad\mbox{in } \{u_\varepsilon>0\}.\]
 Hence $\partial\{u_\varepsilon>0\}$ is
convex.
\item The curvature of $\partial\{u_\varepsilon>0\}$ is bounded by $4$, and it equals $1$ at the origin.
\end{enumerate}
The last condition says the free boundaries are uniformly bounded in
$C^{1,1}$. Because free boundaries converge to lines in low
regularity spaces, by some interior regularity results we show that
the curvature at the origin converges to $0$, hence a contradiction
is obtained. 

By noting that in dimension $2$, minimal surfaces are exactly
straight lines, which have zero curvature, we can improve the
conclusion of Theorem \ref{thm curvature decay} to
\begin{coro}
Under the assumptions in Theorem \ref{thm curvature decay},
\[H(x)=o\left(\frac{1}{|x|}\right), \ \ \ \ \mbox{as }\ x\in \partial\Omega \mbox{ and } |x|\to+\infty.\]
\end{coro}
\begin{proof}
Assume there exists $x_k\in\partial\Omega$, $|x_k|\to+\infty$ such
that
\[\lim_{k\to+\infty}|x_k|H(x_k)>0.\]
Define
\[u_k(x):=u\left(x_k+\frac{x}{|x_k|}\right).\]
By Theorem \ref{thm curvature decay}, the curvature of
$\partial\{u_k>0\}\cap B_{1/2}(0)$ is uniformly bounded, while the
curvature at $0$ converges to a positive constant. The remaining
proof is exactly the same as in the proof of Theorem \ref{thm
curvature decay}.
\end{proof}

Using this corollary we prove
\begin{prop}\label{finite bounded components}
There exists an $R_4>0$ such that all bounded connected components
of $\Omega^c$ are contained in $B_{R_4}(0)$.
\end{prop}
\begin{proof}
Assume by the contrary, there are infinitely many bounded connected
components of $\Omega^c$, denoted by $\mathcal{G}_k$, such that
\[\max_{x\in\overline{\mathcal{G}_k}}|x|\to+\infty.\]

Let $x_k\in\overline{\mathcal{G}_k}$ attain this maxima. By the
previous corollary,
\[H(x_k)=o\left(\frac{1}{|x_k|}\right).\]
On the other hand, since $\mathcal{G}_k$ is contained in
$B_{|x_k|}(0)$ and these two sets touch at $x_k$,
\[H(x_k)\geq H\lfloor_{\partial B_{|x_k|}(0)}(x_k)=\frac{1}{|x_k|}.\]
This is a contradiction.
\end{proof}

\section{Blowing down analysis}\label{sec 10}
\setcounter{equation}{0}

In this section we perform the blowing down analysis and give a
description of the blowing down limit.

In the previous two sections we have established the finiteness of
unbounded and bounded connected components of $\Omega^c$. Now we
prove the natural energy growth bound.
\begin{thm}\label{thm energy growth}
There exists a constant $C$ such that, for any $R>1$,
\begin{equation}\label{natrual energy growth bound}
\int_{B_R(0)}\left(\frac{1}{2}|\nabla u|^2+W(u)\chi_\Omega\right)\leq CR.
\end{equation}
\end{thm}
\begin{proof}
By Proposition \ref{finite bounded components}, all bounded
connected components of $\Omega^c$ are contained in $B_{R_4}(0)$.
There are two cases, depending wether there are unbounded components
of $\Omega^c$.

{\bf Case 1.} First assume there is no unbounded component of
$\Omega^c$, hence $u>0$ in $B_{R_4}(0)^c$. For any $x_k\to+\infty$,
let
\[u_k(x):=u(x_k+x).\]
Assume it (up to a subsequence) converges to a limit $u_\infty$ in
$C_{loc}(\R^2)$. Because for any $R>0$, if $k$ large, $u_k>0$ in
$B_R(0)$, we have
\[\Delta u_k=W^\prime(u_k) \quad \mbox{in } B_R(0).\]
By standard elliptic regularity and Arzela-Ascoli theorem, $u_k$
converges to $u_\infty$ in $C^2_{loc}(\R^2)$. Hence, by noting that
$R$ can be arbitrarily large, we get
\[\Delta u_\infty=W^\prime(u_\infty) \quad\mbox{in } \R^2.\]
As in Lemma \ref{translation at infinity}, $u_\infty$ is stable.
Then similar to Theorem \ref{stable solutions 2}, $u_\infty$ is one
dimensional. After a rotation, assume it to be a function of $x_1$
only. Hence $u_\infty$ satisfies
\begin{equation}\label{1d equation}
\frac{d^2 u_\infty}{dx_1^2}=W^{\prime}(u_\infty) \quad\mbox{on } \R.
\end{equation}
 By noting that $u_\infty\geq0$, it is easily seen that
$u_\infty\equiv 1$.

Since this limit is independent of $x_k\to\infty$, we obtain the
uniform convergence
\[\lim_{|x|\to+\infty}u(x)=1.\]
In particular, there exists an $\tilde{R}_E>0$ such that $u>\gamma$
outside $B_{\tilde{R}_E}(0)$. Direct calculation gives
\[\Delta\left(1-\tilde{u}\right)\geq c\left(1-\tilde{u}\right) \quad \mbox{outside } B_{\tilde{R}_E}(0).\]
From this differential inequality we deduce that
\begin{equation}\label{10.1.1}
1-u(x)\leq Ce^{-c|x|} \quad \mbox{outside } B_{\tilde{R}_E}(0).
\end{equation}
By standard elliptic regularity, we also have
\begin{equation}\label{10.1.2}
|\nabla u(x)|\leq Ce^{-c|x|} \quad \mbox{outside }
B_{\tilde{R}_E}(0).
\end{equation}

Because $u$ is a classical solution, we have the following Pohozaev
identity
\begin{equation}\label{Pohozaev}
2\int_{B_R(0)}W(u)\chi_{\{u>0\}}=R\int_{\partial B_R(0)}\left(|\nabla
u|^2-2\big|\frac{x}{|x|}\cdot\nabla u(x)\big|^2+2W(u)\chi_{\{u>0\}}\right),
\quad \forall R>0.
\end{equation}
Letting $R\to+\infty$, by \eqref{10.1.1} and \eqref{10.1.2}, the
right hand side converges to $0$ exponentially. This leads to
\[\int_{\R^2}W(u)\chi_{\{u>0\}}=0.\]
This is only possible if $u\equiv 1$ or $u\equiv 0$.

{\bf Case 2.} Now assume there are unbounded components of
$\Omega^c$. Take an unbounded component of $\Omega\setminus
B_{R_4}(0)$, which we denote by $\mathcal{C}$. Its boundary consists of a
part of $\partial B_{R_4}(0)$ and two convex curves, denoted by
$\Gamma^\pm$. Assume $\Gamma^\pm$ lies on the boundary of $D^\pm$,
two connected components of $\Omega^c$. (We do not claim these two
components of $\Omega^c$ to be different.) Let $L^\pm$ be two rays,  strictly 
contained in $D^\pm$. Assume $L^-$ to be
$\{(x_1,x_2):x_2=0, x_1>R_E\}$ for some $R_E>R_4$. Then $\Gamma^-$
has the form
\[\left\{(x_1,x_2): x_2=f_-(x_1)\right\},\]
where $f_-\geq 0$ is a concave function defined on $(R_E,+\infty)$.

By the concavity, $\lim_{x_1\to+\infty}f_-^\prime(x_1)$ exists,
which is nonnegative because $f_-\geq0$.

{\bf Subcase 2.1.} If the angle between $L_+$ and $L_-$ is smaller
than $\pi/2$, then $\Gamma^+$ can also be represented by
\[\{(x_1,x_2): x_2=f_+(x_1)\},\]
where $f_+>f_-$ is a convex function defined on $(R_E,+\infty)$.

By Lemma \ref{lem 4.15},
\[\lim_{x_1\to+\infty}f^\prime_+(x_1)>\lim_{x_1\to_+\infty}f_-^\prime(x_1).\]
Hence, by letting
\[\lambda:=\frac{1}{2}\left(\lim_{x_1\to+\infty}f^\prime_-(x_1)+\lim_{x_1\to_+\infty}f_+^\prime(x_1)\right),\]
there exist two constants $\bar{R}_E\geq R_4$ and $t\in\R$  so that
the ray
\[L^\ast:=\left\{(x_1,x_2): x_2=\lambda x_1+t, x_1>\bar{R}_E\right\}\]
belongs to $\mathcal{C}$.

Because $u>0$ in $\mathcal{C}$, similar to the derivation of
\eqref{10.1.1}, we have
\begin{equation}\label{10.1.3}
1-u(x)\leq Ce^{-c(x_1-\bar{R}_E)} \quad \mbox{on } L^\ast.
\end{equation}
We claim that there exists a constant $C$ such that
\begin{equation}\label{Hamiltonian identity}
H(x_1):=\int_0^{\lambda
x_1+t}\left(\frac{u_{x_2}^2-u_{x_1}^2}{2}+W(u)\chi_\Omega \right)dx_2\equiv
C+O(e^{-c(x_1-\bar{R}_E)}),\quad \forall\ x_1>\bar{R}_E.
\end{equation}
This is the Hamiltonian identity and it can be proved by
differentiation and integration by parts. The detailed calculation
is postponed to the proof of Proposition \ref{Hamiltonian identity
2} below.

By the Modica inequality, \eqref{Hamiltonian identity} implies that
\begin{equation}\label{3.1}
\int_0^{\lambda x_1+t}\Big|\frac{\partial u}{\partial
x_2}(x_1,x_2)\Big|^2 dx_2\leq C,\quad \forall\ x_1>\bar{R}_E.
\end{equation}
By rotating the plane a little, we also get for a small $\delta$
(depending on $\lambda$), such that
\begin{equation}\label{3.2}
\int_0^{\lambda x_1+t}\left(\Big|\frac{\partial u}{\partial
x_2}(x_1,x_2)\Big|^2+\delta \Big|\frac{\partial u}{\partial
x_1}(x_1,x_2)\Big|^2 \right)dx_2\leq 2C,\quad \forall\ x_1>\bar{R}_E.
\end{equation}
Combining these two inequalities we see, for any $R>\bar{R}_E$,
\begin{equation}\label{6.4}
\int_{B_R\cap\{f_+(x_1)<x_2<\lambda x_1+t\}}|\nabla u|^2\leq CR.
\end{equation}
Because
\[\int_{B_R\cap\{f_+(x_1)<x_2<\lambda x_1+t\}}\left(\frac{u_{x_2}^2-u_{x_1}^2}{2}+W(u)\chi_\Omega\right)\leq CR, \quad\forall R>\bar{R}_E,\]
adding \eqref{6.4} into this we obtain
\[\int_{B_R\cap\{f_+(x_1)<x_2<\lambda x_1+t\}}\left(\frac{|\nabla u|^2}{2}+W(u)\chi_\Omega\right)\leq CR, \quad\forall R>\bar{R}_E.\]
A similar one holds for the energy in $B_R\cap\{\lambda
x_1+t<x_2<f_-(x_1)\}$. This shows that the energy in
$B_R\cap\mathcal{C}$ grows linearly in $R$.

{\bf Subcase 2.2.} Assume the angle between $L^\pm$ is not smaller
than $\pi/2$. In this case we can take two different rays
$\tilde{L}^\pm$, totally contained in $\mathcal{C}$, so that the
angle between $L^+$ and $\tilde{L}^+$ (and $L^-$ between
$\tilde{L}^-$) is smaller than $\pi/2$. These two rays dividing
$\mathcal{C}$ into three subdomains:
\begin{enumerate}
\item $\mathcal{C}^+$, bounded by $\Gamma^+$ and $\tilde{L}^+$;
\item $\mathcal{C}^0$, bounded by $\tilde{L}^+$ and $\tilde{L}^-$;
\item $\mathcal{C}^-$, bounded by $\tilde{L}^-$ and $\Gamma^-$.
\end{enumerate}
The energy in $\mathcal{C}^\pm$ can be estimated as in Subcase 2.1.
In $\mathcal{C}^0$, similar to \eqref{10.1.1}, we have
\[W(u(x))\leq C\left(1-u(x)\right)^2\leq Ce^{-c|x|}.\]
Hence
\[\int_{\mathcal{C}^0}W(u)<+\infty.\]
By the Modica inequality,
\[\int_{\mathcal{C}^0}\left(\frac{1}{2}|\nabla u|^2+W(u)\right)\leq2\int_{\mathcal{C}^0}W(u)<+\infty.\]
Combining this with the energy estimate in $\mathcal{C}^\pm$ we get
the linear growth energy bound in $\mathcal{C}\cap B_R$, for any
$R>1$.

Since there are only finitely many unbounded components of
$\Omega^c$, putting Subcase 2.1 and Subcase 2.2 together we get the
linear energy growth bound.
\end{proof}
Some remarks are in order.
\begin{rmk}
If $u$ is nontrivial, there must exist unbounded connected
components of $\Omega^c$.

In the following we show that  unless $u$ is one dimensional, for
all $R$ large, there are at least two unbounded components of
$\Omega^c\setminus B_R$.
\end{rmk}

For each $\varepsilon>0$, let
\[u_\varepsilon(x)=u(\varepsilon^{-1} x),\]
and $\Omega_\varepsilon:=\varepsilon\Omega$.

By Theorem \ref{thm energy growth}, as $\varepsilon\to0$, the measures
\[\mu_\varepsilon:=\left(\frac{\varepsilon|\nabla u_\varepsilon|^2}{2}+W(u_\varepsilon)\chi_{\{u_\varepsilon>0\}}\right)dx\]
have uniformly bounded mass on any compact of $\R^2$. Hence we can
assume that, perhaps after passing to a subsequence, it converges to
a positive Radon measure $\mu$, weakly on any compact set of $\R^2$.

For application below, we also assume that as Radon measures,
\[\varepsilon|\nabla
u_\varepsilon|^2dx\rightharpoonup\mu_1 \quad \mbox{and } \quad \frac{1}{\varepsilon}W(u_\varepsilon)dx\rightharpoonup\mu_2\]
on any compact set of $\R^2$. Note that $\mu=\mu_1/2+\mu_2$.
 In the following we denote $\Sigma=\mbox{spt}\mu$, the support of the measure $\mu$.

Furthermore, we can also assume the matrix valued measures
\[\varepsilon\nabla u_\varepsilon\otimes \nabla u_\varepsilon dx\rightharpoonup[\tau_{\alpha\beta}]\mu_1,\]
where $[\tau_{\alpha\beta}]$, $1\leq\alpha,\beta\leq 2$, is
measurable with respect to $\mu_1$. Moreover, $\tau$ is nonnegative
definite $\mu_1$-almost everywhere and
\[\sum_{\alpha=1}^2\tau_{\alpha\alpha}=1,\quad \mu_1-a.e.\]

By the Hutchinson-Tonegawa theory (see \cite{Wang-Wei 1}), $\Sigma$ is
countably $1$-rectifiable and $I-\tau=T_x\Sigma$
$\mathcal{H}^1$-a.e. on $\Sigma$.

Define the varifold $V$ by
\[<V,\Phi>:=\int_{\Sigma}\Phi(x,T_x\Sigma)\Theta(x)d\mathcal{H}^1.\]
By \cite{Wang-Wei 1},  $V$ is stationary.

Moreover, $\mu$ can be represented by $\mu=\Theta \mathcal{H}^1\lfloor_{\Sigma}
$,
where $\Theta/\sigma_0$ are positive integers $\mathcal{H}^1$-a.e.
on $\Sigma$.

In the above $\Theta$ is defined by
\[\Theta(x):=\lim_{r\to0}\frac{\mu(B_r(x))}{r}.\]
The existence of this limit is guaranteed by the monotonicity of
$\mu(B_r(x))/r$ (the monotonicity formula for stationary varifolds).

In our setting, because $u_\varepsilon$ is the blowing down sequence
constructed from $u$, for any $r>0$,
\[\frac{\mu(B_r(0))}{r}=\lim_{\varepsilon\to0}\frac{\mu_\varepsilon(B_r(0))}{r}=\lim_{r\to+\infty}\frac{1}{r}\int_{B_r(0)}\left(\frac{1}{2}|\nabla u|^2
+W(u)\chi_\Omega\right),\] which is a constant independent of $r$,
thanks to the monotonicity formula for $u$ (see for example
\cite[Proposition 2.4]{Wang-Wei 1}) and the energy growth bound
Theorem \ref{thm energy growth}. Then by the monotonicity formula
for stationary varifolds, we can show that $\Sigma$ is a cone with
respect to the origin. Hence we have the following characterization
of the blowing down limit.
\begin{prop}\label{limit varifold}
There exist finitely many unit vectors $e_\alpha^{\ast\ast}$
and positive integers $n_\alpha$ such that
\[\mu=\sigma_0\sum_\alpha n_\alpha\mathcal{H}^1\lfloor_{\{re_\alpha^{\ast\ast},r\geq0\}}.\]
Moreover, we have the balancing formula
\[\sum_\alpha n_\alpha e_\alpha^{\ast\ast}=0.\]
\end{prop}
The balancing formula is equivalent to the stationary condition for
the varifold $V$.

In the following we assume $e_\alpha^{\ast\ast}$ are in clockwise
order.

\begin{rmk}\label{rmk 10.4}
If there are only two unit vectors $e_1^\ast$ and $e_2^\ast$, then
\[e_1^\ast=-e_2^\ast, \quad\mbox{and}\quad n_1=n_2.\]
\end{rmk}

Let
\[w(x):=\Phi(u(x))=\int_0^{u(x)}\sqrt{2W(t)}dt,\]
and $w_\varepsilon(x):=w(\varepsilon^{-1}x)$. For any $R>0$,
\begin{eqnarray*}
\int_{B_R(0)}|\nabla
w_\varepsilon|&=&\int_{B_R(0)}\sqrt{2W(u_\varepsilon)}|\nabla
u_\varepsilon|\\
&\leq&\int_{B_R(0)}\left(\frac{\varepsilon}{2}|\nabla
u_\varepsilon|^2+\frac{1}{\varepsilon}W(u_\varepsilon)\chi_{\{u_\varepsilon>0\}}\right)\leq
CR.
\end{eqnarray*}
Since $0\leq w_\varepsilon\leq \int_0^1\sqrt{2W(t)}dt$, it is
uniformly bounded in $BV_{loc}(\R^2)$. Then up to a subsequence
$w_\varepsilon$ converges in $L^1_{loc}(\R^n)$ to a function
$w_\infty\in BV_{loc}(\R^n)$.

By extending $\Phi$ suitably to (-1,1), there exists a continuous
inverse of it. Then $u_\varepsilon=\Phi^{-1}(w_\varepsilon)$
converges to $\Phi^{-1}(w_\infty)$ in $L^1_{loc}(\R^2)$. Since
\[\int_{B_1}W(u_\varepsilon)\chi_{\{u_\varepsilon>0\}}\leq C\varepsilon,\]
$u_\varepsilon\to0$ or $1$ a.e. in $B_1$. Hence there exists a
measurable set $\Omega_\infty$ such that
\[u_\varepsilon\to\chi_{\Omega_\infty},\quad \mbox{in}\ L^1_{loc}(\R^2).\]
Because $w_\infty=(\int_0^1\sqrt{2W(t)}dt)\chi_\Omega$,
$\chi_\Omega\in BV_{loc}(\R^2)$.

As $\varepsilon\to0$,
$\Psi_\varepsilon(x):=\varepsilon\Psi(\varepsilon^{-1}x)$ converges
uniformly to a function $\Psi_\infty$ on any compact set of $\R^2$.
By the {\em vanishing viscosity method}, we can prove that, in the
open set $\{\Psi_\infty>0\}$, $\Psi_\infty$ is a viscosity solution
of the eikonal equation
\[|\nabla\Psi_\infty|^2-1=0.\]
See \cite[Appendix A]{Wang} for more details.

\begin{lem}
$\Psi_\infty=0$ on $\Sigma$.
\end{lem}
\begin{proof}
Assume by the contrary, $\Psi_\infty(x_0)>0$ for some
$x_0\in\Sigma$. Then there exists a ball $B_r(x_0)$ such that
$\Psi_\infty$ has a positive lower bound in this ball. By the
uniform convergence of $\Psi_{\varepsilon_k}$, for all
$\varepsilon_k$ small, $\Psi_{\varepsilon_k}$ also has a uniform
positive lower bound in this ball. Then by the definition of
$\Psi_{\varepsilon_k}$,
\[u_{\varepsilon_k}\geq 1-Ce^{-c\varepsilon_k^{-1}}\quad\mbox{in this ball}.\]
From this it can be checked directly that
\[\frac{\varepsilon_k}{2}|\nabla u_{\varepsilon_k}|^2+\frac{1}{\varepsilon_k}W(u_{\varepsilon_k})\rightarrow 0
\quad\mbox{uniformly in this ball}.\] Hence $x_0$ does not belong to
$\Sigma$. This is a contradiction.
\end{proof}

\begin{lem}
The blowing down limit $\Sigma$ and $\Omega_\infty$ are unique.
\end{lem}
\begin{proof}
Assume for a sequence $\varepsilon_k\to0$, the limit varifold $V$ of
$u_{\varepsilon_k}$ has the form as in Proposition \ref{limit
varifold}. We also assume that $u_{\varepsilon_k}$ converges to
$\chi_{\Omega_\infty}$.

Let $D_\alpha$, $1\leq \alpha\leq K$, be the unbounded connected
components of $\Omega^c$. Since they are open convex sets, the
blowing down limit of $D_\alpha$ (in the Hausdorff distance),
$\lim_{\varepsilon\to 0}\varepsilon D_\alpha$
exists, which we denote by $D_\alpha^\infty$. $D_\alpha^\infty$ is a
convex cone. (It may have no interior points, depending on whether
the opening angle of $D_\alpha$ is positive or zero.) Note that this
limit is independent of $\varepsilon\to0$.

{\bf Claim 1.} $\Psi_\infty=0$ on
$\cup_\alpha\overline{D_\alpha^\infty}$.

This is because, for any compact set
$K\subset D_\alpha^\infty$ which is disjoint from the
origin, it belongs to $\varepsilon D_\alpha$ for all $\varepsilon$
small.

This also implies $u_\infty=0$ a.e. in $\cup_\alpha
D_\alpha^\infty$.

{\bf Claim 2.} $\partial D_\alpha^\infty\subset\Sigma$.\\
Otherwise, because both $\partial D_\alpha^\infty$ and $\Sigma$ are
cones,
\[\delta:=dist\left(\partial D_\alpha^\infty\setminus B_{1/2},\Sigma\setminus B_{1/2}\right)>0.\]
For all $\varepsilon_k$ small,
\begin{equation}\label{5.01}
dist_H\left(\varepsilon_k\partial D_\alpha\cap(B_2\setminus
B_{1/2}),
\partial D_\alpha^\infty\cap(B_2\setminus B_{1/2})\right)\leq\delta/16.
\end{equation}
Hence
\[dist_H\left(\varepsilon_k\partial D_\alpha\cap(B_2\setminus B_{1/2}),
\Sigma\cap(B_2\setminus B_{1/2})\right)\geq\delta/2.\]
 Take an
$x_0\in \partial D_\alpha^\infty\cap\partial B_{3/2}$. By the
definition of $\Sigma$,
\[\lim_{\varepsilon_k\to0}\int_{B_{\delta/2}(x_0)}\left(\frac{\varepsilon_k}{2}|\nabla u_{\varepsilon_k}|^2
+\frac{1}{\varepsilon_k}W(u_{\varepsilon_k})\chi_{\{u_{\varepsilon_k}>0\}}\right)=0.\]
Hence by the Clearing Out Lemma (see \cite[Proposition
3.1]{Wang-Wei 1}), either $u_{\varepsilon_k}\equiv 0$ or
$u_{\varepsilon_k}\geq 1-\gamma$ in $B_{\delta/4}(x_0)$. By
\eqref{5.01}, $\varepsilon_k D_\alpha$ intersects
$B_{\delta/8}(x_0)$. Thus the latter case does not happen. However,
\eqref{5.01} also implies that $\{u_{\varepsilon_k}>0\}$ intersects
$B_{\delta/8}(x_0)$. This is a contradiction and the claim is
proven.

{\bf Claim 3.} $u_\infty=1$ a.e. in $\R^2\setminus\cup_\alpha
\overline{D_\alpha^\infty}$.

By Proposition \ref{finite bounded components}, there exists an
$R_5>0$ such that $u>0$ in
$\mathcal{P}:=\R^2\setminus(B_{R_5}\cup\cup_\alpha D_\alpha)$.
Because $\partial\mathcal{P}\setminus B_{R_5}$ consists only of
finitely many unbounded smooth convex curves, the blowing down limit
$\varepsilon\mathcal{P}$ has a unique limit $\mathcal{P}_\infty$ as
$\varepsilon\to0$, which is exactly $\R^2\setminus\cup_\alpha
D_\alpha^\infty$.

By Proposition \ref{translation at infinity}, for any $\delta>0$,
there exists an $R>R_5$ such that, $u>1-\delta$ in
$\mathcal{P}\setminus(\partial\mathcal{P})_R$, where
\[(\partial\mathcal{P})_R:=\left\{x\in\mathcal{P}: \mbox{dist}(x,\partial\mathcal{P})\leq
R\right\}.\]
 Hence for any $\varepsilon>0$, $u_\varepsilon>1/2$ in
 $\varepsilon[\mathcal{P}\setminus(\partial\mathcal{P})_R]$.
After passing to the limit and by noting the $L^1_{loc}(\R^2)$
convergence of $u_\varepsilon$, we see $u_\infty=1$ a.e. in
$\mathcal{P}_\infty$.

Together with the Clearing Out Lemma (see \cite[Proposition
3.1]{Wang-Wei 1}), the above claims imply that the interior point of
$\mathcal{P}_\infty$ does not belong to $\Sigma$, and the interior
of $\cup_\alpha D_\alpha^\infty$ does not belong to $\Sigma$ either.

Combining Claim 1-3 we see $\Omega_\infty=\R^2\setminus\cup_\alpha
D_\alpha^\infty$ and $\Sigma=\cup_\alpha
\partial D_\alpha^\infty$. By the uniqueness of $D_\alpha^\infty$ we finish the proof.
\end{proof}

\begin{prop}
The density $n_\alpha=1$ or $2$. Moreover, if the opening angle of
$D_\alpha$ is positive, the density on $\partial D_\alpha^\infty$
equals $1$, and if the opening angle of $D_\alpha$ equals $0$, the
density on $\partial D_\alpha^\infty$ is $2$.
\end{prop}
\begin{proof}
Assume $V=\sigma_0\sum_\alpha n_\alpha[L_\alpha]$ where
$n_\alpha\geq 1$ and $L_\alpha=\{re_\alpha^{\ast\ast}: r\geq0\}$. By
the previous lemma, $\cup_\alpha L_\alpha=\Sigma$ is unique, i.e.
independent of $\varepsilon\to0$. Moreover, each $L_\alpha$ belongs
to $\partial D_\beta^\infty$ for some $\beta$.

We have proved that $\Omega^c_\infty=\cup_\alpha D_\alpha^\infty$.
By Lemma \ref{lem 4.15}, if $\alpha\neq\beta$, $\partial
D_\alpha^\infty$ and $\partial D_\beta^\infty$ are disjoint outside
the origin.

Since $e_\alpha^{\ast\ast}$ are all distinct, for each $\alpha$
there is an open neighborhood $U_\alpha$ of
$L_\alpha\cap(B_2\setminus B_{1/2})$ such that these open sets are
disjoint. By Theorem \ref{stable solutions 2} and the proof of
\cite[Theorem 5]{Tonegawa}, for any $\varepsilon>0$ small,
$\partial\{u_\varepsilon>0\}\cap(B_2\setminus B_{1/2})$ consists of
exactly $n_\alpha$ smooth components of
$\partial\{u_\varepsilon>0\}$ in $U_\alpha$, which can be
represented by the graph of convex or concave functions defined on
$L_\alpha$ with small Lipschitz constants.

First assume $D_\alpha^\infty$ to be open. Then there are two unit
vectors $e_+\neq e_-$ such that $\partial D_\alpha^\infty=\{re_\pm:
r\geq0\}$. By the proof of the previous proposition, there exists an
$R>0$ such that, there are exactly two connected components of
$\partial\Omega\setminus B_R$ asymptotic to $\{re_\pm: r\geq0\}$
respectively. Hence the density of the varifold $V$ on $\{re_\pm:
r\geq0\}$ is $1$.

If $D_\alpha^\infty$ is not open, it is a ray in the form $\{re:
r\geq0\}$ for some unit vector $e$. There exists an $R>0$ such that,
there are exactly two connected components of
$\partial\Omega\setminus B_R$ asymptotic to $\{re: r\geq0\}$.  Hence
the density of the varifold $V$ on $\{re_\pm: r\geq0\}$ is $2$.
\end{proof}

Using this result we can prove Corollary \ref{coro 1}.
\begin{proof}[Proof of Corollary \ref{coro 1}]
By Remark \ref{rmk 10.4} and the previous proposition, we can assume
the blowing down limit $\Sigma$ is the $x_1$ axis, with density $1$
or $2$ on it. Since we have assumed $u$ has two ends and each
unbounded connected components of $\Omega^c$ gives two ends, there
is only one unbounded connected component of $\Omega^c$. The blowing
limit of its boundary is the $x_1$ axis. Because it is convex, it
can only be the half space. Then by Lemma \ref{convexity}, $u$ is
one dimensional.
\end{proof}

\section{Refined asymptotics at infinity}\label{sec 11}
\label{section refined asymptotics}
 \setcounter{equation}{0}

In this section, we prove that $u$ is exponentially close to its
ends (one dimensional solutions) at infinity.

Let $\mathcal{C}=\{(x_1,x_2): |x_2|<\lambda x_1, x_1>R\}$ for some
$\lambda>0$ and $R>0$. Assume that $u=0$ in
$\mathcal{C}\cap\{x_2>f_+(x_1)\}$ and
$\mathcal{C}\cap\{x_2<f_-(x_1)\}$, where $f_\pm$ are convex
(concave, respectively) functions defined on $[R,+\infty)$,
satisfying
\[-\lambda x_1<f_-(x_1)<f_+(x_1)<\lambda x_1.\]
Recall that Lemma \ref{lem 4.15} says
\[-\lambda\leq \lambda_-:=\lim_{x_1\to+\infty}f^\prime_-(x_1)<\lambda_+:=\lim_{x_1\to+\infty}f^\prime_+(x_1)\leq\lambda.\]
By Proposition \ref{finite bounded components}, there exists an
$R>0$ such that
\[u>0 \quad \mbox{in}\ \left\{(x_1,x_2): f_-(x_1)<x_2<f_+(x_1), x_1>R\right\}.\]
Hence after taking another larger $R$, the ray
\[\left\{(x_1,x_2): x_2=\frac{\lambda_-+\lambda_+}{2}x_1+\frac{f_-(R)+f_+(R)}{2}, x_1>R\right\}\]
is contained in $\{u>0\}$.

After a rotation, we are in the following situation:
\begin{enumerate}
\item[{\em (H1)}] There are two positive constants $R>0$ large and
$\lambda>0$.
\item[{\em (H2)}] There is a
positive concave function $x_2=f(x_1)$ defined on $[R,+\infty)$ such
that
\[f^\prime(x_1)>0,\quad \lim_{x_1\to+\infty}f^\prime(x_1)=0.\]
In particular, as $x_1\to+\infty$, $f(x_1)=o(x_1)$.
\item[{\em (H3)}] The domain $\mathcal{C}:=\left\{(x_1,x_2): f(x_1)<x_2<\lambda x_1,
x_1>R\right\}$.
\item[{\em (H4)}] $u\in C^2(\overline{\mathcal {C}})$ and $u>0$ in
$\mathcal{C}$. Moreover,
\begin{equation*}
 \left\{\begin{aligned}
&\Delta u=W^\prime(u),\ \ \ &\mbox{in}\ \mathcal {C},\\
&u=0,\ \ \ &\mbox{on}\ \{(x_1,x_2):x_2=f(x_1)\},\\
&|\nabla u|=\sqrt{2W(0)},~~&\mbox{on}~\{(x_1,x_2):x_2=f(x_1)\}.
                          \end{aligned} \right .
\end{equation*}
\end{enumerate}

For simplicity, we will also assume $u=0$ below $\{x_2=f(x_1)\}$.
Note that by the regularity theory in \cite{AC} and \cite{KN}, $f$
is smooth.

As in Part I, we want to emphasize that no stability condition is
needed here.

This section is devoted to prove the following theorem, which also
finishes the proof of Theorem \ref{main result 2}.
\begin{thm}
There exists a constant $t$ such that
\[|f(x_1)-t|\leq Ce^{-\frac{x_1}{C}},\]
and
\[|u(x_1,x_2)-g(x_2-t)|\leq Ce^{-\frac{x_1}{C}}.\]
\end{thm}

First we note that
\begin{lem}\label{exponential decay}
In $\overline{\mathcal{C}}$,
\[1-u(x_1,x_2)\leq Ce^{-\frac{x_2-f(x_1)}{C}}.\]
\end{lem}
\begin{proof}
Because $0<f^\prime(x_1)<1/2$, in $\mathcal{C}$ the distance to the
curve $\{x_2=f(x_1)\}$ is comparable to $x_2-f(x_1)$.

Next by our assumptions, the translation of $u$ along
$\{x_2=f(x_1)\}$ converges to $g(x_2)$ uniformly on compact sets of
$\R^2$. Hence we can assume that, for some $L>0$, $u(x_1,x_2)\geq
1-\gamma$ in $\{(x_1,x_2):f(x_1)+L<x_2<\lambda x_1\}$. By the
equation for $u$,
\[\Delta(1-u)\geq c(1-u) \quad \mbox{in}\ \{(x_1,x_2):f(x_1)+L<x_2<\lambda x_1\}.\]
The conclusion then follows from some standard methods, e.g.
comparison with a sup solution.
\end{proof}
A direct corollary is
\begin{equation}\label{5.1}
1-u(x_1,x_2)\sim O(e^{-cx_1})\quad \mbox{on}\ \{x_2=\lambda x_1\}.
\end{equation}

By Proposition \ref{translation at infinity}, the limit at infinity
of translations of $u$ along $(x_1,f(x_1))$ is $g(x_2)$. The
previous lemma then implies that
\begin{equation}\label{5.1.1}
\lim_{x_1\to+\infty}\sup_{x_2\in\R}|u(x_1,x_2)-g(y-f(x_1))|=0.
\end{equation}

Another consequence of this exponential decay is:
\begin{coro}\label{exponential decay derivatives}
In $\mathcal{C}$,
\[|u_{x_1}(x_1,x_2)|+|u_{x_1x_1}(x_1,x_2)|\leq Ce^{-c\left(x_2-f(x_1)\right)}.\]
\end{coro}
This follows from standard interior gradient estimates and boundary
gradient estimates. (Note that $\partial\mathcal{C}$ is smooth with
uniform $C^3$ bound.)

For application below, we also note the following Hamiltonian
identity.
\begin{prop}\label{Hamiltonian identity 2}
Let
\[\rho(x_1):=\int_{f(x_1)}^{\lambda x_1}\left(\frac{u_{x_2}^2-u_{x_1}^2}{2}+W(u)\right)dx_2,\quad x_1>R. \]
Then
\[\rho(x_1)=\sigma_0+O(e^{-cx_1}).\]
\end{prop}
\begin{proof}
Differentiating in $x_1$ and integrating by parts give
\begin{eqnarray*}
\rho^\prime(x_1)&=&\int_{f(x_1)}^{\lambda
x_1}\biggl(u_{x_2}u_{x_2x_1}-u_{x_1}u_{x_1x_1}+W^\prime(u)u_{x_1}\biggr)dx_2\\
&&+\left(\frac{u_{x_2}(x_1,f(x_1))^2-u_{x_1}(x_1,f(x_1))^2}{2}+W\left(u(x_1,f(x_1))\right)\right)
f^\prime(x_1)+O(e^{-cx_1})\\
&=&\int_{f(x_1)}^{\lambda
x_1}\biggl(u_{x_2}u_{x_2x_1}+u_{x_1}u_{x_2x_2}\biggr)dx_2+u_{x_2}(x_1,f(x_1))^2
f^\prime(x_1)+O(e^{-cx_1})\\
&=&-u_{x_2}(x_1,f(x_1))u_{x_1}(x_1,f(x_1))+u_{x_2}(x_1,f(x_1))^2 f^\prime(x_1)+O(e^{-cx_1})\\
&=&O(e^{-cx_1}).
\end{eqnarray*}
By the convergence of translations of $u$ along $(x_1,f(x_1))$,
Proposition \ref{exponential decay} and Corollary \ref{exponential
decay derivatives}, we
get
\[\lim_{x_1\to+\infty}\rho(x_1)=\sigma_0,\] and the conclusion
follows.
\end{proof}

Let
\[v(x_1,x_2):=u(x_1,x_2)-g(x_2-f(x_1)).\]
By Lemma \ref{exponential decay} and \eqref{5.1.1},
\begin{equation}\label{5.2}
\lim_{x_1\to+\infty}\|v\|_{L^2(0,\lambda x_1)}=0.
\end{equation}

In the following we denote
\[g_\ast:=g(x_2-f(x_1)).\]
For applications below, we note the following fact.
\begin{lem}\label{lem 9.1}
In $\mathcal{C}$, $u\leq g_\ast$.
\end{lem}
\begin{proof}
Recall that the distance type function $\Psi$ satisfies
\[\Psi_{x_2}\leq 1, \quad\mbox{in}\ \mathcal{C}.\]
Because $\Psi=0$ on $\{x_2=f(x_1)\}$,
\[\Psi(x_1,x_2)\leq x_2-f(x_1), \quad\mbox{in}\ \mathcal{C}.\]
Since $g$ is non-decreasing, we then obtain
\[u=g(\Psi)\leq g(x_2-f(x_1)), \quad\mbox{in}\ \mathcal{C}.\qedhere\]
\end{proof}

Similar to the calculation in \cite[page 927]{Gui}, we have
\begin{eqnarray*}
&&\int_{f(x_1)}^{\lambda
x_1}\left(\frac{u_{x_2}^2-|g_\ast^\prime|^2}{2}+W(u)-W(g_\ast)-\frac{u_{x_1}^2}{2}\right)dx_2\\
&=&\int_{f(x_1)}^{\lambda
x_1}\left[W(u)-W(g_\ast)-\frac{W^\prime(u)+W^\prime(g_\ast)}{2}\left(u-g_\ast\right)\right]dx_2\\
&&+\frac{1}{2}\int_{f(x_1)}^{\lambda
x_1}\left[\left(u-g_\ast\right)u_{x_1x_1}-u_{x_1}^2\right]dx_2+O(e^{-cx_1})\\
 &=&\frac{1}{2}\int_{f(x_1)}^{\lambda
x_1}\left[\left(u-g^\ast\right)u_{x_1x_1}-u_{x_1}^2\right]dx_2+o(\|v\|^2)+O(e^{-cx_1}).
\end{eqnarray*}
Hence
\begin{equation}\label{5.5}
\int_{f(x_1)}^{\lambda
x_1}\left[u_{x_1x_1}\left(u-g_\ast\right)-u_{x_1}^2\right]=O(e^{-cx_1})+o(\|v\|^2).
\end{equation}

The following result says the second eigenvalue if $g$ is positive.
\begin{prop}\label{prop second eigenvalue positive}
 For any $L>0$ and $\phi\in H_0^1(0,L)$, there exists a
constant $\mu>0$ (independent of $L$) such that
\begin{equation}\label{5.6}
\int_0^L\biggl[\phi^\prime(t)^2+W^{\prime\prime}(g(t))
\phi(t)^2\biggr]dt\geq\mu\int_0^L\phi(t)^2dt.
\end{equation}
\end{prop}
 This can be proved by a contradiction argument, see Proposition \ref{prop nondegeneracy 2}.

Because $v(x_1,f(x_1))=0$, \eqref{5.6} applies to $v$, which gives
\begin{eqnarray}\label{5.7}
&&\int_{f(x_1)}^{\lambda
x_1}\left[-\left(u-g_\ast\right)_{x_2x_2}+W^{\prime\prime}(g_\ast)\left(u-g_\ast\right)\right]\left(u-g_\ast\right)dx_2
\nonumber\\
&=&\int_{f(x_1)}^{\lambda
x_1}\biggl[\big|\left(u-g_\ast\right)_{x_2}\big|^2+W^{\prime\prime}(g_\ast)\left(u-g_\ast\right)^2\biggr]dx_2+O(e^{-cx_1})
\\
&\geq&\mu\|v\|^2+O(e^{-cx_1}).\nonumber
\end{eqnarray}
Differentiating $\|v\|^2$ twice in $x_1$ leads to
\[
\frac{1}{2}\frac{d}{dx_1}\|v\|^2=\int_{f(x_1)}^{\lambda
x_1}\left(u-g_\ast\right)\left[u_{x_1}+g_\ast^\prime
f^\prime(x_1)\right]dx_2+O(e^{-cx_1}).
\]

\begin{eqnarray*}
\frac{1}{2}\frac{d^2}{dx_1^2}\|v\|^2
 &=&\int_{f(x_1)}^{\lambda x_1}\biggl[ u_{x_1}^2+2u_{x_1}g_\ast^\prime
f^\prime(x_1)+\big|g_\ast^\prime\big|^2f^\prime(x_1)^2+
u_{x_1x_1}\left(u-g^\ast\right)\biggr]dx_2\\
&&-f^\prime(x_1)^2\int_{f(x_1)}^{\lambda x_1}
g_\ast^{\prime\prime}\left(u-g_\ast\right)dx_2+f^{\prime\prime}(x_1)\int_{f(x_1)}^{\lambda
x_1} g_\ast^{\prime}\left(u-g_\ast\right)dx_2+O(e^{-cx_1})\\
 &=&\int_{f(x_1)}^{\lambda x_1} \biggl[\frac{16}{9}u_{x_1}^2+2u_{x_1}g_\ast^\prime
f^\prime(x_1)+\big|g_\ast^\prime\big|^2f^\prime(x_1)^2+\frac{2}{9}
u_{x_1x_1}\left(u-g^\ast\right)\biggr]dx_2  \quad\mbox{(by \eqref{5.5})}\\
&&-f^\prime(x_1)^2\int_{f(x_1)}^{\lambda x_1}
g_\ast^{\prime\prime}\left(u-g_\ast\right)dx_2+f^{\prime\prime}(x_1)\int_{f(x_1)}^{\lambda
x_1} g_\ast^{\prime}\left(u-g_\ast\right)dx_2\\
&&+o(\|v\|^2)+O(e^{-cx_1})\\
&=&\int_{f(x_1)}^{\lambda x_1}\biggl[\left(
\frac{4}{3}u_{x_1}^2+\frac{3}{4}g_\ast^\prime
f^\prime(x_1)\right)^2+\frac{2}{9}
u_{x_1x_1}\left(u-g^\ast\right)\biggr]dx_2\\
&&+f^\prime(x_1)^2\int_{f(x_1)}^{\lambda
x_1}\left[\frac{7}{16}\big|g_\ast^\prime\big|^2-
g_\ast^{\prime\prime}\left(u-g_\ast\right)\right]dx_2+f^{\prime\prime}(x_1)\int_{f(x_1)}^{\lambda
x_1} g_\ast^{\prime}\left(u-g_\ast\right)dx_2\\
&&+o(\|v\|^2)+O(e^{-cx_1}).
\end{eqnarray*}
Note that the last integral is non-negative because
$f^{\prime\prime}\leq 0$, $g_\ast^\prime\geq0$ and $u-g_\ast\leq 0$
(see Lemma \ref{lem 9.1}). We also have
\begin{equation}\label{10.7}
f^\prime(x_1)^2\int_{f(x_1)}^{\lambda
x_1}\left[\frac{7}{16}\big|g_\ast^\prime\big|^2-
g_\ast^{\prime\prime}\left(u-g_\ast\right)\right]dx_2\geq0,
\end{equation}
 because we have
\[\lim_{x_1\to+\infty}\int_{f(x_1)}^{\lambda
x_1}\frac{7}{16}\big|g_\ast^\prime\big|^2dx_2=\frac{7}{16}\sigma_0,\]
while
\[\int_{f(x_1)}^{\lambda
x_1}g_\ast^{\prime\prime}\left(u-g_\ast\right)dx_2\leq\left[\int_{f(x_1)}^{\lambda
x_1}\big|g_\ast^{\prime\prime}\big|^2dx_2\right]^{\frac{1}{2}}\|v\|,\]
which converges to $0$ as $x_1\to+\infty$, thanks to Corollary
\ref{exponential decay derivatives} and \eqref{5.2}.

Thus
\[
\frac{1}{2}\frac{d^2}{dx_1^2}\|v\|^2\geq\frac{2}{9}\int_{f(x_1)}^{\lambda
x_1} u_{x_1x_1}\left(u-g^\ast\right)dx_2+o(\|v\|^2)+O(e^{-cx_1}).
\]


Next, similar to \cite{Gui} we also have
\begin{eqnarray}\label{5.8}
&&\int_{f(x_1)}^{\lambda
x_1}u_{x_1x_1}\left(u-g^\ast\right)dx_2\nonumber\\
&=&\int_{f(x_1)}^{\lambda
x_1}\left(W^\prime(u)-u_{x_2x_2}\right)\left(u-g^\ast\right)dx_2\nonumber\\
&=&\int_{f(x_1)}^{\lambda
x_1}\biggl[W^\prime(u)-W^\prime(g^\ast)-W^{\prime\prime}(g^\ast)\left(u-g^\ast\right)\biggr]\left(u-g^\ast\right)dx_2\\
&&+\int_{f(x_1)}^{\lambda
x_1}\left(g^{\ast\prime\prime}-u_{x_2x_2}\right)\left(u-g^\ast\right)+W^{\prime\prime}(g^\ast)\left(u-g^\ast\right)^2dx_2\nonumber\\
&\geq&\left(\mu+o(1)\right)\|v\|^2+O(e^{-cx_1}).\nonumber
\end{eqnarray}
Here the last step is deduced from \eqref{5.7}.

This then implies that
\[\frac{d^2}{dx_1^2}\|v\|^2\geq c\|v\|^2+O(e^{-cx_1}),\quad\mbox{for all $x_1$ large}.\]
From this inequality and \eqref{5.2} we deduce that
\[\|v\|^2\leq Ce^{-cx_1},\quad\mbox{for all $x_1$ large}.\]

 By Corollary \ref{exponential decay derivatives}, for
any $x_1>R$,
\[\int_{f(x_1)}^{\lambda x_1}u_{x_1x_1}^2dx_2\leq C.\]
Hence the Cauchy-Schwartz inequality gives
\[\int_{f(x_1)}^{\lambda x_1}u_{x_1x_1}\left(u-g_\ast\right)dx_2=O(e^{-cx_1}).\]
Then by \eqref{5.5},
\[\int_{f(x_1)}^{\lambda x_1}u_{x_1}^2dx_2=O(e^{-cx_1}).\]
By the Cauchy-Schwartz inequality,
\[\int_{f(x_1)}^{\lambda x_1}\big|u_{x_1}\big|dx_2=O(e^{-cx_1}).\]
Integrating this in $x_1$ and noting that $u_{x_1}=0$ below
$\{x_2=f(x_1)\}$, we get a function $u_\infty(x_2)=g(x_2-t)$ for
some constant $t$, such that
\[\int_{-\infty}^{+\infty}\big|u(x_1,x_2)-u_\infty(x_2)\big|dx_2=O(e^{-cx_1}).\]
This can also be lifted to an estimate in $L^\infty(\R)$ by the
uniform Lipschitz bound on $u(x_1,x_2)-u_\infty(x_2)$.

For all $x_1$ large, in $[f(x_1),f(x_1)+1]$,
\[u_{x_2}\geq\sqrt{2W(u)}/2,\]
which has a uniform positive lower bound. By this nondegeneracy
property and a similar one for the one dimensional solution $g$, we
deduce that
\begin{equation}\label{8.1}
|f(x_1)-t|=O(e^{-cx_1}).
\end{equation}

Another method to prove \eqref{8.1} is by noting \eqref{10.7}, we in
fact have
\[\frac{d^2}{dx_1^2}\|v\|^2\geq cf^\prime(x_1)^2.\]
Take a nonnegative function $\eta\in C_0^\infty(-2,2)$ with
$\eta\equiv1$ in $(-1,1)$. For any $t$ large, testing the above
inequality with $\eta(x_1+t)$ and integrating by parts, we obtain
\[\int_{t-1}^{t+1}f^\prime(x_1)^2dx_1=O(e^{-cx_1}).\]
Because $f$ is concave and hence $f^\prime(x_1)$ is non-increasing
in $x_1$, this implies that
\[\big|f^\prime(x_1)\big|=O(e^{-cx_1}),\]
and the exponential convergence of $f$ follows.

\part{Proof of Theorem \ref{thm curvature decay}}

This part is devoted to a proof of Theorem \ref{thm curvature
decay}. It is organized as follows. We first
give some preliminary construction using the doubling lemma of
Pol\'{a}cik-Quittner-Souplet \cite{Polacik-Q-S} in Section \ref{Sec
12}. Then the proof is divided into two cases. In Section \ref{sec Fermi coordinates} some results on Fermi coordinates are established. The first case is
treated in Section \ref{sec proof of case 1} and the second one in Section
\ref{sec case 2}-\ref{sec reduction of stable}. 

\section{Reduction to a local estimate}\label{Sec 12}
\setcounter{equation}{0}

In this section $u$ always denotes a classical solution of \eqref{equation
0.2} in $\R^2$, which is stable outside a compact set. By Lemma \ref{translation at infinity}, $H$ is bounded and
 \begin{equation}\label{A.0.1}
  H(x)\to0,\quad\mbox{as
}x\to\infty.
\end{equation}

With the extrinsic distance, $\partial\Omega$ is a complete metric
space. Assume by the contrary, there exists $x_k\in
\partial\Omega$ such that $H(x_k)|x_k|\geq 2k$. Because $H$ is bounded on $\partial\Omega$, $|x_k|\to+\infty$. By the doubling lemma in \cite{Polacik-Q-S}, there exist $y_k\in
\partial\Omega$ satisfying
\[H(y_k)\geq H(x_k), \ \ \ \ H(y_k)|y_k|\geq 2k,\]
\[H(x)\leq 2H(y_k)\ \ \ \ \ \text{for}\ \ x\in \partial\Omega\cap B_{kH(y_k)^{-1}}(y_k).\]

Because $H$ is strictly positive and bounded (see \eqref{A.0.1}), we must have $|y_k|\to+\infty$ and hence by
\eqref{A.0.1},
\[\varepsilon_k:=H(y_k)\to0.\]

Define
\[u_k(x):=u(y_k+\varepsilon_k^{-1}x).\]
It satisfies
\begin{equation}\label{equation 0.2 scaled}
 \left\{\begin{aligned}
&\varepsilon_k\Delta u_k=\frac{1}{\varepsilon_k}W^\prime(u_k),\ \ \ \mbox{in}\ \{u_k>0\},\\
&|\nabla
u_k|=\frac{1}{\varepsilon_k}\sqrt{2W(0)},~~\mbox{on}~\partial\{u_k>0\}.
                          \end{aligned} \right .
\end{equation}
Because
\[|y_k|\geq 2k\varepsilon_k^{-1},\]
$u_k$ is stable in $B_k(0)$. Moreover, by denoting $H_k$ the
curvature of $\partial\{u_k>0\}$ (note that it is still positive), a
rescaling gives
\begin{equation}\label{normalized curvature 1}
H_k(0)=1,\quad\mbox{and    }\ \  H_k\leq 2\quad\mbox{on} \
\partial\{u_k>0\}\cap B_k(0).
\end{equation}
By the latter curvature bound, for any $x\in
\partial\{u_k>0\}\cap B_{k-1}(0)$, the connected component of
$\partial\{u_k>0\}\cap B_{1/8}(x)$ containing $x$ can be represented
by the graph of a convex function with its $C^{1,1}$ norm bounded by
$4$.

Denote the connected component of $\partial\{u_k>0\}\cap B_{1/8}(0)$
containing $0$ by $\Gamma_{k,0}$, which we assume to be
\[\{(x_1,x_2):x_2=f_k(x_1)\},\quad \mbox{for } x_1\in(-1/8,1/8),\]
where $f_k(0)=f_k^\prime(0)=0$. By \eqref{normalized curvature 1},
$-4\leq f_k^{\prime\prime}\leq 0$ and
\begin{equation}\label{normalized curvature 2}
f_k^{\prime\prime}(0)=-1.
\end{equation}

Assume $\Gamma_{k,0}$ to be the boundary of a connected component of
$\{u_k>0\}\cap B_{1/8}(0)$, $\Omega_k$. Without loss of generality,
assume $u_k=0$ in $B_{1/8}(0)\setminus \Omega_k$, that is, we ignore
other connected components of $\{u_k>0\}\cap B_{1/8}(0)$ other than
$\Omega_k$.

We divide the proof into two cases.

{\bf Case 1.}
$\lim_{k\to+\infty}\mbox{dist}(0,\partial\Omega_k\setminus\Gamma_{k,0})>0$.

In this case, there exists a constant $r>0$ such that in
$B_r(0)$, $u_k=0$ below $\{x_2=f_k(x_1)\}$ and $u_k>0$ above this
curve.

In Section \ref{sec proof of case 1} we will prove
\begin{prop}\label{prop curvature estimate 1}
There exists a constant $C$ independent of $\varepsilon_k$ such that
\[|f_k^{\prime\prime}(0)|\leq C\varepsilon_k^{1/2}.\]
\end{prop}
This is a contradiction with
\eqref{normalized curvature 2}. Hence this case is impossible.

{\bf Case 2.}
$\lim_{k\to+\infty}\mbox{dist}(0,\partial\Omega_k\setminus\Gamma_{k,0})=0$.

In this case, there exists
$z_k\in\partial\Omega_k\setminus\Gamma_{k,0}$ such that $|z_k|\to0$.
Assume $z_k$ attains
$\mbox{dist}(0,\partial\Omega_k\setminus\Gamma_{k,0})$. Recall that
in $B_{1/8}(z_k)$, the connected component of $\partial\Omega_k$
passing through $z_k$ is a graph and it is disjoint from
$\Gamma_{k,0}$. Thus for all $k$ large, in $B_{1/16}(0)$ this
component has the form
\[\{x_2=\tilde{f}_k(x_1)\}=:\Gamma_{k,1},\]
where $\tilde{f}_k$ is a convex function with its $C^{1,1}$ norm
bounded by $8$.

We claim that there exists a constant $r>0$ independent of $k$ such
that $\{(x_1,x_2):f_k(x_1)<x_2<\tilde{f}_k(x_1), |x_1|<r\}$ is a
connected component of $B_r(0)\cap\Omega_k$.

Assume by the contrary, there exists a third connected component of
$B_r(0)\cap\partial\Omega_k$ lying between $\{x_2=f_k(x_1)\}$ and
$\{x_2=\tilde{f}_k(x_1)$. Take an arbitrary point $z$ on it. Let $T$
be the tangent line of this component at $z$. This component
can be represented by the graph of a convex function defined on an
interval of this line, which contains $z$ and its length is not smaller than
$1/8$. Because different components of $\partial\Omega_k\cap
B_{1/2}(0)$ are disjoint, the tangent line at $z$ must be almost
parallel to the $x_1$-axis. Since $z$ is arbitrary, this implies
that this third component is also a graph on the $x_1$-axis, which
is defined on $(-r,r)$. Since this curve lies between
$\{x_2=f_k(x_1)\}$ and $\{x_2=\tilde{f}_k(x_1)$, while
$0\in\{x_2=f_k(x_1)\}$ and $z_k\in\{x_2=\tilde{f}_k(x_1)$, there
exists a point on this third component, which is closer to $0$ than
$z_k$. This is a contradiction with the choice of $z_k$ and finishes
the proof of this claim.

In the following  we will prove
\begin{prop}\label{prop curvature estimate 2}
There exists a constant $C$ independent of $\varepsilon_k$ such that
\[|f_k^{\prime\prime}(0)|\leq C\varepsilon_k^{1/7}.\]
\end{prop}
 This is a contradiction with
\eqref{normalized curvature 2}. Hence this case is also impossible
and we finish the proof of Theorem \ref{thm curvature decay}.

\section{Fermi coordinates}\label{sec Fermi coordinates}
\setcounter{equation}{0}

In this section we present several results related to Fermi coordinates with respect to free boundaries, which will be used later. This discussion will be mainly concentrated on Case 2, while Case 1 can be easily  recovered.
We will work in the stretched setting, i.e. after the rescaling $x\mapsto\varepsilon^{-1}x$. The dependence on $\varepsilon$ will not be written explicitly.
In the following $R:=\varepsilon^{-1}$. Therefore there exists a concave function $f_1$ and a convex function $f_2$ with
$f_1(0)=f_1^{\prime}(0)=0$ and
$f_2>f_1$ in $(-R,R)$, such that (here $\Omega=\{u>0\}$ and $Q_R:=\{|x_1|<R, |x_2|<R\}$)
\[\Omega\cap Q_R=\{(x_1,x_2):f_1(x_1)<x_2<f_2(x_1)\}.\]
The curves $\{x_2=f_i(x_1)\}$ will be denoted by $\Gamma_i$.

\subsection{Definition}

 The curvature of $\Gamma_i$ ($i=1,2$) with respect to the parametrization $y\mapsto (y,f_i(y))$ is given by
\[H^i(y,0)=\pm\frac{f_i^{\prime\prime}(y)}{\left(1+f_i^\prime(y)^2\right)^2}.\]
Here we take the negative sign for $i=1$  and the positive  one for $i=2$, which gives $H^i\geq0$ for both $i=1,2$.

\smallskip

The Fermi coordinate is defined by $(y,z)\mapsto x$ as $x=(y,f_i(y))+zN_i(y)$,
where
\[N_i(y)=\pm\frac{(-f_i^\prime(y),1)}{\sqrt{1+f_i^\prime(y)^2}}.\]
Here we take the positive sign for $i=1$  and the negative one for $i=2$ so that this vector points into $\Omega$.
Note that here $z$ is nothing else but the  distance to $\Gamma_i$.
By the convexity of $\Gamma_i$, the Fermi coordinate is well defined and smooth in the open set $\Omega$.

Define the vector field
\[X^i:=\frac{\partial}{\partial y}+z\frac{\partial N_i}{\partial y}= \left(1+zH^i(y,0)\right)\frac{\partial}{\partial y}.\]
For any $z>0$, let
$\Gamma_{i,z}:=\{dist(x,\Gamma_i)=z\}$. The Euclidean metric restricted
to $\Gamma_{i,z}$ is denoted by $\lambda^i(y,z)dy^2$, where
\begin{equation}\label{metirc tensor}
\lambda^i(y,z)=<X^i(y,z),X^i(y,z)> =\left[1+zH^i(y,0)\right]^2\lambda^i(y,0).
\end{equation}
Here the metric coefficient of $\Gamma_i$ is
\begin{equation}\label{metirc tensor 0}
\lambda^i(y,0)=1+f_i^\prime(y)^2.
\end{equation}

The curvature of $\Gamma_{i,z}$ is
\begin{equation}\label{A(z)}
H^i(y,z)=\frac{H^i(y,0)}{1+zH^i(y,0)}.
\end{equation}

\subsection{Error in $z$} In this subsection we collect several estimates on the error of various terms in $z$.

By \eqref{normalized curvature 1}, $0\leq H^i(y,0)\lesssim \varepsilon$. Thus for any $z>0$, $0\leq H^i(y,z)\lesssim \varepsilon$. We also have
\begin{lem} In $(-R,R)$,
\begin{equation}\label{bound on 3rd derivatives}
|\partial_yH^i(y,0)|+|\partial_{yy}H^i(y,0)|\lesssim\varepsilon.
\end{equation}
\end{lem}
\begin{proof}
By the bound on $H^i(y,0)$ and Lemma \ref{convexity}, $u$ is close to one dimensional solutions near $\Gamma^i$. Hence $|\nabla u|\geq c(b)>0$ in $\{0<u<1-b\}$, where $c(b)$ is a constant depending only on $b$.
Hence $\nu=\nabla u/|\nabla u|$ is well defined and smooth in $\{0<u<1-b\}$.

By direct calculation, we have
\begin{equation}\label{Gauss equation}
\begin{cases}
-\mbox{div}\left(|\nabla u|^2\nabla\nu\right)=|\nabla u|^2|\nabla\nu|^2\nu, \quad & \mbox{in } \Omega,\\
\partial_N\nu=0, \quad &\mbox{on } \partial\Omega.
\end{cases}
\end{equation}
Here $N$ is the inward unit normal vector of $\partial\Omega$.

Recall that $B=\nabla \nu$. Differentiating \eqref{Gauss equation} gives the following Simons type equation
\begin{equation}\label{Simons}
\begin{cases}
-\mbox{div}\left(|\nabla u|^2\nabla B\right)=|\nabla u|^2|B|^2B+|\nabla u|^2\nabla|B|^2\otimes\nu +|B|^2\nabla|\nabla u|^2\otimes\nu & \ \ \ \ \ \ \ \ \ \ \ \ \
\\
\quad\quad\quad\quad\quad\quad\quad\quad\quad\quad\quad\quad\quad\quad\quad\quad +|\nabla u|^2\nabla^2\log|\nabla u|^2\cdot B, \quad & \mbox{in } \Omega,\\
\partial_NB=HB-B\cdot B, \quad &\mbox{on } \partial\Omega.
\end{cases}
\end{equation}
Here $H$ is the mean curvature of $\partial\Omega$ which is assumed to be positive and $\cdot$ denotes matrix multiplication.

For any $x\in\overline{\{0<u<1-2b\}}$,  $|\nabla u|^2$ has a positive lower and upper bound and it is uniformly continuous in $B_{2h(b)}(x)\cap \overline{\{0<u<1-2b\}}$. Because $\partial\Omega$ is bounded in $C^4$, by standard elliptic estimate,
\[ \sup_{B_{h(b)}\cap \{0<u<1-2b\}}|\nabla B| \lesssim \sup_{B_{2h(b)}\cap \{0<u<1-2b\}}|B|+\sup_{B_{2h(b)}\cap \{0<u<1-2b\}}|\mbox{div}\left(|\nabla u|^2\nabla B\right)|\lesssim\varepsilon.
\]
The bound on $|\nabla^2B|$ is obtained by bootstrapping elliptic estimates.
\end{proof}

By \eqref{A(z)},
\begin{equation}\label{error in z 1}
|H^i(y,z)-H^i(y,0)|\lesssim |z||H^i(y,0)|^2\lesssim\varepsilon^2|z|.
\end{equation}
Similarly, by \eqref{metirc tensor}, the error of metric tensors is
\begin{equation}\label{error in z 2}
|\lambda^i(y,z)-\lambda^i(y,0)|\lesssim \varepsilon|z|,
\end{equation}
\begin{equation}\label{error in z 3}
\Big|\frac{1}{\lambda^i(y,z)}-\frac{1}{\lambda^i(y,0)}\Big|\lesssim \varepsilon|z|,
\end{equation}

By \eqref{metirc tensor 0} and \eqref{bound on 3rd derivatives}, for any $0<z<\delta R$,
\begin{equation}\label{derivatives of metric tensor}
|\partial_y\lambda^i(y,z)|+\Big|\partial_y\frac{1}{\lambda^i(y,z)}\Big|\lesssim\varepsilon.
\end{equation}

\smallskip

The Euclidean Laplacian operator in Fermi coordinates has the form
\[\Delta_{\R^2}=\Delta_z+H^i(y,z)\partial_z+\partial_{zz},\]
where
\begin{eqnarray*}
\Delta_z=\frac{1}{\sqrt{\lambda^i(y,z)}}\frac{\partial}{\partial y}\left(\frac{1}{\sqrt{\lambda^i(y,z)}}\frac{\partial}{\partial y}\right)=\frac{1}{\lambda^i(y,z)}\frac{\partial^2}{\partial y^2}+b^i(y,z)\frac{\partial}{\partial y}
\end{eqnarray*}
with
\[b^i(y,z)=-\frac{1}{2}\frac{1}{\sqrt{\lambda^i(y,z)}}\frac{\partial}{\partial y}\left(\log\lambda^i(y,z)\right).\]

By \eqref{error in z 3} and \eqref{derivatives of metric tensor}, we get
\begin{lem}
For any function $\varphi\in C^2(-R,R)$,
\begin{equation}\label{error in z 5}
|\Delta_z\varphi(y)-\Delta_0\varphi(y)|\lesssim\varepsilon|z|\biggl(|\varphi^{\prime\prime}(y)|+|\varphi^\prime(y)|\biggr).
\end{equation}
\end{lem}

\subsection{Comparison of distance functions}

For each $i$, the local coordinates on $\Gamma_i$ is fixed to be the same one, $y\in (-R,R)$, which represents the point $(y,f_i(y))$. The distance to $\Gamma_i$ is denoted by $d_i$.
Given a point $X\in\Omega$, if $(y,f_i(y))$ is the nearest point on $\Gamma_i$ to $X$, we then define $\Pi_i(X):=y$.

As in \cite{Wang-Wei 2}, we have the following estimates on distances to $\Gamma_1$ and $\Gamma_2$.
\begin{lem}\label{comparison of distances}
For any $X\in \Omega$, if
$|d_1(X)|\leq K|\log\varepsilon|$ and $|d_2(X)|\leq K|\log\varepsilon|$, then we have
\begin{equation*}
\mbox{dist}_{\Gamma_1}\left(\Pi_1\circ\Pi_2(X),\Pi_1(X)\right)\leq C(K)\varepsilon^{1/2}|\log\varepsilon|^{3/2},
\end{equation*}
\begin{equation*}
|d_1\left(\Pi_2(X)\right)-d_2\left(\Pi_1(X)\right)|\leq C(K)\varepsilon^{1/2}|\log\varepsilon|^{3/2},
\end{equation*}
\begin{equation*}
|d_1(X)+d_2(X)-d_1\left(\Pi_2(X)\right)|\leq C(K)\varepsilon^{1/2}|\log\varepsilon|^{3/2},
\end{equation*}
\begin{equation*}
|d_1(X)+d_2(X)-d_2\left(\Pi_1(X)\right)|\leq C(K)\varepsilon^{1/2}|\log\varepsilon|^{3/2},
\end{equation*}
\begin{equation*}
1-\nabla d_1(X)\cdot \nabla d_2(X)\leq C(K)\varepsilon^{1/2}|\log\varepsilon|^{3/2}.
\end{equation*}
\end{lem}

\section{Proof in Case 1}\label{sec proof of case 1}
\setcounter{equation}{0}

In this section we consider Case 1 introduced in Section \ref{Sec
12} and prove Proposition \ref{prop curvature estimate 1}. In the  following we will work in the stretched version, that is after the rescaling $x\mapsto \varepsilon^{-1}x$ and do not write the dependence on $\varepsilon$ explicitly.
Hence now $\Omega=\{x_2>f(x_1)\}$, where $f$ is smooth concave function satisfying $f(0)=f^\prime(0)=0$.

Let $d(x)$ be the distance to $\partial\Omega$, which is a convex function in $\Omega$. Denote
\[\phi:=g(d)-u.\]
By the Modica inequality (see Proposition \ref{Modica inequality 2}), $\phi>0$ in $\Omega$. By the free boundary condition on $u$ and $g$, $\phi=|\nabla\phi|=0$ on $\partial\Omega$.

Written in the Fermi coordinates with respect to $\partial\Omega$, the equation for $\phi$ is
\begin{equation}\label{error equation 0}
\Delta_z\phi+H(y,z)\phi_z+\phi_{zz}=W^\prime\left(g(z)\right)-W^\prime\left(g(z)-\phi\right)+g^\prime(z)H(y,z).
\end{equation}
Here $H(y,z)$ is the mean curvature of the level set $\{d=z\}$, which is positive by the convexity of these level sets.

We first give an estimate on $\phi$.
\begin{lem}
There exists a constant $C$ such that
\begin{equation}\label{15.2}
\|\phi\|_{C^{2,1/2}(\overline{\Omega\cap Q_{R/2}})}\lesssim\varepsilon^{5/2}.
\end{equation}
\end{lem}
\begin{proof}
Multiplying \eqref{error equation 0} by $\phi$ and integrating in $z$ from $0$ to $R/2$ leads to
\begin{eqnarray}\label{15.1}
&&\frac{1}{2}\Delta_0\left(\int_0^{\frac{R}{2}}\phi(y,z)^2dz\right) \nonumber\\
&\geq&\int_0^{\frac{R}{2}}\Delta_z\phi(y,z)\phi(y,z)dz-O(\varepsilon)\int_0^{\frac{R}{2}}\biggl(|\phi_{yy}(y,z)|+|\phi_y(y,z)|\biggr)g^\prime(z)dz -O\left(e^{-c\varepsilon^{-1}}\right)\nonumber\\
&\geq&\int_0^{\frac{R}{2}}\biggl[\phi_z(y,z)^2+W^{\prime\prime}\left(g(z)\right)\phi(y,z)^2+O\left(\phi(y,z)^3\right)\biggr]dz\\
&-&\int_0^{\frac{R}{2}}H(y,z)\phi(y,z)\phi_z(y,z)dz-O(\varepsilon) \nonumber\\
&\geq&\frac{\mu}{2}\int_0^{\frac{R}{2}}\phi(y,z)^2dz-C\varepsilon.\nonumber
\end{eqnarray}
Here we have used Proposition \ref{prop nondegeneracy 2} in the last step.

By this differential inequality we deduce that
\[\int_0^{\frac{R}{2}}\phi(y,z)^2dz\leq C\varepsilon, \quad \forall y\in(-4R/5,4R/5).\]
Then by standard elliptic estimates and noting that $H+|\nabla H|\lesssim\varepsilon$, we get
\[\|\phi\|_{C^{2,1/2}(\overline{\Omega\cap Q_{4R/5}})}\lesssim\varepsilon^{1/2}.\]

Substituting this estimate into the first inequality in \eqref{15.1} an improvement is gained, that is $O(\varepsilon)$ can be replaced by $O(\varepsilon^{3/2})$. This procedure can be iterated and after finitely many
times we get \eqref{15.2}.
\end{proof}

Substituting the estimate \eqref{15.2} into the equation \eqref{15.1} and evaluating at $0$, after a rescaling we obtain
Proposition \ref{prop curvature estimate 1}.

\section{Proof in Case 2}\label{sec case 2}
\setcounter{equation}{0}

From this section we start to consider Case 2 introduced in Section \ref{Sec
12}. In the  following we will work in the stretched version, that is after the rescaling $x\mapsto \varepsilon^{-1}x$ and do not write the dependence on $\varepsilon$ explicitly.

 Now we have
\begin{itemize}
\item [{\bf (H1)}]
$u$ is a solution of
\eqref{equation 0.2} in $Q_R:=\{|x_1|<R, |x_2|<R\}$, where $R=\varepsilon^{-1}$.

\item[{\bf (H2)}] $u$ satisfies the Modica inequality. Hence $\partial\{u>0\}$ is
convex.
\item [{\bf (H3)}]There exists a concave function $f_1$ and a convex function $f_2$ with
$f_1(0)=f_1^{\prime}(0)=0$ and
$f_2>f_1$ in $(-R,R)$, such that
\[\{u>0\}\cap Q_R=\{(x_1,x_2):f_1(x_1)<x_2<f_2(x_1)\}.\]
\item [{\bf (H4)}] For $i=1,2$, $f_{i}^{\prime\prime}=O\left(\varepsilon\right)$.
\end{itemize}
Under these hypothesis we will show that
\begin{equation}\label{III.5.1}
  \big|f_1^{\prime\prime}(0)\big|\leq C\varepsilon^{\frac{8}{7}}.
\end{equation}
After a rescaling this gives Proposition \ref{prop curvature estimate 2}.

First we recall the following bound on intermediate distance between the two components of free boundaries,  $\Gamma_1$ and $\Gamma_2$.
\begin{lem}\label{lem ep apart}
For any $x_1\in(-R,R)$, $f_2(x_1)-f_1(x_1)\gg 1$.
\end{lem}
This can be proved by a contradiction argument, by using the fact that $u$ is close to one dimensional profiles near free boundaries.

The proof of this case is organized in the following way. In Section \ref{sec approximate solution} an approximate solution is constructed and it is used to give the equation satisfied by the error. In Section \ref{sec Toda system} the equations for $f_1$ and $f_2$ is obtained. In Section \ref{sec estimates on error function} some estimates on the error are given.  In Section \ref{sec reduction of stable} we reduce the stability condition of $u$ to a corresponding condition for $f_1$ or $f_2$, and then use information on stable solutions of the Liouville equation to finish the proof of Proposition \ref{prop curvature estimate 2}.

\section{An approximate solution}\label{sec approximate solution}
\setcounter{equation}{0}

Let $\eta$ be a smooth, even function defined on $\R$ satisfying $\eta\equiv 1$ in $(-1,1)$, $\eta\equiv 0$ outside $(-2,2)$, $|\eta^\prime|^2+|\eta^{\prime\prime}|\leq 64$. Define $\eta_1:=\eta\circ d_2$ and $\eta_2:=\eta\circ d_1$. These two cut-off functions are introduced to localize the effects of each component of transition layers, so that only interactions between adjacent components are taken into consideration.

In the Fermi coordinates with respect to $\Gamma_i$ ($i=1,2$), given two functions $h_i\in C(-R,R)$, let
\[g_i(y,z):=g(z-h_i(y)),\]
and
\[g(y,z;h_1,h_2):=\left[g_1(y,z)\left(1-\eta_1(y,z)\right)+\eta_1(y,z)\right]+\left[g_2(y,z)\left(1-\eta_2(y,z)\right)+\eta_2(y,z)\right]-1.\]
For simplicity of notations, we denote in the following
\[g_i^\prime:=\partial_z g_i, \quad g_i^{\prime\prime}:=\partial_{zz}g_i, \quad \cdots.\]

\begin{prop}
There exist two functions $h_1, h_2\in C(-R,R)$ so that the following holds. For each $i=1,2$, in the Fermi coordinates with respect to $\Gamma_i$,
\begin{equation}\label{orthogonal condition}
\int_{0}^{+\infty}\left[u(y,z)-g(y,z;h_1(y),h_2(y))\right]g_i^\prime(y,z)\left(1-\eta_i(y,z)\right)dz=0, \quad\forall y\in(-R,R).
\end{equation}
\end{prop}
\begin{proof}
Define the map $F$ from $\mathcal{X}:=C(-R,R)\oplus C(-R,R)$ to itself as
\[F_i(h_1,h_2):=\int_{0}^{+\infty}\left[u(y,z)-g(y,z;h_1(y),h_2(y))\right]g_i^\prime(y,z)\left(1-\eta_i(y,z)\right)dz, \quad i=1,2.\]
Clearly $F$ is a $C^1$ map.

By direct calculations we obtain (and a similar expressions for the other component of $DF$)
\[D_1F_1(h_1,h_2)(\delta h)_1=(\delta h)_1(y)\int_{0}^{+\infty}\left[|g_i^\prime|^2+\left(u-g(\cdot;h_1,h_2)\right)g_i^{\prime\prime}\right]\left(1-\eta_i\right)dz.\]
By Lemma \ref{lem ep apart}, there exists a $\delta>0$ such that once $\|(h_1,h_2)\|_{\mathcal{X}}<\delta$, $DF(h_1,h_2)$ is an isomorphism with $\|DF^{-1}\|$ bounded by a constant independent of $\varepsilon$ and $\delta$.

Using Lemma \ref{lem ep apart} again we see  $\|F(0,0)\|_{\mathcal{X}}\ll 1$. Therefore by the inverse function theorem we conclude the proof.
\end{proof}
In the following $(h_1, h_2)$ always denotes the functions defined in this proposition. Denote
\[g_\ast(y,z);=g(y,z;h_1(y),h_2(y)), \quad \phi:=g_\ast-\phi.\]
\begin{rmk}
  By differentiating \eqref{orthogonal condition}, we see for any $k\geq 1$, $\|h_1\|_{C^k(-R,R)}+\|h_2\|_{C^k(-R,R)}\ll 1$.
\end{rmk}

The equation for $\phi$ written in the Fermi coordinates with respect to $\Gamma_1$ reads as
\begin{eqnarray}\label{error equation}
\Delta_z\phi+H^1(y,z)\phi_z+\phi_{zz}&=&g_1^{\prime\prime}\left(1-\eta_1\right)  \nonumber\\
&+&g_1^\prime \left(1-\eta_1\right)\left[H^1(y,z)-\Delta_zh_1\right]+g_1^{\prime\prime}\left(1-\eta_1\right)|\nabla_zh_1|^2 \nonumber\\
&+&\left(1-g_1\right) \left[\eta_1^\prime H^2(y,z)+\eta_1^{\prime\prime}\right]-2g_1^\prime\eta_1^\prime d_{2,z} \nonumber\\
&+&g_2^{\prime\prime}\left(1-\eta_2\right) \\
&+&g_2^\prime\left(1-\eta_2\right)\mathcal{R}_1+g_2^{\prime\prime}\left(1-\eta_2\right)\mathcal{R}_2 \nonumber\\
&+&\left(1-g_2\right) \left[\eta^\prime(z)H^1(y,z)+\eta^{\prime\prime}(z)\right]-2g_2^\prime\eta^\prime(z) d_{2,z}  \nonumber\\
&-&W^\prime(g_\ast-\phi).\nonumber
\end{eqnarray}
Here written in the Fermi coordinates with respect to $\Gamma_2$, we have
\[\mathcal{R}_1(y,z)=H^2(y,z)-\Delta_zh_2(y), \quad \mathcal{R}_2(y,z)=|\nabla_zh_2(y)|^2.\]

The boundary condition on $\Gamma_1$ reads as
\begin{equation}\label{boundary condition}
\phi(y,0)=g(-h_1(y)), \quad \partial_z\phi(y,0)=g^\prime(-h_1(y))-g^\prime(0).
\end{equation}
This can be rewritten as a nonlinear Robin boundary condition
\begin{equation}\label{Robin}
\partial_z\phi(y,0)=f(\phi(y,0)),
\end{equation}
where $f(t):=g^\prime\biggl(-g^{-1}(-t)\biggr)-g^\prime(0)$.

By \eqref{boundary condition} we obtain
\begin{equation}\label{boundary value bound}
|h_1(y)|\lesssim |\phi(y,0)| \quad \mbox{and } \quad |\phi(y,0)|+|\partial_z\phi(y,0)|\lesssim |h_1(y)|.
\end{equation}

In the following we denote
\begin{eqnarray*}
W^\prime(g_\ast-\phi)&=&W^\prime(g_\ast)-W^{\prime\prime}(g_\ast)\phi+\mathcal{R}(\phi),
\end{eqnarray*}
where $\mathcal{R}(\phi)=O(\phi^2)$.

The interaction between $g_1$ and $g_2$ is contained in the following term
\[\mathcal{I}:=g_1^{\prime\prime}\left(1-\eta_1\right)+g_2^{\prime\prime}\left(1-\eta_2\right)-W^\prime(g_\ast).\]
We have the following estimates on $\mathcal{I}$. In the following we denote
\[D(y):=\min\biggl\{\mbox{dist}\bigl((y,f_1(y)),\Gamma_2\bigr),\mbox{dist}\bigl((y,f_1(y)),\Gamma_2\bigr)\biggr\}.\]
\begin{lem}\label{lem interaction terms}
In $\{0<d_1<d_2+4\}$,
\[\mathcal{I}=\left[W^{\prime\prime}(g_1)-2\right]\left[1-g(d_2)\right]\left[1-\eta_2\right]+O\left(e^{-2\sqrt{2}d_2}\right).\]
\[\big|\mathcal{I}(y,z)\big|\lesssim \varepsilon^2+e^{-\sqrt{2}D(y)}.\]
\[\|\mathcal{I}\|_{C^\theta(B_1(y,z))}\lesssim\varepsilon^2+\sup_{(y-2,y+2)}e^{-\sqrt{2}D(y)}.\]
\end{lem}
Similar results hold in $\{0<d_2<d_1+4\}$. These can be proved as in \cite{Wang-Wei 2}.

\section{A Liouville equation}\label{sec Toda system}
\setcounter{equation}{0}

In this section we prove that
\begin{lem}\label{lem Liouville eqn}
For $y\in(-R,R)$, it holds that
\begin{equation}\label{Toda system}
H^1(y,0)-\Delta_0h_1(y)
=\frac{4 A^2}{\sigma_0}e^{-\sqrt{2}D(y)}+\widetilde{E}(y),
\end{equation}
where
\[\sup_{(-r,r)}|\widetilde{E}(y)|\lesssim \varepsilon^2+\sup_{(-r-1,r+1)}e^{-\frac{4\sqrt{2}}{3}D(y)}+\varepsilon^{\frac{1}{6}}\sup_{(-r-1,r+1)}e^{-\sqrt{2}D(y)}
+\|\phi\|_{C^1(\overline{\Omega\cap Q_{r+1}})}^2.\]

\end{lem}
\begin{proof}

In the Fermi coordinates with respect to $\Gamma_1$, multiplying \eqref{error equation} by $g_1^\prime\left(1-\eta_1\right)$ and integrating in $z$, after integrating by parts in $z$ and some calculations as in \cite{Wang-Wei 2} we obtain \eqref{Toda system}. Here we only discuss several points which are different from \cite{Wang-Wei 2}.

\begin{enumerate}

\item[(1)] In $\{\eta^\prime(z)\neq0\}$, if $g^\prime(d_2)\geq\varepsilon^3$, by Lemma \ref{comparison of distances} we have $\eta(d_2)\equiv 0$,
\[g^\prime\left(d_2(y,z)\right)=\bigl(1+O(\varepsilon^{1/3})\bigr)g^\prime\bigl(D(y)-z\bigr),\]
and
\[\nabla d_1\cdot\nabla d_2=1+O(\varepsilon^{1/3}).\]
 Therefore
\begin{eqnarray*}
&&\int_0^{+\infty}g^\prime(z)g^\prime(d_2)\eta^\prime(z)\nabla d_1\cdot\nabla d_2 dz\\
&=&\int_0^{+\infty}g^\prime(z)\eta^\prime(z)g^\prime\bigl(D(y)-z\bigr)dz+O\left(\varepsilon^{1/3}\right)\int_1^2g^\prime\bigl(D(y)-z\bigr)g^\prime(z)dz
+O\left(\varepsilon^2\right)\\
&=&\int_0^{+\infty}g^\prime(z)\eta^\prime(z)g^\prime\bigl(D(y)-z\bigr)dz+O\left(\varepsilon^{1/3}e^{-\sqrt{2}D(y)}\right)+O\left(\varepsilon^2\right).
\end{eqnarray*}

\item[(2)] In $\{\eta^{\prime\prime}(z)\neq0\}$, $\eta(d_2)=0$. Hence by Lemma \ref{comparison of distances}, we get
\begin{eqnarray*}
\int_0^{+\infty}\bigl[1-g(d_2)\bigr]\eta^{\prime\prime}(z)g^\prime(z)dz=\int_0^{+\infty}\biggl[1-g\bigl(D_2(y)-z\bigr)\biggr]\eta^{\prime\prime}(z)g^\prime(z)dz+O\left(\varepsilon^{\frac{1}{3}}e^{-\sqrt{2}D(y)}\right).
\end{eqnarray*}

\item[(3)] Similar to the derivation in \cite{Wang-Wei 2}, we have
\[\int_0^{+\infty}\mathcal{I}g^\prime(z)\left[1-\eta(z)\right]dz=-4A^2e^{-\sqrt{2}D(y)}+O\left(e^{-\frac{4\sqrt{2}}{3}D(y)}\right)+O(\varepsilon^2).\]

Here we note that, for any $L>2$,
\begin{eqnarray*}
&&\int_0^L\bigl[W^{\prime\prime}(g(z))-2\bigr]g^\prime(z)\bigl[1-g(D(y)-z)\bigr]\bigl[1-\eta(z)\bigr]\\
&=&g^{\prime\prime}(L)\bigl[1-g\bigl(D(y)-L\bigr)\bigr]-g^\prime(L)g^\prime\bigl(D(y)-L\bigr)\\
&-&2\int_0^Lg^\prime(z)\bigl[1-g\bigl(D(y)-z\bigr)\bigr]\bigl[1-\eta(z)\bigr]-\int_0^Lg^\prime(z) g^{\prime\prime} \bigl(D(y)-z\bigr) \bigl[1-\eta(z)\bigr]\\
&+&2\int_0^Lg^\prime(z)g^\prime\bigl(D(y)-z\bigr)\eta^\prime(z) -\int_0^Lg^\prime(z)\bigl[1-g\bigl(D(y)-z\bigr)\bigr]\eta^{\prime\prime}(z)\\
&=:&I+II+III+IV+V+VI.
\end{eqnarray*}
Note that in the above equalities there is no boundary term at $z=0$ because $\eta(z)\equiv 1$ for $z\in[0,1)$.

The first two terms $I$ and $II$ can be computed by the asymptotics of $g$ at infinity.

Because $D(y)$ is large, for $z\in(0,L)$,
\[1-g(D(y)-z)=Ae^{-\sqrt{2}(D(y)-z)}+O\left(e^{-2\sqrt{2}D(y)}\right), \]
and
\[g^{\prime\prime} (D(y)-z) =-2Ae^{-\sqrt{2}(D(y)-z)}+O\left(e^{-2\sqrt{2}D(y)}\right). \]
Therefore the $III$ and $IV$ terms cancel with each other with a remainder term of the order $O\left(D(y)e^{-2\sqrt{2}D(y)}\right)
=O\left(e^{-\frac{3}{2}\sqrt{2}D(y)}\right)$.

The $V$ and $VI$ terms cancel with the main order terms in (1) and (2). \qedhere
\end{enumerate}
\end{proof}

\section{Estimates on $\phi$}\label{sec estimates on error function}
\setcounter{equation}{0}

In this section we prove
\begin{prop}\label{prop schauder}
There exist two constants $K\gg 1$ and $C(K)$ such that
\begin{eqnarray*}
&&\|\phi\|_{C^{2,\theta}(\overline{\Omega\cap Q_r})}+\sum_{i=1,2}\|H^i-\Delta_0h_i\|_{C^\theta(-r,r)}\\
&\leq& C(K)\varepsilon^2+C(K)\sup_{(-r-K|\log\varepsilon|,r+K|\log\varepsilon|)}e^{-\sqrt{2}D}.
\end{eqnarray*}
\end{prop}

Fix a large constant $L$ and let
\[\mathcal{N}:=\{d_1>L, d_2>L\}, \quad \mbox{and} \quad  \mathcal{M}^i:=\{0<d_i<2L\}, \quad i=1,2,\]
and
\[\mathcal{N}_r:=\mathcal{N}\cap Q_r, \quad \mbox{and} \quad  \mathcal{M}_r^i:=\mathcal{M}^i\cap Q_r, \quad i=1,2,\]
We consider the estimate of $\phi$ in these subdomains separately.

\subsection{Estimates in $\mathcal{N}$}
In $\mathcal{N}$, $\eta_1=\eta_2=0$. Thus \eqref{error equation} reads as
\begin{eqnarray}\label{error equation 2}
\Delta\phi&=&W^\prime(g_1)+W^\prime(g_2)-W^\prime(g_1+g_2-1-\phi) \nonumber\\
&+&g_1^\prime \left[H^1(y,z)-\Delta_zh_1\right]+g_1^{\prime\prime}|\nabla_zh_1|^2 +g_2^\prime\mathcal{R}_1+g_2^{\prime\prime}\mathcal{R}_2 \\
&=&W^{\prime\prime}(g_\ast)\phi+\mathcal{R}(\phi)+\mathcal{I}+g^\prime(d_1)\Delta d_1+g^\prime(d_2)\Delta d_2. \nonumber
\end{eqnarray}
In $\mathcal{N}$ the interaction term becomes
\[\mathcal{I}=W^\prime(g_1)+W^\prime(g_2)-W^\prime(g_1+g_2-1).\]
By Lemma \ref{lem interaction terms}, for any $X\in\mathcal{N}$ we have
\[\|\mathcal{I}\|_{C^\theta(B_1(X))}\lesssim \varepsilon^2+e^{-\sqrt{2}D_1(\Pi_1(X))}.\]

Using this estimate and standard Schauder estimates, we obtain a  constant $\sigma(L)\ll 1$ such that
\begin{eqnarray}\label{Schauder 1}
\|\phi\|_{C^{2,\theta}(\overline{\mathcal{N}_r})}&\leq& C\varepsilon^2+C\sum_{i=1}^{2}\|H_i-\Delta_0h_i\|_{C^\theta(-r-L,r+L)}+C\sup_{(-r-L,r+L)}e^{-\sqrt{2}D}  \nonumber\\
&+&\sigma(L)\|\phi\|_{C^{2,\theta}(\overline{\Omega\cap Q_{r+L}})}.
\end{eqnarray}

\subsection{Estimates in $\mathcal{M}^1$}
In $\mathcal{M}^1$, $\eta(d_2)\equiv 0$ and $g_\ast=g(d_1)+\left[1-g(d_2)\right]\left[1-\eta(d_1)\right]$. Hence using the Fermi coordinates with respect to $\Gamma_1$ the equation for $\phi$ reads as
\begin{equation}\label{error equation 3}
\Delta_z\phi+H^1(y,z)\phi_z+\phi_{zz}=W^{\prime\prime}(g_1)\phi+\mathcal{R}(\phi)+g_1^\prime\left(H^1-\Delta_0h_1\right)+E^1.
\end{equation}
Here the error term $E^1$ can be estimated in the following way:
\begin{equation}
\|E^1\|_{C^\theta(\mathcal{M}^1(r))}\lesssim \varepsilon^2+\sup_{(-r-1,r+1)}e^{-\sqrt{2}D}.
\end{equation}

Recall that on $\Gamma_1$, $\phi$ satisfies the nonlinear Robin condition \eqref{Robin}.

Then by \eqref{orthogonal condition} and standard Schauder estimates, proceeding as in \cite{Wang-Wei 2} we obtain
\begin{eqnarray}\label{Schauder 2}
\|\phi\|_{C^{2,\theta}(\overline{\mathcal{M}^1_r})}&\leq
& C\varepsilon^2+C\|H^1-\Delta_0h_1\|_{C^\theta(-r-L,r+L)}+C\sup_{(-r-L,r+L)}e^{-\sqrt{2}D}  \nonumber\\ &+&\sigma(L)\|\phi\|_{C^{2,\theta}(\overline{\Omega\cap Q_{r+L}})}.
\end{eqnarray}

Combining \eqref{Schauder 1}, \eqref{Schauder 2} and a similar estimate in $\mathcal{M}^2_r$ we obtain
\begin{eqnarray}\label{III.8.1}
\|\phi\|_{C^{2,\theta}(\overline{\Omega\cap Q_r})}&\leq& C\varepsilon^2+C\sum_{i=1,2}\|H^i-\Delta_0h_i\|_{C^\theta(-r-L,r+L)}+C\sup_{(-r-L,r+L)}e^{-\sqrt{2}D}  \\
&+&\sigma(L)\|\phi\|_{C^{2,\theta}(\overline{\Omega\cap Q_{r+L}})}.
\end{eqnarray}

By \eqref{Toda system}, we also have $C^\theta$ bounds on $H^i-\Delta_0h_i$, see \cite{Wang-Wei 2} for more details.
Substituting \eqref{Toda system} into \eqref{III.8.1} gives
\[
\|\phi\|_{C^{2,\theta}(\overline{\Omega\cap Q_r})}\leq C\varepsilon^2+C\sup_{(-r-L,r+L)}e^{-\sqrt{2}D}
+\sigma(L)\|\phi\|_{C^{2,\theta}(\overline{\Omega\cap Q_{r+L}})}.
\]
An iteration of this estimate in $O(|\log\varepsilon|)$ steps and using \eqref{Toda system} again we finish the proof of
Proposition \ref{prop schauder}.

\subsection{Improved estimates on $\phi_y$}\label{sec improved estimates on horizontal derivatives}
\setcounter{equation}{0}
In the Fermi coordinates with respect to $\Gamma_1$,
differentiating \eqref{error equation} in $y$ gives the equation for $\phi_y:=\frac{\partial\phi}{\partial y}$:
\begin{equation}\label{error equation diff}
\Delta_z\phi_y+\partial_{zz}\phi_y=W^{\prime\prime}\left(g_1\right)\phi_y+g_1^\prime \left[H_1^\prime(y)-\Delta_0 h_1^\prime(y)\right]+\widetilde{E}_1,
\end{equation}
where
\begin{eqnarray*}
\|\widetilde{E}_1\|_{L^\infty(\overline{\Omega\cap Q_r})}&\lesssim& \varepsilon^2+\sup_{(-r-1,r+1)}e^{-2\sqrt{2}D}+\|\phi\|_{C^2(\overline{\Omega\cap Q_{r+1}})}^2\\
&+&\varepsilon^{1/5}\sup_{(-r-1,r+1)}e^{-\sqrt{2}D}+\|\phi\|_{C^2(\overline{\Omega\cap Q_{r+1}})}\sup_{(-r-1,r+1)}e^{-\frac{\sqrt{2}}{2}D}.
\end{eqnarray*}
For more details see \cite{Wang-Wei 2}.

Differentiating \eqref{boundary condition} in $y$ gives the Robin boundary condition of $\phi_y$ on $\Gamma_1$ (and a similar one on $\Gamma_2$)
\begin{equation}\label{Robin 2}
\partial_z\phi_y+\frac{g^{\prime\prime}(-h_1(y))}{g^\prime(-h_1(y))}\phi_y=0, \quad \mbox{on } \Gamma_1.
\end{equation}

Then similar to the discussion in the previous section, we obtain
\begin{eqnarray*}
\|\phi_y\|_{C^{1,\theta}(\overline{\Omega\cap Q_r})}&\leq& C(K)\varepsilon^2+C(K)\sup_{(-r-K|\log\varepsilon|,r+K|\log\varepsilon|)}e^{-\frac{3}{2}\sqrt{2}D}+\varepsilon^{\frac{1}{5}}\sup_{(-r-K|\log\varepsilon|,r+K|\log\varepsilon|)}e^{-\sqrt{2}D}.
\end{eqnarray*}

Combining this estimate with Lemma \ref{lem Liouville eqn} and the bound on $H^i(y,0)$, we obtain
\begin{coro}\label{a first bound}
There exists a constant $C$ such that for any $y\in(-R/2,R/2)$,
\[e^{-\sqrt{2}D(y)}\leq C\varepsilon.\]
\end{coro}

\section{Reduction of the stability condition}\label{sec reduction of stable}
\setcounter{equation}{0}

 In this section we reduce the stability condition of $u$ to a corresponding condition for the Liouville equation  \eqref{Toda system}.

The curve $\{d_1=d_2\}$ can be represented by the graph  $\{z=\rho(y)\}$ in the Fermi coordinates with respect to $\Gamma_1$, where $\rho(y)$ is a smooth function of $y$.

Fix a smooth function $\eta_3$ defined on $\R$ satisfying $\eta_3\equiv 1$ in $(-\infty,0)$, $\eta_3\equiv 0$ in $(1,+\infty)$ and $|\eta_3^\prime|+|\eta_3^{\prime\prime}|\leq 16$. Take a large constant $L$ and define
\[\chi(y,z):=\eta_3\left(\frac{z-\rho(y)}{L}\right).\]
Clearly we have $\chi\equiv 1$ in $\mathcal{M}^1$, $\chi\equiv 0$ in $\{z>\rho(y)+L\}$. Moreover,
 $|\nabla\chi|\lesssim L^{-1}$, $|\nabla^2\chi|\lesssim L^{-2}$.

For any $\psi\in C_0^\infty(-5R/6,5R/6)$, let
\[\varphi(y,z):=\psi(y)g_1^\prime(y,z)\chi(y,z).\]
The stability condition for $u$ implies that
\begin{equation}\label{stability to be reduced}
\int_{Q_{5R/6}}\left[|\nabla\varphi|^2+W^{\prime\prime}(u)\varphi^2\right]\geq 2W(0)\int_{\Gamma_1}\left[H^1(y)-\frac{W^\prime(0)}{\sqrt{2W(0)}}\right]g^\prime(-h_1(y))^2\psi(y)^2 d\mathcal{H}^1.
\end{equation}
The purpose of this section is to rewrite this inequality as a stability condition for the Liouville equation \eqref{Toda system}.
\begin{prop}\label{prop stable for Toda}
If $L$ is large enough, we have
\begin{equation}\label{reduction of stability}
\int_{-5R/6}^{5R/6}\psi(y)^2e^{-\sqrt{2}D(y)}dy\leq C\int_{-5R/6}^{5R/6}\psi^\prime(y)^2dy+C\varepsilon^{\frac{4}{3}}\int_{-5R/6}^{5R/6}\psi(y)^2dy.
\end{equation}
\end{prop}

In the Fermi coordinates with respect to $\Gamma_1$, we have
\[|\nabla\varphi(y,z)|^2=\Big|\frac{\partial\varphi}{\partial z}(y,z)\Big|^2+\lambda^1(y,z)\Big|\frac{\partial\varphi}{\partial y}(y,z)\Big|^2.\]
We discuss these two terms separately.

\subsection{The horizontal part}\label{subsec 9.1}
We have
\[\frac{\partial\varphi}{\partial y}=\psi^\prime(y)g_1^\prime\chi+\psi(y) g_1^\prime\chi_y-\psi(y)h_1^\prime(y)g_1^{\prime\prime}\chi.\]
Here and in the following $\chi_y$ denotes the partial derivative $\frac{\partial\chi}{\partial y}$.

Since $c\leq\lambda^1(y,z)\leq C$,
\begin{eqnarray*}
&&\int_{Q_{5R/6}}\Big|\frac{\partial\varphi}{\partial y}(y,z)\Big|^2\lambda^1(y,z)dzdy\\
&\lesssim&\int_{-5R/6}^{5R/6}\int_{-\delta R}^{\delta R}|\psi^\prime|^2|g^\prime|^2\chi^2+\psi^2|g^\prime|^2\chi_y^2+\psi^2|h_1^\prime|^2|g_1^{\prime\prime}|^2\chi^2\\
&\lesssim&\int_{-5R/6}^{5R/6}\psi^\prime(y)^2dy+\frac{1}{L}\int_{-5R/6}^{5R/6}\psi(y)^2e^{-2\sqrt{2}\rho(y)}dy+\varepsilon^{\frac{3}{2}}\int_{-5R/6}^{5R/6}\psi(y)^2dy.
\end{eqnarray*}

Here the last term follows from the following three facts:
\begin{itemize}
\item in $\{\chi_y\neq0\}$, which is exactly $\{\rho(y)<z<\rho(y)+L\}$, $|\chi_y|\lesssim L^{-1}$;

\item in $\{\rho(y)<z<\rho(y)+L\}$, $g^\prime\lesssim e^{-\sqrt{2}\rho(y)}$;

\item by \eqref{boundary value bound}, Proposition \ref{prop schauder} and Corollary \ref{a first bound}, for any $y\in(-5R/6,5R/6)$,
$h_1^\prime(y)^2\lesssim \varepsilon^{3/2}$.
\end{itemize}

\subsection{The vertical part}
As before we have
\[\varphi_z=\psi g_1^{\prime\prime}\chi+\psi g_1^\prime \chi_z.\]
Thus by a direct expansion and integrating by parts, we have
\begin{eqnarray*}
\int_{Q_{5R/6}}\varphi_z^2\lambda^1(y,z)dzdy&=&\int_{-5R/6}^{5R/6}\psi(y)^2\left[\int_0^{+\infty}|g_1^{\prime\prime}|^2\chi^2\lambda^1+2g_1^\prime g_1^{\prime\prime}\chi\chi_z\lambda^1+
|g^\prime|^2\chi_z^2\lambda dz\right]dy\\
&=&-\int_{-5R/6}^{5R/6}\psi(y)^2 g^\prime\left(-h_1(y)\right)g^{\prime\prime}\left(-h_1(y)\right)\lambda^1(y,0)dy\\
&-&\int_{-5R/6}^{5R/6}\psi(y)^2\left[\int_0^{+\infty}W^{\prime\prime}(g_1)|g_1^\prime|^2\chi^2\lambda^1 dz\right]dy\\
&&-\int_{-5R/6}^{5R/6}\psi(y)^2\left[\int_0^{+\infty}g_1^\prime g_1^{\prime\prime}\chi^2\lambda^1_z-|g_1^\prime|^2\chi_z^2\lambda^1 dz\right]dy.
\end{eqnarray*}

These terms can be estimated exactly as in \cite{Wang-Wei 2}, except the following one
\begin{eqnarray*}
&&-\int_{-5R/6}^{5R/6}\psi(y)^2\left[\int_0^{+\infty}g_1^\prime g_1^{\prime\prime}\chi^2\lambda^1_z dz\right]dy\\
&=&g^\prime(0)^2\int_{-5R/6}^{5R/6}H^1(y)\psi(y)^2\lambda^1(y,0)dy\\
&+&\int_{-5R/6}^{5R/6}\psi(y)^2\left[\int_0^{+\infty}|g_1^\prime|^2\chi \chi_z\lambda^1_z dz\right]dy
+\frac{1}{2}\int_{-5R/6}^{5R/6}\psi(y)^2\left[\int_0^{+\infty}|g_1^\prime|^2\chi^2\lambda^1_{zz}dz\right]dy,
\end{eqnarray*}
where we have used the fact that $\lambda^1_z(y,0)=2H^1(y,0)$ and an integration by parts in $z$. 
The last two integrals are estimated in the following way.
\begin{itemize}
\item[(i)] As in Subsection \ref{subsec 9.1}  and  by the estimate
\[\lambda^1_z=-2\lambda^1(y,0)H^1(y)\left(1-zH^1(y)\right)=O(\varepsilon),\]
we get
\[\int_{-5R/6}^{5R/6}\psi(y)^2\left[\int_0^{+\infty}|g_1^\prime|^2\chi \chi_z\lambda^1_z dz\right]dy=O(\varepsilon)\int_{-5R/6}^{5R/6}\psi(y)^2e^{-2\sqrt{2}\rho(y)}dy.\]
\item[(ii)] Because
\[\lambda_{zz}=2H^1(y)^2\lambda^1(y,0)=O(\varepsilon^2),\]
we obtain
\[\int_{-5R/6}^{5R/6}\psi(y)^2\left[\int_0^{+\infty}|g_1^\prime|^2\chi^2\lambda^1_{zz}dz\right]dy=O(\varepsilon^2)\int_{-5R/6}^{5R/6}\psi(y)^2dy.\]
\end{itemize}

Therefore we get
\begin{eqnarray*}
\int_{Q_{5R/6}}\varphi_z^2\lambda^1(y,z)dzdy&=&-\int_{-5R/6}^{5R/6}\psi(y)^2\left[\int_0^{+\infty}W^{\prime\prime}(g_1)|g_1^\prime|^2\chi^2\lambda^1 dz\right]dy\\
&+&\int_{\Gamma_1}\left[H^1(y)-\frac{g^{\prime\prime}(-h_1(y))}{g^\prime(-h_1(y))}\right]\varphi(y,0)^2 d\mathcal{H}^1\\
&+&O(\varepsilon^2)\int\psi(y)^2dy+O\left(\frac{1}{L}+\varepsilon\right)\int\psi(y)^2e^{-2\sqrt{2}\rho(y)}dy.
\end{eqnarray*}

By the estimate on $\phi$ in Section \ref{sec estimates on error function} and \eqref{boundary value bound}, we get
\begin{eqnarray*}
\frac{g^{\prime\prime}(-h_1(y))}{g^\prime(-h_1(y))}-\frac{W^\prime(0)}{\sqrt{2W(0)}}&=&\left[\frac{g^{\prime\prime}(0)^2}{g^\prime(0)^2}-W^{\prime\prime}(0)\right]h_1(y)+O\left(|h_1(y)^2\right)\\
&=&\left[\frac{g^{\prime\prime}(0)^2}{g^\prime(0)^2}-W^{\prime\prime}(0)\right]h_1(y)+O\left(\varepsilon^{\frac{3}{2}}\right).
\end{eqnarray*}
Now the stability condition for $u$ is transformed into
\begin{eqnarray}\label{9.3.1}
0&\leq&C\int_{-5R/6}^{5R/6}\psi^\prime(y)^2dy +C\left(\frac{1}{L}+\varepsilon\right)\int_{-5R/6}^{5R/6}\psi(y)^2e^{-2\sqrt{2}\rho(y)}dy\nonumber\\
&-&\left[g^{\prime\prime}(0)^2-W^{\prime\prime}(0)g^\prime(0)^2\right]\int_{-5R/6}^{5R/6}h_1(y)\psi(y)^2 \lambda^1(y,0)dy\\
&+&\int_{-5R/6}^{5R/6}\psi(y)^2\left[\int_0^{+\infty}\left(W^{\prime\prime}(u)-W^{\prime\prime}(g_1)\right)|g_1^\prime|^2\chi^2\lambda^1 dz\right]dy. \nonumber
\end{eqnarray}
It remains to rewrite the last integral.

\subsection{The interaction part}
Differentiating \eqref{error equation} in $z$ leads to
\begin{eqnarray}\label{error equation 3}
&&\frac{\partial}{\partial z}\Delta_z\phi+\frac{\partial}{\partial z}\left(H^1(y,z)\partial_z\phi\right)+\partial_{zzz}\phi \nonumber\\
&=&W^{\prime\prime}(g_1)g_1^\prime-W^{\prime\prime}(u)\left[g_1^\prime-\phi_z+g_2^\prime d_{2,z}\left(1-\eta(z)\right)+\left(1-g_2\right)\eta^\prime(z)\right] \nonumber\\
&+&W^{\prime\prime}(g_2)g_2^\prime d_{2,z}\left(1-\eta(z)\right)-g_2^{\prime\prime}\eta^\prime(z)  \nonumber\\
&+&\frac{\partial}{\partial z}\left[g_1^\prime\left(H^1-\Delta_zh^1\right)\right]+\frac{\partial}{\partial z}\left(g^{\prime\prime}|\nabla_zh_1|^2\right)\\
&+&\frac{\partial}{\partial z}\left[g_2^\prime\left(1-\eta(z)\right)\mathcal{R}_1\right]+\frac{\partial}{\partial z}\left[g_2^{\prime\prime}\left(1-\eta(z)\right)\mathcal{R}_2\right]\nonumber\\
&+&\frac{\partial}{\partial z}\left[\left(1-g_2\right)\left(\eta^\prime(z)H^1+\eta^{\prime\prime}(z)\right)\right]-2\frac{\partial}{\partial z}\left[g_2^\prime d_{2,z}\eta^\prime(z)\right]. \nonumber
\end{eqnarray}

Multiply this equation by $\psi^2g_1^\prime\chi^2\lambda^1$ and then integrate in $y$ and $z$.
As in \cite{Wang-Wei 2}, the main order term is
\begin{eqnarray*}
&&\int_{-\frac{5R}{6}}^{\frac{5R}{6}}\psi(y)^2\left[\int_0^{+\infty}\left(W^{\prime\prime}(u)-W^{\prime\prime}(1)\right)g_2^\prime g_1^\prime \left(1-\eta_2\right)\chi^2\lambda^1dz\right]dy\\
&=&-4A^2\int_{-5R/6}^{5R/6}\psi(y)^2e^{-\sqrt{2}D(y)}\lambda^1(y,0) dy+h.o.t.
\end{eqnarray*}

During integrating by parts, the term $\partial_{zzz}\phi$ gives a boundary integral in the form
\begin{eqnarray*}
&&\int_{-\frac{5R}{6}}^{\frac{5R}{6}}\psi(y)^2\left[\phi_{zz}(y,0)g_1^\prime(y,0)\lambda^1(y,0)-\phi_z(y,0)g_1^{\prime\prime}(y,0)\lambda^1(y,0)-\phi_z(y,0)g_1^{\prime\prime}(y,0)\partial_z\lambda^1(y,0)\right]dy\\
&=&\left[g^{\prime\prime}(0)^2-W^{\prime\prime}(0)g^\prime(0)^2\right]\int_{-5R/6}^{5R/6}h_1(y)\psi(y)^2 \lambda^1(y,0)dy-g^\prime(0)^2\int_{-5R/6}^{5R/6}H^1(y)\psi(y)^2 \lambda^1(y,0)dy\\
&+&O\left(\varepsilon^{\frac{3}{2}}\right)\int_{-\frac{5R}{6}}^{\frac{5R}{6}}\psi(y)^2dy.
\end{eqnarray*}
The first one cancel with the boundary integral in \eqref{9.3.1}. The second one cancel with another boundary integral which comes from the process of integrating by parts the term involving $\frac{\partial}{\partial z}\left[g_1^\prime\left(H^1-\Delta_zh_1\right)\right]$. (There are some terms involving $\Delta_0h_1$ left, which however are of higher order.)

\medskip

Combining all of these estimates together we obtain \eqref{reduction of stability} and the proof of Proposition \ref{prop stable for Toda} is complete.

Finally, the proof of Proposition \ref{prop curvature estimate 2} is exactly the same as in \cite{Wang-Wei 2}.

\appendix{}
\section{Nodegeneracy of the one dimensional solution}

In this appendix we present two forms of nondegeneracy property of the one dimensional solution $g$.
\begin{prop}\label{prop nondegeneracy 1}
There exists a constant $\mu>0$ so that the following holds. Suppose $\varphi\in H^1(0,+\infty)$ satisfies
\begin{equation}\label{1D orthogonal}
\int_0^{+\infty}\varphi(t)g^\prime(t)dt=0,
\end{equation}
then
\begin{equation}\label{spetral gap}
\int_0^{+\infty}\left[\varphi^\prime(t)^2+W^{\prime\prime}(g(t))\varphi(t)^2\right]dt\geq\mu\int_0^{+\infty}\varphi(t)^2dt.
\end{equation}
\end{prop}
\begin{proof}
The proof is via a contradiction argument. Suppose there exists a sequence of $\varphi_k\in H^1(0,+\infty)$ satisfying the orthogonal condition \eqref{1D orthogonal} and the normalization condition
\begin{equation}\label{A1}
\int_0^{+\infty}\varphi_k(t)^2dt=1,
\end{equation}
but it holds that
\begin{equation}\label{A2}
\int_0^{+\infty}\left[\varphi_k^\prime(t)^2+W^{\prime\prime}(g(t))\varphi_k(t)^2\right]dt\leq\frac{1}{k}.
\end{equation}

First by \eqref{A1} and \eqref{A2},
\[\int_0^{+\infty}\varphi_k(t)^2dt\leq \frac{1}{k}+\max_{u\in[0,1]}\big|W^{\prime\prime}(u)\big|.\]
Hence after passing to a subsequence, $\varphi_k$ converges weakly to a limit $\varphi_\infty$ in $H^1_{loc}(0,+\infty)$.

{\bf Claim. $\varphi_\infty\equiv 0$.}

By the above convergence of $\varphi_k$ we get
\[
\int_0^{+\infty}\varphi_\infty(t)^2dt\leq1\]
and
\begin{equation}\label{A3}
\int_0^{+\infty}\left[\varphi_\infty^\prime(t)^2+W^{\prime\prime}(g(t))\varphi_\infty(t)^2\right]dt\leq 0.
\end{equation}
By the exponential decay of $g^\prime$ at infinity, the orthogonal condition for $\varphi_k$ passes to the limit, that is, $\varphi_\infty$ satisfies \eqref{1D orthogonal}.

Since $g^\prime>0$, let $\psi:=\varphi_\infty/g^\prime$. Note that $\varphi_\infty$ and hence $\psi$ are continuous. Then \eqref{A3} is rewritten as
\[\int_0^{+\infty}\left[\psi^\prime(t)^2g^\prime(t)^2+2\psi(t)\psi^\prime(t)g^\prime(t)g^{\prime\prime}(t)+\psi(t)^2g^{\prime\prime}(t)^2+W^{\prime\prime}(g(t))g^\prime(t)^2\psi(t)^2\right]dt\leq 0.\]
After integration by parts this leads to
\[-\psi(0)^2g^\prime(0)g^{\prime\prime}(0)+\int_0^{+\infty}\psi^\prime(t)^2g^\prime(t)^2\leq 0.\]
Since $g^\prime(0)=\sqrt{2W(0)}>0$, $g^{\prime\prime}(0)=W^\prime(0)\leq 0$, this is only possible if $\psi\equiv 0$. The claim follows.

There exists $\Lambda>0$ such that $W^{\prime\prime}(g(t))\geq\kappa$ for $t\in[\Lambda,+\infty)$. By the above claim and the strong convergence of $\phi_k$ in $L^2_{loc}(0,+\infty)$,
\begin{equation}\label{A4}
\lim_{k\to+\infty}\int_0^\Lambda W^{\prime\prime}(g(t))\varphi_k(t)^2 dt=0.
\end{equation}
Substituting this into \eqref{A2} gives
\begin{equation}\label{A5}
\lim_{k\to+\infty}\int_\Lambda^{+\infty}W^{\prime\prime}(g(t))\varphi_k(t)^2dt=0.
\end{equation}
Combining \eqref{A4} and \eqref{A5} we get a contradiction with \eqref{A1}. This finishes the proof.
\end{proof}

The second form of nondegeneracy is
\begin{prop}\label{prop nondegeneracy 2}
There exists a constant $\mu>0$ so that the following holds. Suppose $\varphi\in H^1(0,+\infty)$ satisfies
$\varphi(0)=0$, then
\begin{equation}\label{spetral gap 2}
\int_0^{+\infty}\left[\varphi^\prime(t)^2+W^{\prime\prime}(g(t))\varphi(t)^2\right]dt\geq\mu\int_0^{+\infty}\varphi(t)^2dt.
\end{equation}
\end{prop}
The proof is similar to Proposition \ref{prop nondegeneracy 1} and it will not be repeated here.

\bigskip

\noindent {\bf Acknowledgments.} The author's research was partially
supported by NSFC No. 11631011  and by  ``the Fundamental Research
Funds for the Central Universities".

\end{document}